\documentclass[a4paper,12pt]{book}
\usepackage{amssymb}
\usepackage{amsthm}          
\usepackage[arrow,matrix,curve]{xy}
\usepackage{graphicx}
\usepackage{makeidx}         
\makeindex                   

\newtheorem{lemma}{Lemma}[section]
\newtheorem{theorem}[lemma]{Theorem}
\newtheorem{corollary}[lemma]{Corollary}
\newtheorem{remark}[lemma]{Remark}
\newtheorem{proposition}[lemma]{Proposition}
\newtheorem{definition}[lemma]{Definition}

\newcommand{\N}{\mathbb{N}}
\newcommand{\Z}{\mathbb{Z}}
\newcommand{\C}{\mathbb{C}}
\newcommand{\R}{\mathbb{R}}
\newcommand{\quat}{\mathbb{H}}
\newcommand{\K}{\mathbb{K}}
\newcommand{\sym}[1]{C^{#1}(\mathrm{Sym}^2(T^*M))}       
\newcommand{\spinor}[3]{C^{#3}(\Sigma^{#2}#1)}
\newcommand{\Lspinor}[3]{L^{#3}(\Sigma^{#2}#1)}
\newcommand{\Hspinor}[3]{H^{#3}(\Sigma^{#2}#1)}
\newcommand{\bmatrix}[2]{\left(\begin{array}{#1}#2\end{array}\right)}

\def\1{\mbox{\rm 1 \hspace{-1.05 em} 1}}

\def\dom{\mathrm{dom}}
\def\Hom{\mathrm{Hom}}
\def\End{\mathrm{End}}
\def\Aut{\mathrm{Aut}}
\def\Mat{\mathrm{Mat}}
\def\grad{\mathrm{grad}}

\def\scal{\mathrm{scal}}
\def\im{\mathrm{im}\,}

\def\Re{\mathrm{Re}}
\def\div{\mathrm{div}}
\def\tr{\mathrm{tr}}
\def\SO{\mathrm{SO}}
\def\SU{\mathrm{SU}}
\def\GL{\mathrm{GL}}
\def\Mat{\mathrm{Mat}}
\def\Spin{\mathrm{Spin}}
\def\P{\mathrm{P}}
\def\Cl{\mathrm{Cl}}
\def\CCl{\mathbb{C}\mathrm{l}}
\def\dv{\mathrm{dv}}

\def\id{\mathrm{Id}}

\def\span{\mathrm{span}}
\def\vol{\mathrm{vol}}
\def\spec{\mathrm{spec}}
\def\codim{\mathrm{codim}\,}

\def\supp{\mathrm{supp}}

\def\gdw{\Longleftrightarrow}

\def\ep{{\varepsilon}}

\def\th{{\vartheta}}

\def\ph{{\varphi}}
\def\phi{{\varphi}}

\def\geucl{g_{\mathrm{eucl}}}
\def\gcan{g_{\mathrm{can}}}

\begin{document}

\begin{titlepage}
\begin{center}
 {}
\end{center}
\vspace{10em}
\begin{center}
{\Huge Dirac eigenspinors for generic metrics}\\ 
\end{center}
\vspace{3em}
\begin{figure}[h]
\center
\end{figure}
\vspace{5em}
\begin{center}
\textsc{Dissertation zur Erlangung des Doktorgrades\\
der Naturwissenschaften (Dr.\,rer.\,nat.)\\
an der Fakult\"at f\"ur Mathematik der Universit\"at Regensburg}\\
\vspace{2em}
vorgelegt von \\
\textbf{Andreas Hermann} \\
Regensburg, im Januar 2012
\end{center}
\end{titlepage}

\newpage

\noindent
{}
\vspace{30em}

\noindent 
Promotionsgesuch eingereicht am 17.1.2012\\
\vspace{2em}

\noindent
Die Arbeit wurde angeleitet von Prof. Dr. Bernd Ammann\\
\vspace{2em}

\noindent
Pr\"ufungsausschuss:\\
\vspace{0em}

\begin{tabular}{ll}
Vorsitzender:                    & Prof. Dr. Harald Garcke\\
1. Gutachter:                    & Prof. Dr. Bernd Ammann\\
2. Gutachter:                    & Prof. Dr. Marc Herzlich, Universit\"at Montpellier\\
weiterer Pr\"ufer:               & Prof. Dr. Roman Sauer\\
Ersatzpr\"ufer:                  & Prof. Dr. Felix Finster
\end{tabular}

\newpage

\begin{center}
  \textbf{Abstract}
\end{center}

We consider a Riemannian spin manifold $(M,g)$ with a fixed spin structure. 
The zero sets of solutions of generalized Dirac equations on $M$ 
play an important role in some questions arising in 
conformal spin geometry and in mathematical physics. 
In this setting the mass endomorphism has been defined as
the constant term in an expansion of Green's function for the Dirac operator. 
One is interested in obtaining metrics, for which it is not zero. 

In this thesis we study the dependence of the zero sets of eigenspinors 
of the Dirac operator on the Riemannian metric. 
We prove that on closed spin manifolds of dimension $2$ or $3$ 
for a generic Riemannian metric the non-harmonic eigenspinors have no zeros. 
Furthermore we prove that on closed spin manifolds of dimension $3$ 
the mass endomorphism is not zero for a generic Riemannian metric. 

\begin{center}
  \textbf{Zusammenfassung}
\end{center}

Sei $(M,g)$ eine Riemannsche Spin-Mannigfaltigkeit mit einer fixierten Spin-Struktur. 
In manchen Fragen aus der konformen Spin-Geo\-me\-trie oder der mathematischen 
Physik spielen Null\-stel\-len\-men\-gen von L\"osungen verallgemeinerter Dirac-Gleichungen 
auf $M$ eine wichtige Rolle. 
In diesem Zusammenhang wurde der Mas\-sen-En\-do\-mor\-phis\-mus
als der konstante Term in einer asymptotischen 
Entwicklung der Greenschen Funktion des Dirac-Operators definiert. 
Gesucht sind Riemannsche Metriken, f\"ur die er nicht Null ist. 

In dieser Dissertation un\-ter\-suchen wir die Ab\-h\"an\-gig\-keit der Null\-stel\-len\-men\-ge der 
Eigenspinoren des Dirac-Operators von der Rie\-mann\-schen Metrik. 
Wir beweisen, dass auf einer geschlossenen Spin-Man\-nig\-fal\-tig\-keit der Dimension 
$2$ oder $3$ f\"ur eine generische Rie\-mann\-sche Metrik die nicht-har\-mo\-ni\-schen Eigenspinoren 
keine Null\-stel\-len haben. Weiter zeigen wir, dass auf einer geschlossenen Spin-Mannigfaltigkeit 
der Dimension $3$ f\"ur eine generische Rie\-mann\-sche Metrik der Mas\-sen-En\-do\-mor\-phis\-mus nicht Null ist.

\newpage
\textbf{Acknowledgements}

I am grateful to many people, who have supported me during the preparation of this thesis. 
First of all I would like to thank my advisor Prof.\,\,Bernd Ammann for accepting me 
as his student and for proposing a very interesting research topic. 
I have profited very much from his explanations and numerous mathematical discussions 
with him. He has encouraged me a lot during the time that progress on my work was slow. 
Many thanks go to my colleagues for creating a nice working atmosphere and especially 
to Dr.\,\,Nadine Gro\ss e and Dr.\,\,Nicolas Ginoux for many interesting discussions. 

I am grateful to Dr.\,\,Mattias Dahl who invited me to Stockholm and from whom I could learn 
very much about techniques for demonstrating generic properties. 
Also I would like to thank Prof.\,\,Emmanuel Humbert for an invitation to Nancy. 

Thanks are due to the DFG for supporting part of my work by the graduate program 
``Curvature, Cycles and Cohomology''. 
I have also profited very much from conferences and seminars within this program. 

Finally I would like to thank my family for their emotional support during all the years of my studies.

\tableofcontents

\newpage

\chapter{Overview}

Some questions arising in conformal spin geometry and in mathematical phy\-sics 
involve the study of the zero sets of solutions to generalized Dirac equations.
For the first example let $(M,g)$ be a compact Riemannian spin manifold 
with a fixed orientation and a fixed spin structure. 
One is interested in finding bounds on the eigenvalues 
of the Dirac operator $D^g$ which are uniform in the conformal class $[g]$ of $g$. 
The two conformal invariants 
\begin{displaymath}
  \lambda_{\min}^{\pm}:=\inf_{h\in[g]}|\lambda_1^{\pm}(h)|\,\vol(M,h)^{1/n}
\end{displaymath}
have been studied by many authors. 
A natural question is whether the infimum is attained 
at a Riemannian metric. By a result of B.\,Ammann 
this is the case, if the nonlinear partial differential equation 
\begin{equation}
  \label{nonlinear_dirac_eq}
  D^g\psi=\lambda_{\min}^+|\psi|_g^{2/(n-1)}\psi,\qquad 
  \|\psi\|_{2n/(n-1)}=1
\end{equation}
has a solution $\psi$, which is nowhere zero on $M$ (see \cite{am09}). 
It is not obvious that a solution without zeros exists. 

A second example comes from general relativity, more precisely 
from a remarkable proof of the positive energy theorem 
obtained by E.\,Witten (see \cite{wi}). 
He uses harmonic spinors (i.\,e.\,spinors $\psi$ satisfying $D^g\psi=0$) 
on asymptotically flat manifolds, which are called Witten spinors. 
It has been suggested to use these spinors in order to construct 
special orthonormal frames of the tangent bundle of an asymptotically 
flat manifold of dimension $3$ (see \cite{n}, \cite{dm}, \cite{fns}). 
It turns out that this is possible, if one can find a Witten spinor 
which is nowhere zero. 
However it is not clear that such a spinor exists.

In this thesis we first consider the zero sets of eigenspinors 
of the Dirac operator on closed spin manifolds. 
It is interesting for several reasons. 
First of all it is easier than the questions 
mentioned above, since the underlying manifold is compact and 
all the operators involved are linear. 
Apart from that it is useful to have eigenspinors, which are nowhere zero. 
For example one obtains in this case a simple proof of Hijazi's inequality.

The spectrum of the Dirac operator has been computed explicitly 
for some Riemannian spin manifolds. 
The result on the round sphere (see e.\,g.\,\cite{baer96}) 
shows that the multiplicity of an eigenvalue can 
be greater than the rank of the spinor bundle, which implies 
that there exist eigenspinors with non-empty zero set for special 
choices of a Riemannian metric. 
However we can show that the situation is different for generic Riemannian metrics 
(see Section \ref{gen_eigenspinor_section} for a precise definition of the term ``generic''). 
More precisely let $M$ be a closed spin manifold and denote by $R(M)$ the set of all 
smooth Riemannian metrics on $M$. 
For every $g\in R(M)$ denote by $[g]\subset R(M)$ the conformal class of $g$. 
Furthermore let $N(M)$ be the set of all $g\in R(M)$ such that all the non-harmonic 
eigenspinors of $D^g$ are nowhere zero on $M$. 
Then we prove the following. 

\begin{theorem}
\label{main_theorem}
Let $M$ be a closed connected spin manifold of dimension $2$ or $3$ 
with a fixed orientation and a fixed spin structure. 
Then the set $N(M)\cap[g]$ is residual in $[g]$ 
with respect to every $C^k$-topology, $k\geq1$.
\end{theorem}

Recall that a subset is residual, if it contains a countable 
intersection of open and dense sets. 
In Section \ref{surfaces_section} we will give an example showing that 
in dimension $2$ an analogue of this theorem 
for harmonic spinors does not hold. 

The main idea of the proof of Theorem \ref{main_theorem} is as follows. 
If $g,h$ are two Riemannian metrics on the spin manifold $M$, 
then a natural isomorphism between the two vector bundles 
$\Sigma^gM$ and $\Sigma^hM$ is well known. 
We construct a continuous map $F$ defined on a suitable space of Riemannian 
metrics, which associates to every metric $h$ 
an eigenspinor of the corresponding Dirac operator $D^h$ viewed as a section of $\Sigma^gM$. 
Theorem \ref{main_theorem} then follows from a transversality theorem. 
In order to apply this theorem we have to make sure that the evaluation map 
corresponding to $F$ is transverse to the zero section of $\Sigma^gM$. 
Our assumption that this is not the case 
leads to an equation involving Green's function 
for the operator $D^g-\lambda$ with $\lambda\in\R$. 
From the expansion of this Green's function 
we obtain a contradiction using the unique continuation property 
of the Dirac operator. 



In this thesis we also treat a certain aspect of the question on 
conformal bounds of Dirac eigenvalues mentioned above. 
The non-linear partial differential equation (\ref{nonlinear_dirac_eq}) 
cannot be solved by standard methods, since the corresponding Sobolev embedding 
is critical. 
However under some additional assumptions on the Riemannian spin manifold $(M,g)$
the existence of a solution has been shown, if there is a point $p$ on $M$ 
such that a certain endomorphism of the fibre $\Sigma^g_pM$ of the spinor bundle 
does not vanish (see \cite{ahm}). 
This endomorphism can be regarded as the constant term in an expansion of 
Green's function for the Dirac operator around $p$. 
It is called the mass endomorphism at $p$ because of an analogy in 
conformal geometry: the constant term of Green's function $\Gamma(.,p)$ 
for the conformal Laplacian at $p$ is related to the mass of 
the asymptotically flat Riemannian manifold $(M\setminus\{p\},\Gamma(.,p)^{4/(n-2)}g)$. 
Unfortunately the mass endomorphism is known explicitly 
only for very few spin manifolds. 
However in dimension $3$ we can show that in the generic case it is 
not zero. 
More precisely, for a fixed point $p\in M$ we denote by $R_p(M)$ 
the set of all Riemannian metrics on $M$, such that the mass endomorphism 
at $p$ can be defined and we denote by $S_p(M)$ the set of all such 
Riemannian metrics, for which the mass endomorphism does not vanish. 
Then we prove the following. 

\begin{theorem}
Let $M$ be a closed spin manifold of dimension $3$ with a fixed spin 
structure and let $p\in M$. Then $S_p(M)$ is dense in $R_p(M)$.
\end{theorem}




The structure of this thesis is as follows. 
In Chapter \ref{recap_chapter} we review the basic definitions 
and results in spin geometry. We begin with the definition 
of spin groups and the Dirac operator and then proceed 
with an overwiew on real and quaternionic structures on 
spinor modules. 
Since for two different Riemannian 
metrics the spinor bundles are two different vector bundles, 
we need a natural way of identifying 
spinors for different metrics. 
This is known in the literature and is described 
in Section \ref{identification_section}. 
Results from perturbation theory, which apply in our 
situation, are collected in the Appendix.

In Chapter \ref{motivation_chapter} we explain the questions 
mentioned above in more detail. 

Chapter \ref{examples_chapter} 
contains some examples of zero sets of Dirac eigenspinors. 
This should serve as an illustration. 
The reader who is only interested in the main results 
may skip Sections \ref{sphere_section}, \ref{tori_section}. 

In Chapter \ref{green_chapter} we introduce 
Green's function for the Dirac operator, which 
will be one of our tools in the proofs of our results. 
We describe a method for the explicit calculation 
of some terms in the expansion of Green's function 
around the singularity. 
After that we give the definition of the mass endomorphism 
from the literature. 

In Chapter \ref{gen_eigenspinor_chapter} we state and 
prove our main results. First let $M$ be a closed 
spin manifold of dimension $2$ or $3$. We prove 
that for a generic Riemannian metric on $M$ 
the non-harmonic eigenspinors of the Dirac operator 
are nowhere zero. 
The proof is based on a well known 
transversality theorem, which we state in Section 
\ref{transversality_section} including its proof. 
After that we prove that on every 
closed spin manifold of dimension $3$ 
for a generic Riemannian metric
the mass endomorphism is not zero.

\chapter{Recapitulation of some facts}
\label{recap_chapter}

\section{Review of spin geometry}

We review some basic definitions and results in spin geometry in order to fix 
the notation. Brief and nicely written introductions to spin geometry 
can be found in \cite{hij99}, \cite{bgm05}. The reader can find 
a detailed treatment of this subject in the books \cite{lm}, \cite{fr00}. 
Our notation will be similar to some of the notation in these texts. 

Let $V$ be a real vector space with a scalar product $g$. Then the 
real Clifford algebra $\Cl(V,g)$ for $(V,g)$ is the unital $\R$-algebra 
generated by $V$ with the relation
\begin{equation}
  \label{clifford_relation}
  v\cdot w+w\cdot v=-2g(v,w)1,\quad v,w\in V.
\end{equation}
We denote by $\CCl(V,g):=\Cl(V,g)\otimes_{\R}\C$ its complexification 
and call it the Clifford algebra for $(V,g)$. 
In the case $V=\R^n$, $n\in\N\setminus\{0\}$, $g=\geucl$ we use the notation
\begin{displaymath}
  \Cl(n):=\Cl(\R^n,\geucl),\quad\CCl(n):=\CCl(\R^n,\geucl).
\end{displaymath}
Every $v\in\R^n\setminus\{0\}$ is invertible in $\Cl(n)$ and the map
\begin{displaymath}
  Ad_v:\quad\Cl(n)\to\Cl(n),\quad w\mapsto v\cdot w\cdot v^{-1}
\end{displaymath}
preserves the subspace $\R^n\subset\Cl(n)$. The restriction acts on $\R^n$ as 
the reflection at the line generated by $v$ and thus is in $O(n)$. 

Let $\Cl(n)^*$ be the multiplicative group of invertible elements of $\Cl(n)$. 
We define the spin group $\Spin(n)$ as the subgroup of $\Cl(n)^*$ 
generated by elements of the form $v_1\cdot v_2$, where $v_1$, $v_2\in\R^n$, $|v_1|=|v_2|=1$. 
We obtain a group homomorphism
\begin{displaymath}
  \vartheta:\quad\Spin(n)\to\SO(n),\quad q\mapsto Ad_q
\end{displaymath}
which is a two-fold covering. The covering is nontrivial for $n\geq2$ 
and it is the universal covering for $n\geq3$. 

If $(E_i)_{i=1}^n$ denotes the standard basis of $\R^n$, then 
\begin{displaymath}
  \omega_{\C}:=i^{[(n+1)/2]}E_1\cdot...\cdot E_n\in\CCl(n)
\end{displaymath}
is called the complex volume element. Here $[.]$ denotes the integer part. 
We have $\omega_{\C}^2=1$ and
\begin{displaymath}
  \omega_{\C}\cdot v=(-1)^{n-1}v\cdot\omega_{\C}
\end{displaymath}
for all $v\in\R^n$.

For even $n$ there is exactly one irreducible complex representation of $\CCl(n)$. 
The module has complex dimension $2^{n/2}$ and is denoted by $\Sigma_n$. 
It is the direct sum of the eigenspaces $\Sigma^{\pm}_n$ of $\omega_{\C}$ 
for the values $\pm1$: 
\begin{displaymath}
  \Sigma_n=\Sigma^+_n\oplus\Sigma^-_n.
\end{displaymath}
Since $\omega_{\C}$ anticommutes with elements of $\R^n$, the two eigenspaces 
have the same dimension.

For odd $n$ the complex volume element commutes with all elements 
of $\CCl(n)$ and thus by Schur's lemma acts as a multiple of the identity 
on every irreducible module. There exist exactly two inequivalent irreducible 
complex representations of $\CCl(n)$, both of dimension $2^{(n-1)/2}$, and they 
are distinguished by the action of $\omega_{\C}$ as $\id$ or $-\id$ respectively. 
In this thesis we will use the representation for which $\omega_{\C}$ acts as $\id$. 
The module is again denoted by $\Sigma_n$.

Thus for every $n$ we have $\Sigma_n\cong\C^N$, where $N:=2^{[n/2]}$. 
The representation will be denoted by $\rho$. 
The action of $\CCl(n)$ on $\Sigma_n$ via $\rho$ is called the Clifford 
multiplication on $\Sigma_n$. It will be denoted by 
$v\cdot\sigma:=\rho(v)\sigma$ for $v\in\CCl(n)$, $\sigma\in\Sigma_n$.

On $\Sigma_n$ there exists a positive definite hermitian inner product 
$\langle.,.\rangle$, such that Clifford multiplication with all 
elements of $\R^n$ is antisymmetric with respect to $\langle.,.\rangle$, i.\,e. 
such that
\begin{displaymath}
  \langle v\cdot\psi,\varphi\rangle+\langle\psi,v\cdot\varphi\rangle=0
\end{displaymath}
for all $\varphi$, $\psi\in\Sigma_n$ and all $v\in\R^n$.

Let $M$ be an $n$-dimensional oriented manifold. We denote by $\GL^+(n,\R)$ 
the group of all real $n\times n$-matrices with positive determinant and we write 
$A$: $\widetilde{\GL}^+(n,\R)\to\GL^+(n,\R)$ for its connected two-fold covering. 
Let
\begin{displaymath}
  \pi:\quad \P_{\GL^+}(M)\to M
\end{displaymath}
be the principal $\GL^+(n,\R)$-bundle over $M$ whose fibre over $x\in M$ 
consists of all positively oriented bases of $T_xM$. 
It is called the bundle of positively oriented frames of the tangent bundle. 

\begin{definition}
A spin structure on $M$ is a principal $\widetilde{\GL}^+(n,\R)$-bundle
\begin{displaymath}
  \pi':\quad \P_{\widetilde{\GL}^+}(M)\to M
\end{displaymath}
over $M$ together with a two-fold covering 
$\Theta$: $\P_{\widetilde{\GL}^+}(M)\to\P_{\GL^+}(M)$, such that the following diagram commutes
\[
\begin{xy}
  \xymatrix{          
    \P_{\widetilde{\GL}^+}(M)\times\widetilde{\GL}^+(n,\R) \ar[r] \ar[dd]_{\Theta\times A} & \P_{\widetilde{\GL}^+}(M) \ar[dd]^\Theta \ar[rd]^{\pi'} &\\
                                                                                           &                                                         & M\\
    \P_{\GL^+}(M)\times\GL(n,\R)     \ar[r]                                                & \P_{\GL^+}(M) \ar[ru]_{\pi}                             &
  }
\end{xy}
\]
where the horizontal arrows denote the group actions. 
$M$ is called a spin manifold, if there exists a spin structure on $M$.
\end{definition}

Not every oriented manifold has a spin structure and some oriented ma\-ni\-folds 
have more than one spin structure. 
An oriented manifold $M$ is a spin manifold if and only if the second 
Stiefel-Whitney class 
\begin{displaymath}
  w_2(TM)\in H^2(M,\Z/2\Z)
\end{displaymath}
vanishes.
If $w_2(TM)=0$, then the distinct spin structures on $M$ are in one-to-one correspondence with 
the elements of $H^1(M,\Z/2\Z)$. 
For example every orientable manifold $M$ of dimension $n\leq3$ is a 
spin manifold. 

For every Riemannian metric $g$ on an oriented spin manifold $M$ we denote by $\P_{\SO}(M,g)\subset\P_{\GL^+}(M)$ 
the principal $\SO(n)$-bundle over $M$ whose fibre over $x\in M$ 
consists of all positively oriented $g$-orthonormal bases of $T_xM$. 
It is called the bundle of positively oriented $g$-orthonormal frames of the tangent bundle. 
The restriction of $A$: $\widetilde{\GL}^+(n,\R)\to\GL^+(n,\R)$ to 
the preimage of $\SO(n)\subset\GL^+(n,\R)$ coincides with $\vartheta$: $\Spin(n)\to\SO(n)$. 
Furthermore 
\begin{displaymath}
  \P_{\Spin}(M,g):=\Theta^{-1}(P_{\SO}(M,g))
\end{displaymath}
is a principal $\Spin(n)$-bundle over $M$ and the maps in the above commutative diagram 
restrict to the following commutative diagram
\[
\begin{xy}
  \xymatrix{          
    \P_{\Spin}(M,g)\times\Spin(n) \ar[r]  \ar[dd]_{\Theta\times\th} & \P_{\Spin}(M,g) \ar[dd]^\Theta \ar[rd]^{\pi'}  &\\
                                                                    &                                                & M\\
    \P_{\SO}(M,g)\times\SO(n)     \ar[r]                            & \P_{\SO}(M,g) \ar[ru]_{\pi}                    &
  }
\end{xy}
\]

By a Riemannian spin manifold $(M,g,\Theta)$ we will always mean 
an orientable Riemannian spin manifold $(M,g)$ together with a fixed orientation and a 
fixed spin structure $\Theta$.

To every principal bundle over $M$ we can associate a vector bundle 
in the following way. Let $G$ be a Lie group and let $\pi$: $P\to M$ be 
a principal $G$-bundle over a manifold $M$. 
Furthermore let $\K=\R$ or $\C$, let $V$ be a vector space 
over $\K$ and let $\rho$: $G\to\Aut_{\K}(V)$ be a representation of $G$. 
On $P\times V$ we define an equivalence relation by
\begin{displaymath}
  (p,v)\sim(p',v')\gdw\exists g\in G\textrm{ such that }p'=pg^{-1}\textrm{ and }v'=\rho(g)v.
\end{displaymath}
We denote by $[p,v]$ the equivalence class of $(p,v)$ and by 
$P\times_{\rho}V$ the set of all equivalence classes. 
Then $P\times_{\rho}V$ is a vector bundle over $M$ with fibre $V$. 
It is called the associated vector bundle to $P$ via $\rho$.

As an example let $\tau$: $\SO(n)\to\Aut_{\R}(\R^n)$ be the 
standard representation. Then there exists a canonical isomorphism 
of vector bundles 
\begin{displaymath}
  TM\cong \P_{\SO}(M,g)\times_{\tau}\R^n.
\end{displaymath}

Now let $(M,g,\Theta)$ be an $n$-dimensional Riemannian spin manifold. 
The restriction of the complex spinor representation $\rho$ 
to $\Spin(n)\subset\CCl(n)$ is again denoted by $\rho$, i.\,e.
\begin{displaymath}
  \rho:\quad\Spin(n)\to\Aut_{\C}(\Sigma_n).
\end{displaymath}

\begin{definition}
The complex spinor bundle $\Sigma^gM$ over $M$ for the metric $g$ and 
the spin structure $\Theta$ is the associated vector bundle
\begin{displaymath}
  \Sigma^gM:=\P_{\Spin}(M,g)\times_{\rho}\Sigma_n.
\end{displaymath}
\end{definition}
The complex spinor bundle $\Sigma^gM$ is a vector bundle with fibre $\Sigma_n\cong\C^N$. 
For $(M,g)=(\R^n,\geucl)$ equipped with the unique spin structure we will write 
\begin{displaymath}
  \Sigma\R^n:=\Sigma^{\geucl}\R^n.
\end{displaymath}

We define the Clifford multiplication on $\Sigma^g_xM$, $x\in M$, by
\begin{displaymath}
  T_xM\otimes\Sigma^g_xM\to\Sigma^g_xM,\quad
  [\Theta(s),v]\otimes[s,\sigma]\mapsto[s,v\cdot\sigma].
\end{displaymath}
The inner product $\langle.,.\rangle$ on $\Sigma_n$ yields a hermitian metric 
on $\Sigma^gM$ which we also denote by $\langle.,.\rangle$. It is defined by 
\begin{displaymath}
  \langle[s,\sigma_1],[s,\sigma_2]\rangle:=\langle\sigma_1,\sigma_2\rangle.
\end{displaymath}
We will denote the induced norm by 
\begin{displaymath}
  |\psi|_g:=\langle\psi,\psi\rangle^{1/2}.
\end{displaymath}
Clifford multiplication with elements of $TM$ is antisymmetric 
with respect to $\langle.,.\rangle$, i.\,e. we have
\begin{displaymath}
  \langle v\cdot\psi,\varphi\rangle+\langle\psi,v\cdot\varphi\rangle=0
\end{displaymath}
for all $v\in T_xM$, $\psi$, $\varphi\in\Sigma^g_xM$, $x\in M$.

Let $(e_i)_{i=1}^n$ be a positively oriented local orthonormal frame of $TM$. Then 
\begin{displaymath}
  \omega_{\C}:=i^{[(n+1)/2]}e_1\cdot...\cdot e_n
\end{displaymath}
is independent of the choice of the $e_i$ and thus can be defined on all of $M$. 
Fibrewise Clifford multiplication by $\omega_{\C}$ is an endomorphism of $\Sigma^gM$ 
which is again denoted by $\omega_{\C}$. We find $\omega_{\C}^2=1$ and
\begin{displaymath}
  \omega_{\C}\cdot v=(-1)^{n-1}v\cdot\omega_{\C}
\end{displaymath}
for all $v\in TM$. If $n$ is odd, we have by convention $\omega_{\C}=\id$. 
If $n$ is even, then, since $\omega_{\C}$ commutes with $\Spin(n)$, 
the above splitting of the spinor module induces a splitting of the spinor bundle
\begin{displaymath}
  \Sigma^gM=\Sigma^+M\oplus\Sigma^-M.
\end{displaymath}

A local section of the spinor bundle $\Sigma^gM$ is called a spinor. 
If $n$ is even the local sections of $\Sigma^{\pm}M$ are called positive 
respectively negative spinors. 
We denote by $\spinor{M}{g}{r}$, $r\in\N$, (resp. $\spinor{M}{g}{\infty}$) 
the space of all $r$ times continuously differentiable (resp. smooth) spinors. 

In order to define a covariant derivative on the spinor bundle we recall the following 
general fact. Let $G$ be a Lie group, $P\to M$ a principal $G$-bundle and 
let $\rho$: $G\to\Aut_{\C}(V)$ be a representation of $G$ on a complex vector space $V$. 
Let $\omega$ be a connection one-form on $P$. 
Every section $\psi$ of the associated vector bundle $P\times_{\rho}V$ is locally given 
by $\psi=[s,\sigma]$, where $s$ is a locally defined section of $P$ on an open subset 
$U\subset M$ and $\sigma$ is a function on $U$ with values in $V$. 
Then for $X\in TM|_U$ we define the spinor $\nabla_X\psi$ on $U$ by
\begin{equation}
  \label{cov_deriv}
  \nabla_X\psi:=[s,X(\sigma)+d\rho(\omega(ds(X)))\sigma].
\end{equation}
Here and henceforth for any differentiable function $f$ defined on an open subset $U\subset M$ 
with values in a real or complex vector space and any vector field $X$ on $U$ we denote 
by $X(f)$ the derivative of $f$ in the direction $X$. 
One can show that $\nabla$ is well-defined and yields a covariant derivative 
on $P\times_{\rho}V$. Furthermore if $V$ carries a $G$-invariant hermitian scalar product, 
then one obtains a hermitian metric on $P\times_{\rho}V$ and $\nabla$ 
is compatible with this metric.

The Levi Civita connection $\nabla^g$ for the Riemannian metric $g$ on $M$ 
induces a covariant derivative on $\Sigma^gM$ also denoted by $\nabla^g$ as follows. 
The connection one-form of the Levi Civita connection lifts to a 
connection one-form $\omega$ on $\P_{\Spin}(M,g)$. 
Then we apply the formula (\ref{cov_deriv}) to $\omega$. The result may locally 
be written as follows. 
Let $(e_i)_{i=1}^n$ be a positively oriented local orthonormal frame of $TM$ on an open subset $U\subset M$. 
There exists a locally defined section $s\in\Gamma(U,\P_{\Spin}(M,g)|_U)$ such that 
$(e_i)_{i=1}^n=\Theta\circ s$ on $U$. Let $(E_i)_{i=1}^N$ be the standard basis of $\C^N$. 
The section $s$ determines a local orthonormal frame $(\psi_i)_{i=1}^N$ 
of $\Sigma^gM|_U$ via $\psi_i=[s,E_i]$ for $1\leq i\leq N$. 
We denote by $\partial$ the locally defined flat connection with respect to the 
local frame $(\psi_i)_{i=1}^N$, i.\,e.\,for $h_1$,...,$h_N\in C^{\infty}(U,\C)$ and $X\in TM|_U$ we define
\begin{displaymath}
  \partial_X\big(\sum_{i=1}^N h_i\psi_i\big):=\sum_{i=1}^N X(h_i) \psi_i.
\end{displaymath}
We define the symbols $\widetilde{\Gamma}^k_{ij}$ by the equation
\begin{displaymath}
  \widetilde{\Gamma}^k_{ij}:=g(\nabla^g_{e_i}e_j,e_k)
\end{displaymath}
and write locally $\psi=\sum_{i=1}^N h_i\psi_i$. 
Then by (\ref{cov_deriv}) for all $i\in\{1,...,n\}$ we have on $U$ 
\begin{equation}
  \label{connection_local}
  \nabla^g_{e_i}\psi=\partial_{e_i}\psi+\frac{1}{4}\sum_{j,k=1}^n\widetilde{\Gamma}^k_{ij} e_j\cdot e_k\cdot\psi
  =\partial_{e_i}\psi+\frac{1}{4}\sum_{j=1}^n e_j\cdot(\nabla^g_{e_i}e_j)\cdot\psi
\end{equation}
(see \cite{lm}, p. 103, 110). 
One finds that $\nabla^g$ is a metric connection with respect to $\langle.,.\rangle$ 
and that it satisfies
\begin{displaymath}
  \nabla^g_X(Y\cdot\psi)=(\nabla^g_X Y)\cdot\psi+Y\cdot\nabla^g_X\psi
\end{displaymath}
for all $X\in TM|_U$, $Y\in C^{\infty}(TM|_U)$, $\psi\in\spinor{M|_U}{g}{\infty}$. 

Let $(e_i)_{i=1}^n$ be a local orthonormal frame of $TM$. 
The Dirac operator is defined as
\begin{displaymath}
  D^g:\quad\spinor{M}{g}{\infty}\to\spinor{M}{g}{\infty},\quad D^g\psi:=\sum_{i=1}^n e_i\cdot\nabla^g_{e_i}\psi.
\end{displaymath}
It is easily seen that the definition does not depend on the choice of the 
local frame $(e_i)_{i=1}^n$. If $n$ is even, then with respect to the above 
splitting of the spinor bundle the Dirac operator has the form 
\begin{equation}
  \label{dirac_op_decomp}
  D^g=\bmatrix{cc}{0 & D^- \\ D^+ & 0}
\end{equation}
with $D^{\pm}$: $\spinor{M}{\pm}{\infty}\to\spinor{M}{\mp}{\infty}$. 

If $X$, $Y$ are smooth vector fields on $M$ we define the second covariant 
derivative operator $\nabla^2_{X,Y}$: $\spinor{M}{g}{\infty}\to\spinor{M}{g}{\infty}$ by
\begin{displaymath}
  \nabla^2_{X,Y}\psi:=\nabla^g_X\nabla^g_Y\psi-\nabla^g_{\nabla^g_X Y}\psi
\end{displaymath}
and the connection Laplacian $\nabla^*\nabla$: $\spinor{M}{g}{\infty}\to\spinor{M}{g}{\infty}$ by
\begin{displaymath}
  \nabla^*\nabla\psi:=-\tr(\nabla^2_{.,.}\psi),
\end{displaymath}
i.\,e.\,if $(e_i)_{i=1}^n$ is a local orthonormal frame on $M$ we have
\begin{equation}
  \label{laplacian_local}
  \nabla^*\nabla\psi=-\sum_{i=1}^n \nabla^g_{e_i}\nabla^g_{e_i}\psi
  +\sum_{i=1}^n \nabla^g_{\nabla^g_{e_i}e_i}\psi.
\end{equation}
A very important result is the Schr\"odinger-Lichnerowicz formula 
\begin{theorem}
For all $\psi\in\spinor{M}{g}{\infty}$ the formula
\begin{equation}
  \label{schroed_lichn}
  (D^g)^2\psi=\nabla^*\nabla\psi+\frac{\scal^g}{4}\psi.
\end{equation}
holds, where $\scal^g$ is the scalar curvature of $(M,g)$.
\end{theorem}

\begin{proof}
see \cite{lm} p. 160.
\end{proof}

Let $\psi\in\spinor{M}{g}{\infty}$ and $f\in C^{\infty}(M,\C)$. Then 
the following Leibniz rule holds on $M$:
\begin{displaymath}
  D^g(f\psi)=\grad^g(f)\cdot\psi+fD^g\psi,
\end{displaymath}
where $\grad^g(f)$ is the gradient of $f$ with respect to $g$. 
For a proof see \cite{lm}, p. 116.

Let $\mu\in\C$. A spinor $\psi\in\spinor{M}{g}{\infty}$ 
is called a Killing spinor for $\mu$\index{Killing spinor}, if we have
\begin{displaymath}
  \nabla^g_X\psi=\mu X\cdot\psi
\end{displaymath}
for all $X\in TM$.

A spinor $\psi\in\spinor{M}{g}{\infty}$ 
is called parallel\index{parallel spinor}, if we have $\nabla^g_X\psi=0$ for all $X\in TM$.

\begin{definition}
The elements of $\ker D^g$ are called harmonic spinors.\index{harmonic spinor} 
If $n$ is even, then the elements of $\ker D^{\pm}$ are called positive respectively 
negative harmonic spinors.
\end{definition}

Let $M$ be compact. The hermitian metric $\langle.,.\rangle$ on $\Sigma^gM$ induces 
a positive definite scalar product $(.,.)_2$ on smooth spinors by  
\begin{displaymath}
  (\psi,\varphi)_2:=\int_M\langle\psi,\varphi\rangle\dv^g,
\end{displaymath}
where $\dv^g$ denotes the volume form induced by the Riemannian metric $g$. 
Let $p\in\R$ be positive. The $L^p$-norm of a smooth spinor $\psi$ is by definition 
\begin{displaymath}
  \|\psi\|_p:=\big(\int_M |\psi|_g^p \dv^g\big)^{1/p}.
\end{displaymath}
The completion of $\spinor{M}{g}{\infty}$ 
with respect to $\|.\|_p$ is called $\Lspinor{M}{g}{p}$. 
For $k\in\N$ the Sobolev $k$-norm of a smooth spinor $\psi$ is defined by 
\begin{displaymath}
  \|\psi\|_{H^k}:=\sum_{i=0}^k\|\nabla^i\psi\|_2
\end{displaymath}
and the Sobolev space $\Hspinor{M}{g}{k}$ is the completion of $\spinor{M}{g}{\infty}$ 
with respect to $\|.\|_{H^k}$. 
If $(M,g)$ is complete, then with respect to the scalar product $(.,.)_2$ 
the Dirac operator $D^g$ is essentially self-adjoint, i.\,e.\,its closure 
in $\Lspinor{M}{g}{2}$ is self-adjoint. 
If $(M,g,\Theta)$ is a closed Riemannian spin manifold, then the spectrum $\spec(D^g)$ 
of the Dirac operator consists of a sequence of isolated real eigenvalues, 
which is neither bounded from above nor bounded from below.

\section{Real and quaternionic structures}
\label{real_quat_struct_section}

In certain dimensions $n$ there exist real or quaternionic structures 
on the modules $\Sigma_n$ of the complex spinor representation 
which are $\Spin(n)$-equivariant. 
Given a Riemannian spin manifold $(M,g,\Theta)$ of dimension $n$ these 
structures then induce conjugate linear endomorphisms of $\Sigma^gM$. 

\begin{definition}
Let $W$ be a complex vector space.
\begin{enumerate}
  \item A real structure\index{structure!real} on $W$ is a $\R$-linear map $J$: $W\to W$ 
  such that $J^2=\id$ and $J(iw)=-iJ(w)$ for all $w\in W$. 
  \item A quaternionic structure\index{structure!quaternionic} on $W$ is a $\R$-linear map $J$: $W\to W$ 
  such that $J^2=-\id$ and $J(iw)=-iJ(w)$ for all $w\in W$.
\end{enumerate}
\end{definition}

Let $\rho$: $\CCl(n)\to\End_{\C}(\Sigma_n)$ be the complex spinor 
representation. A real or quaternionic structure $J$ on $\Sigma_n$ 
is called commuting, if it commutes with Clifford multiplication 
by elements of $\R^n$, i.\,e.\,if
\begin{displaymath}
  J(x\cdot\sigma)=x\cdot J(\sigma)
  \textrm{ for all }x\in\R^n\subset\CCl(n),\sigma\in\Sigma_n.
\end{displaymath}
It is called anti-commuting, if it anti-commutes with 
Clifford multiplication by elements of $\R^n$, i.\,e.\,if
\begin{displaymath}
  J(x\cdot\sigma)=-x\cdot J(\sigma)
  \textrm{ for all }x\in\R^n\subset\CCl(n),\sigma\in\Sigma_n.
\end{displaymath}

The existence of real or quaternionic structures on $\Sigma_n$ 
for certain $n$ is proved in \cite{fr00}. 
In the following theorem we state the result and mention some 
further structures. 

\begin{theorem}
\label{theorem_real_quat}
On $\Sigma_n$ the following structures exist
\renewcommand{\labelenumi}{\alph{enumi})}
\begin{enumerate}
\item If $n\equiv 0\bmod 8$ there exist a commuting real structure 
and an anti-commuting real structure.
\item If $n\equiv 1\bmod 8$ there exists an anti-commuting real structure. 
There exists no commuting real structure.
\item If $n\equiv 2\bmod 8$ there exist an anti-commuting real structure 
and a commuting quaternionic structure. 
There exists no commuting real structure.
\item If $n\equiv 3\bmod 8$ there exists a commuting quaternionic structure. 
There exists no commuting real structure.
\item If $n\equiv 4\bmod 8$ there exist a commuting quaternionic structure 
and an anti-commuting quaternionic structure. 
There exists no commuting real structure.
\item If $n\equiv 5\bmod 8$ there exists an anti-commuting quaternionic structure. 
There exists no commuting real structure.
\item If $n\equiv 6\bmod 8$ there exist a commuting real structure 
and an anti-commuting quaternionic structure. 
\item If $n\equiv 7\bmod 8$ there exists a commuting real structure. 
\end{enumerate}
\end{theorem}

\begin{proof}
\renewcommand{\labelenumi}{\alph{enumi})}
\begin{enumerate}
\item $n\equiv 0\bmod 8$: By \cite{fr00}, p. 33 there exists an anti-commuting 
real structure $J$. Then $\omega_{\C}\circ J$ is a commuting real structure.
\item $n\equiv 1\bmod 8$: By \cite{fr00} there exists an anti-commuting 
real structure. Suppose that there was a commuting real structure $J$. 
We restrict the spinor representation to $\Cl(n)\subset\CCl(n)$. 
If $n=8k+1$ then the eigenspace of $J$ corresponding to $1$ is a real 
subspace of $\Sigma_n$ of real dimension $2^{4k}$ and is an invariant 
subspace of this restriction. This is a contradiction, since 
every irreducible real module for $\Cl(n)$ has real dimension 
$2^{4k+1}$ (see \cite{lm}, p. 33).
\item $n\equiv 2\bmod 8$: By \cite{fr00} there exists a commuting 
quaternionic structure $J$. Then $\omega_{\C}\circ J$ is an anti-com\-mut\-ing 
real structure. 
As in the case $n\equiv1$ mod $8$ one sees that there is no 
commuting real structure.
\item $n\equiv 3\bmod 8$: By \cite{fr00} there exists a commuting 
quaternionic structure. 
As in the case $n\equiv 1\bmod 8$ one sees that there is no 
commuting real structure.
\item $n\equiv 4\bmod 8$: By \cite{fr00} there exists an anti-commuting 
quaternionic structure $J$. 
Then $\omega_{\C}\circ J$ is a commuting quaternionic structure. 
As in the case $n\equiv 1\bmod 8$ one sees that there is no 
commuting real structure.
\item $n\equiv 5\bmod 8$: By \cite{fr00} there exists an anti-commuting 
quaternionic structure $J$. 
As in the case $n\equiv1\bmod 8$ one sees that there is no 
commuting real structure.
\item $n\equiv 6\bmod 8$: By \cite{fr00} there exists a commuting 
real structure $J$. 
Then $\omega_{\C}\circ J$ is an anti-commuting quaternionic structure.
\item $n\equiv 7\bmod 8$: By \cite{fr00} there exists a commuting real structure.
\end{enumerate}
The assertion follows.
\end{proof}

As an example we consider the case $n=2$. 
The spinor representation of $\CCl(2)=\End_{\C}(\C^2)$ is the 
standard representation of $\Mat(2,\C)$ on $\C^2$. 
We define the action of $E_1$ and $E_2$ by
\begin{displaymath}
  E_1=\bmatrix{cc}{0 & -1 \\ 1 & 0},\quad
  E_2=\bmatrix{cc}{0 & i \\ i & 0}.
\end{displaymath}
We define the map $J$: $\C^2\to\C^2$ by 
$J(z,w)=(-\overline{w},\overline{z})$. 
Then $J$ is a commuting quaternionic structure. 
A motivation for this definition comes from considering the Hamilton quaternions 
\begin{displaymath}
  \quat:=\{a+ib+cj+dk|a,b,c,d\in\R,i^2=j^2=k^2=-1,ij=-ji=k\}
\end{displaymath}
as a complex vector space, $\C$ acting by quaternion multiplication 
from the right, and from using the $\C$-linear isomorphism
\begin{displaymath}
  \C^2\to\quat,\quad(z,w)\mapsto z+jw.
\end{displaymath}
Under this isomorphism the actions of $E_1$, $E_2$ correspond to 
quaternion multiplication with $j$, $-k$ from the left and $J$ 
corresponds to quaternion multiplication with $j$ from the right.
We define the map $K$: $\C^2\to\C^2$ by 
$K(z,w):=(-\overline{w},-\overline{z})$. 
Then $K$ is an anticommuting real structure. 

Next let $n=3$. 
The Clifford algebra $\CCl(3)=\End_{\C}(\C^2)\oplus\End_{\C}(\C^2)$ 
has two inequivalent irreducible complex representations $\rho_1$, $\rho_2$ given by 
\begin{displaymath}
  \rho_1(A,B)(x):=Ax,\quad\rho_2(A,B)(x):=Bx
\end{displaymath}
for $x\in\C^2$. We define the action of $E_1$, $E_2$, $E_3$ under $\rho_1$ by 
\begin{displaymath}
  E_1=\bmatrix{cc}{i & 0 \\ 0 & -i},\quad
  E_2=\bmatrix{cc}{0 & -1 \\ 1 & 0},\quad
  E_3=\bmatrix{cc}{0 & -i \\ -i & 0 }.
\end{displaymath}
Under the isomorphism 
\begin{displaymath}
  \C^2\to\quat,\quad(z,w)\mapsto z+jw.
\end{displaymath}
this corresponds to quaternion multiplication with $i$, $j$, $k$ from the left. 
The map $J$: $\C^2\to\C^2$ by $J(z,w)=(-\overline{w},\overline{z})$ is a commuting 
quaternionic structure.

Now let $(M,g,\Theta)$ be a Riemannian spin manifold of dimension $n$. 
Let $J$ be one of the structures mentioned above. 
Then $J$ commutes with the action of $\Spin(n)$ on $\Sigma_n$ 
and thus induces a map 
$\Sigma^gM\to\Sigma^gM$, $[s,\sigma]\mapsto[s,J\sigma]$ 
which will again be denoted by $J$. This map is 
fibre-preserving and $\R$-linear on the fibres and it 
satisfies $Ji=-iJ$ and $J^2=\pm\id$. 
Furthermore for any $g\in\Spin(n)$ we have 
\begin{displaymath}
  J\circ\rho(g)=\rho(g)\circ J.
\end{displaymath}
Taking the derivative of this equation we obtain $J\circ d\rho=d\rho\circ J$. 
It follows from the formula (\ref{cov_deriv}) that $\nabla^g_XJ\psi=J\nabla^g_X\psi$ 
for all $X\in TM$ and for all $\psi\in\spinor{M}{g}{\infty}$, 
i.\,e.\,$J$ is parallel with respect to $\nabla^g$. 

We note the following important consequences for the 
spectrum of the Dirac operator.

\begin{remark}
\label{remark_real_quat}
If $J$ is an anti-commuting real structure, it anti-commutes 
with the Dirac operator. It follows that if $\lambda$ 
is an eigenvalue of the Dirac operator, then also $-\lambda$ 
is an eigenvalue. 
By Theorem \ref{theorem_real_quat} this occurs in dimensions 
$n\equiv 0,1,2\bmod 8$.

If $J$ is a commuting quaternionic structure, it 
commutes with the Dirac operator. Since for every nonzero spinor 
$\psi$ the system $\{\psi,J\psi\}$ is linearly independent over $\C$, 
every eigenspace of $D^g$ has even complex dimension. 
By Theorem \ref{theorem_real_quat} this is the case 
in dimensions $n\equiv2,3,4\bmod 8$.

If $J$ is an anti-commuting quaternionic structure, 
it anti-commutes with the Dirac operator. In this case the kernel 
of $D^g$ has even complex dimension and, if $\lambda$ 
is an eigenvalue of $D^g$, then also $-\lambda$ is an eigenvalue 
of $D^g$. By Theorem \ref{theorem_real_quat} this occurs in dimensions 
$n\equiv 4,5,6\bmod 8$.
\end{remark}

\section{Spinors for different metrics}
\label{identification_section}

Let $(M,g,\Theta)$ be a Riemannian spin manifold and let $h$ be another 
Riemannian metric on $M$. 
Since the spinor bundles $\Sigma^gM$ and $\Sigma^hM$ 
are two different vector bundles, the question 
arises how one can identify spinors on $(M,g,\Theta)$ with spinors 
on $(M,h,\Theta)$ in a natural way. The case of conformally related 
metrics $g$ and $h$ has been treated in \cite{hit74}, \cite{hij86}. 
For general Riemannian metrics $g$ and $h$ Bourguignon and Gauduchon 
\cite{bg} have solved this problem. 
The question when such an identification can be obtained in 
the case of semi-Riemannian metrics has been treated in \cite{bgm05}. 
For our purpose we will recall the method of \cite{bg} 
and use some remarks from \cite{m}.

Given the metrics $g$ and $h$ there exists a unique 
endomorphism $a_{g,h}$ of $TM$ such that for all $x\in M$ 
and for all $v$, $w\in T_xM$ we have
\begin{displaymath}
  g(a_{g,h} v,w)=h(v, w).
\end{displaymath}
For each $x\in M$ the endomorphism $a_{g,h}(x)\in\End(T_xM)$ is 
$g$-self-adjoint and positive definite. Thus there exists a 
unique endomorphism $b_{g,h}(x)$ of $T_xM$ which is positive 
definite and satisfies $b_{g,h}(x)^2=a_{g,h}(x)^{-1}$. 
In this way we obtain an endomorphism $b_{g,h}$ of $TM$. 
The compositions of the endomorphisms $a_{g,h}$ and $b_{g,h}$ 
have the following properties. 

\begin{lemma}
\label{lemma_agh_ahk}
Let $g$, $h$, $k$ be three Riemannian metrics on $M$. Then we have $a_{g,h}a_{h,k}=a_{g,k}$.
The equation $b_{g,h}\circ b_{h,k}=b_{g,k}$ holds if and only if $a_{g,h}$ and $a_{h,k}$ commute.
\end{lemma}

\begin{proof}
This follows immediately from the definitions of $a_{g,h}$ and $b_{g,h}$.
\end{proof}

For example if $g$, $h$, $k$ are conformally related, 
then $a_{g,h}$ and $a_{h,k}$ commute. 

\begin{lemma}
The endomorphism $b_{g,h}$ induces an isomorphism 
of principal $SO(n)$-bundles
\begin{displaymath}
  c_{g,h}:\quad \P_{\SO}(M,g)\to\P_{\SO}(M,h),
  \quad (e_i)_{i=1}^n\mapsto(b_{g,h} e_i)_{i=1}^n
\end{displaymath}
\end{lemma}

\begin{proof}
One easily finds that $b_{g,h}$ maps positively oriented $g$-orthonormal bases 
to positively oriented $h$-ortho\-nor\-mal bases. Obviously the map $c_{g,h}$ 
is $SO(n)$-equivariant and an isomorphism with $c_{h,g}=c_{g,h}^{-1}$.
\end{proof}

\begin{lemma}
The isomorphism $c_{g,h}$ lifts to an isomorphism of principal 
$\Spin(n)$-bundles
\begin{displaymath}
  \gamma_{g,h}:\quad\P_{\Spin}(M,g)\to\P_{\Spin}(M,h).
\end{displaymath}
\end{lemma}

\begin{proof}
Let $\Theta$: $\P_{\Spin}(M,g)\to\P_{\SO}(M,g)$ be the covering map. 
We define a family of Riemannian metrics $(g_t)_{t\in[0,1]}$ 
on $M$ by $g_t:=(1-t)g+th$ and we define the map
\begin{displaymath}
  F:\quad \P_{\Spin}(M,g)\times[0,1]\to\P_{\GL^+}(M),\quad (s,t)\mapsto c_{g,g_t}(\Theta(s)).
\end{displaymath}
We consider the following commutative diagram
\[
\begin{xy}
  \xymatrix{          
    \P_{\Spin}(M,g)\times\{0\} \ar[r]^/0.7em/j  \ar[d]_i & \P_{\widetilde{\GL}^+}(M) \ar[d]^{\Theta} \\ 
    \P_{\Spin}(M,g)\times[0,1] \ar[r]^/0.9em/F           & \P_{\GL^+}(M) 
  }
\end{xy}
\]
where $i$ and $j$ are inclusions. Since the map $\Theta$ has 
the homotopy lifting property, there exists a unique map
\begin{displaymath}
  G:\quad \P_{\Spin}(M,g)\times[0,1]\to\P_{\widetilde{\GL}^+}(M)
\end{displaymath}
such that 
$\Theta\circ G=F$ and $G(s,0)=s$ for all $s\in\P_{\Spin}(M,g)$. 
We define $\gamma_{g,h}(s):=G(s,1)$. 
The definition of $\gamma_{g,h}$ does not depend on the family 
of Riemannian metrics $g_t$ chosen above. Namely any two paths 
between $g$ and $h$ are homotopic and therefore yield the same 
result $G(s,1)$. We find that $\gamma_{g,h}$ is $\Spin(n)$-equivariant, 
since $c_{g,g_t}$ is $\SO(n)$-equivariant. 
Using the uniqueness of lifts one can show that 
$\gamma_{g,h}$ is an isomorphism with $\gamma_{h,g}=\gamma_{g,h}^{-1}$.
\end{proof}

\begin{lemma}
\label{definition_beta}
The isomorphism $\gamma_{g,h}$ induces an isomorphism of vector bundles
\begin{displaymath}
  \beta_{g,h}:\quad\Sigma^gM\to\Sigma^hM,\quad [s,\sigma]\mapsto[\gamma_{g,h}(s),\sigma]
\end{displaymath}
which is a fibrewise isometry with respect to the hermitian metrics on 
$\Sigma^gM$ and $\Sigma^hM$. Furthermore for all $v\in TM$ and for all 
$\psi\in\Sigma^gM$ we have $\beta_{g,h}(v\cdot\psi)=b_{g,h}v\cdot\beta_{g,h}\psi$, 
where $\cdot$ denotes both Clifford multiplications on $\Sigma^gM$ and on $\Sigma^hM$.
\end{lemma}

\begin{proof}
The map $\beta_{g,h}$ is well defined, since $\gamma_{g,h}$ 
is $\Spin(n)$-equivariant. We find that $\beta_{h,g}=\beta_{g,h}^{-1}$. 
By the definition of the hermitian metrics on $\Sigma^gM$ 
and $\Sigma^hM$ the map $\beta_{g,h}$ is a fibrewise isometry. 
For the last assertion observe that under the isomorphism 
$TM\cong\P_{\SO}(M,g)\times_{\tau}\R^n$ for all $s\in\P_{\Spin}(M,g)$, 
$w\in\R^n$ we have
\begin{displaymath}
  b_{g,h}([\Theta(s),w])=[c_{g,h}(\Theta(s)),w]=[\Theta(\gamma_{g,h}(s)),w]
\end{displaymath}
and therefore for all $s\in\P_{\Spin}(M,g)$, $w\in\R^n$, $\sigma\in\Sigma_n$
\begin{displaymath}
  b_{g,h}([\Theta(s),w])\cdot\beta_{g,h}([s,\sigma])=[\gamma_{g,h}(s),w\cdot\sigma]
  =\beta_{g,h}([\Theta(s),w]\cdot[s,\sigma])
\end{displaymath}
which completes the proof.
\end{proof}

Next we want to compare the Dirac operators $D^g$ and $D^h$. 
The map $\beta_{g,h}$ does not induce an isometry of Hilbert spaces 
$\Lspinor{M}{g}{2}\to\Lspinor{M}{h}{2}$, since the volume forms 
$\dv^g$, $\dv^h$ induced by the metrics $g$, $h$ are different (see \cite{m}). 
In order to compensate this we note, that there exists a smooth positive function $f_{g,h}$ 
on $M$ such that $\dv^h=f_{g,h}^2\dv^g$. We define
\begin{displaymath}
  \overline{\beta}_{g,h}:=\frac{1}{f_{g,h}}\beta_{g,h}:\quad
  \Sigma^gM\to\Sigma^hM.
\end{displaymath}
The maps $\beta_{g,h}$, $\overline{\beta}_{g,h}$ induce isomorphisms 
$\spinor{M}{g}{\infty}\to\spinor{M}{h}{\infty}$, which will also be 
denoted by $\beta_{g,h}$, $\overline{\beta}_{g,h}$. We use the map 
$\overline{\beta}_{g,h}$ to pull back the Dirac operator on $\Sigma^hM$ to 
spinors for the metric $g$.
\begin{displaymath}
  D^{g,h}:=\overline{\beta}_{h,g} D^h \overline{\beta}_{g,h}.
\end{displaymath}

We see that these operators have the following properties.

\begin{lemma}
\label{definition_beta_bar}
The map $\overline{\beta}_{g,h}$ induces an isometry of Hilbert spaces 
\begin{displaymath}
  \overline{\beta}_{g,h}:\quad\Lspinor{M}{g}{2}\to\Lspinor{M}{h}{2}.
\end{displaymath}
If $D^h$ has self-adjoint closure on $\Lspinor{M}{h}{2}$, then $D^{g,h}$ has 
self-adjoint closure on $\Lspinor{M}{g}{2}$.
\end{lemma}

\begin{proof}
This is clear from the definitions of $\overline{\beta}_{g,h}$ and $D^{g,h}$.
\end{proof}

Next we want to express the operator $D^{g,h}$ in terms of a local orthonormal frame. 
For a Riemannian metric $g$ on $M$ we denote both the Levi Civita connection on $TM$ and the 
induced connection on $\Sigma^gM$ by $\nabla^g$.

\begin{theorem}
Let $(e_i)_{i=1}^n$ be a local $g$-orthonormal frame defined on an open 
subset $U\subset M$. Then for all $\psi\in\spinor{M}{g}{\infty}$ we have on $U$:
\begin{eqnarray}
  \nonumber
  D^{g,h}\psi&=&\sum_{i=1}^n e_i\cdot\nabla^g_{b_{g,h}(e_i)}\psi\\
  \nonumber
  &&{}+\frac{1}{4}\sum_{i,j=1}^n e_i\cdot e_j\cdot(b_{h,g}\nabla^h_{b_{g,h(e_i)}}(b_{g,h}e_j)-\nabla^g_{b_{g,h}(e_i)}e_j)\cdot\psi\\
  \label{Dgh}
  &&{}-\frac{1}{f_{g,h}} b_{g,h}(\grad^g (f_{g,h}))\cdot\psi.
\end{eqnarray}
\end{theorem}

Here $\grad^g(f)$ denotes the gradient of a smooth function $f$ with respect to the metric $g$.

\begin{proof}
We find
\begin{eqnarray*}
  D^{g,h}\psi&=&\overline{\beta}_{h,g}\big(\frac{1}{f_{g,h}}D^h\beta_{g,h}\psi-\frac{1}{f_{g,h}^2}\grad^h(f_{g,h})\cdot\beta_{g,h}\psi\big)\\
  &=&\beta_{h,g}D^h\beta_{g,h}\psi-\frac{1}{f_{g,h}}b_{h,g}(\grad^h(f_{g,h}))\cdot\psi\\
  &=&\beta_{h,g}D^h\beta_{g,h}\psi-\frac{1}{f_{g,h}}b_{g,h}(\grad^g(f_{g,h}))\cdot\psi.
\end{eqnarray*}
For the first summand we obtain
\begin{eqnarray*}
  \beta_{h,g}D^h\beta_{g,h}\psi&=&\beta_{h,g}\big(\sum_{i=1}^n b_{g,h}e_i\cdot\nabla^h_{b_{g,h}(e_i)}\beta_{g,h}\psi\big)\\
  &=&\sum_{i=1}^n e_i\cdot\beta_{h,g}\nabla^h_{b_{g,h}(e_i)}\beta_{g,h}\psi\\
  &=&\sum_{i=1}^n e_i\cdot\nabla^g_{b_{g,h}(e_i)}\psi+\sum_{i=1}^n e_i\cdot
  (\beta_{h,g}\nabla^h_{b_{g,h}(e_i)}\beta_{g,h}\psi-\nabla^g_{b_{g,h}(e_i)}\psi).
\end{eqnarray*}
Using the formula (\ref{connection_local}) we get
\begin{eqnarray*}
  &&\beta_{h,g}\nabla^h_{b_{g,h}(e_i)}\beta_{g,h}\psi-\nabla^g_{b_{g,h}(e_i)}\psi\\
  &=&\frac{1}{4}\sum_{j=1}^n \beta_{h,g}(b_{g,h}e_j\cdot(\nabla^h_{b_{g,h}(e_i)}(b_{g,h}e_j))\cdot\beta_{g,h}\psi) 
  -\frac{1}{4}\sum_{j=1}^n e_j\cdot (\nabla^g_{b_{g,h}(e_i)}e_j)\cdot\psi\\
  &=&\frac{1}{4}\sum_{j=1}^n e_j\cdot (b_{h,g}\nabla^h_{b_{g,h}(e_i)}(b_{g,h}e_j)-\nabla^g_{b_{g,h}(e_i)}e_j)\cdot\psi.
\end{eqnarray*}
This gives the formula of the assertion.
\end{proof}

For $r\in\N$ denote by $\sym{r}$ (resp. $\sym{\infty}$) the space of 
$r$ times continuously differentiable (resp. smooth) symmetric $(2,0)$ 
tensor fields on $M$. 
In order to compute the derivative of $D^{g,h}$ with respect to the metric $h$ 
we let $k\in\sym{\infty}$. 
Then there exists an open neighborhood $I\subset\R$ 
of $0$ such that for every $t\in I$ the tensor field $g_t:=g+tk$ 
is a Riemannian metric on $M$.

\begin{theorem}
Let $(e_i)_{i=1}^n$ be a local $g$-orthonormal frame defined on an open 
subset $U\subset M$. Then for all $\psi\in\spinor{M}{g}{\infty}$ we have on $U$:
\begin{equation}
  \label{dirac_derivative}
  \frac{d}{dt}D^{g,g_t}\big|_{t=0}\psi=-\frac{1}{2}\sum_{i=1}^n e_i\cdot\nabla^g_{a_{g,k}(e_i)}\psi
  -\frac{1}{4}\sum_{i=1}^n \div^g(k)(e_i)e_i\cdot\psi.
\end{equation}
\end{theorem}

Here $\div^g(k)$ denotes the divergence\index{divergence} of $k$ with respect to the metric $g$. 
If $(e_i)_{i=1}^n$ is a local orthonormal frame of $TM$, it is the 
one-form on $M$ which is locally defined by
\begin{displaymath}
  \div^g(k)(X):=\sum_{i=1}^n(\nabla^g_{e_i} k)(X,e_i)
\end{displaymath}
for all $X\in TM$. By $\tr^g(k)\in C^{\infty}(M,\R)$ we will denote 
the $g$-trace\index{trace} of $k$, which is locally defined by
\begin{displaymath}
  \tr^g(k):=\sum_{i=1}^n k(e_i,e_i).
\end{displaymath}
Of course these definitions are independent of the choice of 
orthonormal frame.

\begin{proof}
From $g_t=g+tk$ we obtain $b_{g,g_t}=(\id+ta_{g,k})^{-1/2}$. 
It follows that 
\begin{displaymath}
  \frac{d}{dt}b_{g,g_t}\big|_{t=0}=-\frac{1}{2}a_{g,k}.
\end{displaymath}
The first summand is now obtained from the first summand of (\ref{Dgh}). 
By \cite{bes}, p. 62, we have for all $X$, $Y$, $Z\in TM$
\begin{displaymath}
  2g(\frac{d}{dt}\nabla^{g_t}_X Y\big|_{t=0},Z)=(\nabla^g_X k)(Y,Z)+(\nabla^g_Y k)(X,Z)
  -(\nabla^g_Z k)(X,Y).
\end{displaymath}
We calculate
\begin{eqnarray*}
  &&\frac{d}{dt}(b_{g_t,g}\nabla^{g_t}_{b_{g,g_t(e_i)}}(b_{g,g_t}e_j)-\nabla^g_{b_{g,g_t}(e_i)}e_j)\big|_{t=0}\\
  &=&\frac{1}{2}a_{g,k}(\nabla^g_{e_i}e_j)+\frac{d}{dt}\nabla^{g_t}_{e_i}e_j\big|_{t=0}
  -\frac{1}{2}\nabla^g_{a_{g,k}(e_i)}e_j-\frac{1}{2}\nabla^g_{e_i}(a_{g,k}e_j)
  +\frac{1}{2}\nabla^g_{a_{g,k}(e_i)}e_j\\
  &=&\frac{1}{2}\sum_{m=1}^n k(\nabla^g_{e_i}e_j,e_m)e_m+\frac{d}{dt}\nabla^{g_t}_{e_i}e_j\big|_{t=0}
  -\frac{1}{2}\sum_{m=1}^n g(\nabla^g_{e_i}(a_{g,k}e_j),e_m)e_m\\
  &=&\frac{1}{2}\sum_{m=1}^n (k(\nabla^g_{e_i}e_j,e_m)+k(e_j,\nabla^g_{e_i}e_m)-\partial_{e_i} k(e_j,e_m))e_m
  +\frac{d}{dt}\nabla^{g_t}_{e_i}e_j\big|_{t=0}\\
  &=&-\frac{1}{2}\sum_{m=1}^n (\nabla^g_{e_i}k)(e_j,e_m)e_m+\frac{d}{dt}\nabla^{g_t}_{e_i}e_j\big|_{t=0}\\
  &=&\frac{1}{2}\sum_{m=1}^n ((\nabla^g_{e_j}k)(e_i,e_m)-(\nabla^g_{e_m}k)(e_i,e_j))e_m.
\end{eqnarray*}
It follows that
\begin{eqnarray*}
  &&\frac{d}{dt}
  \sum_{i,j=1}^n e_i\cdot e_j\cdot(b_{g_t,g}\nabla^{g_t}_{b_{g,g_t(e_i)}}(b_{g,g_t}e_j)-\nabla^g_{b_{g,g_t}(e_i)}e_j)\big|_{t=0}\\
  &=&\frac{1}{2}\sum_{i,j,m=1}^n (\nabla^g_{e_j}k)(e_i,e_m)e_i\cdot e_j\cdot e_m
  -\frac{1}{2}\sum_{i,j,m=1}^n (\nabla^g_{e_m}k)(e_i,e_j)e_i\cdot e_j\cdot e_m\\
  &=&-\sum_{i,j=1}^n (\nabla^g_{e_j}k)(e_i,e_j)e_i
  -\sum_{i,j,m=1}^n (\nabla^g_{e_m}k)(e_i,e_j)e_i\cdot e_j\cdot e_m\\
  &=&-\sum_{i=1}^n \div^g(k)(e_i)e_i+\grad^g(\tr^g(k)).
\end{eqnarray*}
From $\dv^{g_t}=\det(\id+ta_{g,k})^{1/2}\dv^g$ 
it follows that $f_{g,g_t}=\det(\id+ta_{g,k})^{1/4}$. 
Since $f_{g,g}\equiv1$ we obtain
\begin{displaymath}
  \frac{d}{dt}\frac{1}{f_{g,g_t}} b_{g,g_t}(\grad^g (f_{g,g_t}))\big|_{t=0}
  =\grad^g(\frac{d}{dt}f_{g,g_t}\big|_{t=0})=\frac{1}{4}\grad^g(\tr^g(k)).
\end{displaymath}
The assertion follows.
\end{proof}

We define the following $(2,0)$ tensor field on $M$. 

\begin{definition}
Let $\psi\in\spinor{M}{g}{\infty}$. The energy momentum tensor\index{energy momentum tensor} for $\psi$ 
is a symmetric $(2,0)$ tensor field on $M$ defined by
\begin{displaymath}
  Q_{\psi}(X,Y):=
  \frac{1}{2}\Re\langle X\cdot\nabla^g_Y\psi+Y\cdot\nabla^g_X\psi,\psi\rangle
\end{displaymath}
for $X,Y\in TM$.
\end{definition}

Again let $k\in\sym{\infty}$ and let $I\subset\R$ be an open interval 
containing $0$ such that for all $t\in I$ the tensor field $g_t=g+tk$ 
is a Riemannian metric on $M$.
Let $\lambda$ be an eigenvalue of $D^g$ with 
$d:=\dim_{\C}\ker(D^g-\lambda)$.
By Theorem \ref{rellich_theorem} and Lemma \ref{analyticity_lemma} 
there exist real-analytic functions $\lambda_1$,...,$\lambda_d$ on $I$ such that 
$\lambda_j(t)$ is an eigenvalue of $D^{g,g_t}$ for all $j$ and all $t$ and 
$\lambda_j(0)=\lambda$ for all $j$. 
Furthermore there exist spinors $\psi_{j,t}$, $1\leq j\leq n$, $t\in I$, 
which are real-analytic in $t$, such that $\psi_{j,t}$ is an eigenspinor 
of $D^{g,g_t}$ corresponding to $\lambda_j(t)$. We can choose these spinors 
such that $\|\psi_{j,t}\|_2=1$ for all $j$ and all $t$. 
Then we have
\begin{displaymath}
  \lambda_j(t)=(\psi_{j,t},D^{g,g_t}\psi_{j,t})_2
\end{displaymath}
for all $t$. 
If we take the derivative with respect to $t$ at $t=0$, then, 
since $D^g$ is self-adjoint and $\psi_{j,t}$ is normalized, the only 
contribution comes from the derivative of the Dirac operator. 
\begin{equation}
  \label{lambda_derivative1}
  \frac{d\lambda_j(t)}{dt}\big|_{t=0}=\big(\psi_{j,0},\frac{d}{dt}D^{g,g_t}\big|_{t=0}\psi_{j,0}\big)_2
\end{equation}
Recall that if $(e_i)_{i=1}^n$ is a local orthonormal frame, then by (\ref{dirac_derivative}) we have locally 
\begin{displaymath}
  \frac{d}{dt}D^{g,g_t}\big|_{t=0}\psi_{j,0}=-\frac{1}{2}\sum_{i=1}^n e_i\cdot\nabla^g_{a_{g,k}(e_i)}\psi_{j,0}
  -\frac{1}{4}\sum_{i=1}^n \div^g(k)(e_i)e_i\cdot\psi_{j,0}.
\end{displaymath} 
The scalar product of the first term with $\psi_{j,0}$ can locally be written as
\begin{eqnarray*}
  &&-\frac{1}{2}\sum_{i=1}^n \langle \psi_{j,0}, e_i\cdot\nabla^g_{a_{g,k}(e_i)} \psi_{j,0} \rangle\\
  &=&-\frac{1}{4}\sum_{i,m=1}^n k(e_i,e_m) \langle \psi_{j,0}, e_i\cdot\nabla^g_{e_m} \psi_{j,0}
  +e_m\cdot\nabla^g_{e_i}\psi_{j,0} \rangle
\end{eqnarray*}
The scalar product of the second term with $\psi_{j,0}$ is purely imaginary. 
Since the left hand side of (\ref{lambda_derivative1}) is real, 
it is sufficient to consider the real part. 
Then the second term gives no contribution. 
If $(.,.)$ denotes the standard pointwise inner product of $(2,0)$ tensor fields, 
then we obtain
\begin{equation}
  \label{lambda_derivative2}
  \frac{d\lambda_j(t)}{dt}\big|_{t=0}=-\frac{1}{2}\int_M(k,Q_{\psi_{j,0}})\dv^g.
\end{equation}

We consider now the special case of conformally related Riemannian metrics on $M$. 
Let $u\in C^{\infty}(M,\R)$ and let $g$, $h$ be two Riemannian metrics on $M$ such that 
$h=e^{2u}g$. The connections on the spinor bundles $\Sigma^gM$ and $\Sigma^hM$ are then 
related in the following way.

\begin{theorem}
For all $\psi\in\spinor{M}{g}{\infty}$ and for all $X\in TM$ we have
\begin{eqnarray}
  \nonumber
  \nabla^{h}_X \beta_{g,h}\psi&=&\beta_{g,h}\big( \nabla^g_X\psi-\frac{1}{2}\,X\cdot\grad^g(u)\cdot\psi -\frac{1}{2}\,X(u)\psi\big)\\
  \label{nabla_conform}
  \nabla^{h}_X \overline{\beta}_{g,h}\psi&=&\overline{\beta}_{g,h}\big( \nabla^g_X\psi-\frac{1}{2}\,X\cdot\grad^g(u)\cdot\psi -\frac{n+1}{2}\,X(u)\psi\big).
\end{eqnarray}
\end{theorem}

\begin{proof}
By \cite{bes}, p. 58 we have for all $X$, $Y\in TM$
\begin{displaymath}
  \nabla^h_X Y=\nabla^g_X Y+X(u)Y+Y(u)X-g(X,Y)\grad^g(u).
\end{displaymath}
Note that $b_{g,h}=e^{-u}\id$. 
Let $(e_i)_{i=1}^n$ be a local $g$-orthonormal frame on an open subset $U\subset M$. 
Then $(b_{g,h}e_i)_{i=1}^n$ is a local $h$-orthonormal frame on $U$. 
We obtain by the formula (\ref{connection_local})
\begin{eqnarray*}
  \nabla^h_X\beta_{g,h}\psi&=&\partial_X \beta_{g,h}\psi+\frac{1}{4}\sum_{j=1}^n b_{g,h}e_j\cdot(\nabla^h_X b_{g,h}e_j)\cdot\beta_{g,h}\psi\\
  &=&\partial_X \beta_{g,h}\psi+\frac{1}{4}\sum_{j=1}^n b_{g,h}e_j\cdot b_{g,h}(\nabla^h_X e_j-X(u)e_j)\cdot\beta_{g,h}\psi\\
  &=&\partial_X \beta_{g,h}\psi+\frac{1}{4}\sum_{j=1}^n \beta_{g,h}(e_j\cdot(\nabla^g_X e_j+e_j(u)X)\cdot\psi)\\
  &&{}-\frac{1}{4}\beta_{g,h}(X\cdot\grad^g(u)\cdot\psi)\\
  &=&\beta_{g,h}(\nabla^g_X\psi+\frac{1}{4}\grad^g(u)\cdot X\cdot\psi-\frac{1}{4}X\cdot\grad^g(u)\cdot\psi).
\end{eqnarray*}
The first formula in the assertion now follows from (\ref{clifford_relation}). 
Furthermore using that $f_{g,h}=e^{nu/2}$ we get 
\begin{displaymath}
  \nabla^h_X\overline{\beta}_{g,h}\psi=X(e^{-nu/2})\beta_{g,h}\psi+e^{-nu/2}\nabla^h_X\beta_{g,h}\psi.
\end{displaymath}
The second formula now follows.
\end{proof}

The Dirac operators on $\Sigma^gM$ and $\Sigma^hM$ are related by the following formulas.

\begin{theorem}
For all $\psi\in\spinor{M}{g}{\infty}$ we have
\begin{eqnarray}
  \nonumber
  D^{h}(e^{-(n-1)u/2}\beta_{g,h}\psi)&=&e^{-(n+1)u/2}\beta_{g,h}D^g\psi\\
  \label{dirac_conform}
  D^{h}(e^{u/2}\overline{\beta}_{g,h}\psi)&=&e^{-u/2}\overline{\beta}_{g,h}D^{g}\psi.
\end{eqnarray}
\end{theorem}

\begin{proof}
Let $(e_i)_{i=1}^n$ be a local orthonormal frame. Using the formula (\ref{nabla_conform}) we find
\begin{eqnarray*}
  D^h\beta_{g,h}\psi&=&\sum_{i=1}^n b_{g,h}(e_i)\cdot\nabla^h_{b_{g,h}(e_i)}\beta_{g,h}\psi\\
  &=&\sum_{i=1}^n \beta_{g,h}\big(e_i\cdot e^{-u}(\nabla^g_{e_i}\psi-\frac{1}{2}e_i(u)\psi-\frac{1}{2}e_i\cdot\grad^g(u)\cdot\psi)\big)\\
  &=&\beta_{g,h}\big(e^{-u}D^g\psi+\frac{n-1}{2}e^{-u}\textrm{grad}^g(u)\cdot\psi\big).
\end{eqnarray*}
The first formula follows from this calculation. 
Using $\overline{\beta}_{g,h}=e^{-nu/2}\beta_{g,h}$ one obtains the second formula. 
\end{proof}

We specialise the above formulas for the derivative of the Dirac operator and 
the derivative of the eigenvalue to the case of conformal changes of the Riemannian metric. 
Let $f\in C^{\infty}(M,\R)$ and let $I\subset\R$ be an open interval around $0$ such that 
for all $t\in I$ the tensor field $g_t:=g+tfg$ is a Riemannian metric on $M$. 
Then we immediately obtain the following.

\begin{theorem}
For all $\psi\in\spinor{M}{g}{\infty}$ we have
\begin{equation}
  \label{dirac_derivative_conform}
  \frac{d}{dt}D^{g,g_t}\big|_{t=0}\psi=-\frac{f}{2}D^g\psi-\frac{1}{4}\grad^g(f)\cdot\psi
\end{equation}
and
\begin{equation}
  \label{lambda_derivative_conform}
  \frac{d\lambda_j(t)}{dt}\big|_{t=0}=-\frac{\lambda}{2}\int_M f|\psi_{j,0}|_g^2\dv^g.
\end{equation}
\end{theorem}

\chapter{Motivation}
\label{motivation_chapter}

\section{Conformal bounds on Dirac eigenvalues}

Let $(M,g,\Theta)$ be a compact Riemannian spin manifold of dimension $n\geq2$. 
We denote by $\lambda_1^+(g)$ the smallest positive eigenvalue of $D^g$ and 
by $\lambda_1^-(g)$ the largest negative eigenvalue of $D^g$. 
The question arises how the values $\lambda_1^{\pm}(g)$ depend on the geometry 
of the underlying manifold. 
Lower bounds on $|\lambda_1^{\pm}(g)|$ in terms of the scalar curvature of $(M,g)$ have been obtained in 
\cite{li}, \cite{fr80}, \cite{ki86}, \cite{ki88}, \cite{ksw98}, \cite{ksw99}. 
Searching for lower bounds which are uniform on the conformal class\index{conformal class} 
\begin{displaymath}
  [g]:=\{e^{2u}g|u\in C^{\infty}(M,\R)\}
\end{displaymath}
of the metric $g$ one observes first that for every constant $a>0$ one has 
$\lambda_1^{\pm}(a^2g)=a^{-1}\lambda_1^{\pm}(g)$ by (\ref{dirac_conform}). 
Thus the expressions $|\lambda_1^{\pm}(g)|\vol(M,g)^{1/n}$ are invariant under 
constant rescalings of the metric. 
We define the two invariants 
\begin{eqnarray*}
  \lambda_{\min}^+(M,[g],\Theta)&:=&\inf_{h\in[g]}\lambda_1^+(h)\vol(M,h)^{1/n}\\
  \lambda_{\min}^-(M,[g],\Theta)&:=&\inf_{h\in[g]}|\lambda_1^-(h)|\vol(M,h)^{1/n}
\end{eqnarray*}
of the conformal class $[g]$. It was proven in \cite{am03} that
\begin{displaymath}
  \lambda_{\min}^{\pm}(M,[g],\Theta)>0.
\end{displaymath}
These invariants have been treated by many authors 
(see e.\,g.\,\cite{hij86}, \cite{hij91}, \cite{lo}, \cite{baer92}, \cite{am03hab}, \cite{am09}). 

Explicit values are known only for very few spin manifolds. 
For example if $\gcan$ denotes the standard metric on $S^n$ and 
$\sigma$ the unique spin structure on $S^n$, we use the notation 
$\mathbb{S}^n:=(S^n,[\gcan],\sigma)$. Then we have 
\begin{displaymath}
  \lambda_{\min}^+(\mathbb{S}^n)=\lambda_{\min}^-(\mathbb{S}^n)=\frac{n}{2}\vol(S^n,\gcan)^{1/n}
\end{displaymath}
(see e.\,g.\,\cite{am03}, \cite{am09}). 
Furthermore one has the following inequalities.

\begin{theorem}[\cite{am03}, \cite{aghm}]
For every compact Riemannian spin manifold $(M,g,\Theta)$ we have 
\begin{displaymath}
  \lambda_{\min}^+(M,[g],\Theta)\leq\lambda_{\min}^+(\mathbb{S}^n),\quad
  \lambda_{\min}^-(M,[g],\Theta)\leq\lambda_{\min}^+(\mathbb{S}^n).
\end{displaymath}
\end{theorem}

It is a natural question whether the infimum in the definition of the invariants 
$\lambda_{\min}^{\pm}(M,[g],\Theta)$ 
is attained at a Riemannian metric $h\in[g]$. 
This has been investigated in \cite{am09}, where the following result is obtained. 

\begin{theorem}
\label{theorem_lambdamin+}
Let $(M,g,\Theta)$ be a compact Riemannian spin manifold of dimension $n\geq2$. 
Assume that we have the strict inequality 
\begin{displaymath}
  \lambda_{\min}^+:=\lambda_{\min}^+(M,[g],\Theta)<\lambda_{\min}^+(\mathbb{S}^n).
\end{displaymath}
Then there exists a spinor $\psi\in\spinor{M}{g}{2}$, 
which is smooth on $M\setminus\psi^{-1}(0)$, such that 
\begin{equation}
  \label{nonlinear_ev_equation}
  D^g\psi=\lambda_{\min}^+|\psi|_g^{2/(n-1)}\psi,\qquad 
  \|\psi\|_{2n/(n-1)}=1.
\end{equation}
If there is such a spinor $\psi$, which is nowhere zero on $M$, 
then $h:=|\psi|_g^{4/(n-1)}g$ is a Riemannian metric on $M$ such that 
\begin{displaymath}
  \lambda_{\min}^+=\lambda_1^+(h)\vol(M,h)^{1/n}.
\end{displaymath}
An analogous assertion holds for $\lambda_{\min}^-(M,[g],\Theta)$. 
\end{theorem}

Therefore one is interested in finding solutions to (\ref{nonlinear_ev_equation}) 
which are nowhere zero. It is not clear, whether this is possible. 
An important result by C.\,B\"ar on the zero sets of solutions 
of generalized Dirac equations is the following. 

\begin{theorem}[\cite{baer97}]
Let $(M,g,\Theta)$ be a connected Riemannian spin manifold of dimension $n\geq2$, 
not necessarily compact and possibly with boundary. 
Let $h$ be a smooth endomorphism field of $\Sigma^gM$ and let $\psi$ 
be a nontrivial solution of
\begin{displaymath}
  (D^g+h)\psi=0.
\end{displaymath}
Then the zero set of $\psi$ is a countably $(n-2)$-rectifiable set and 
thus has Hausdorff dimension $n-2$ at most. If $n=2$, then the zero set 
of $\psi$ is a discrete subset of $M$.
\end{theorem}

One can apply this theorem for $n=2$. 
However, if $\psi$ is zero somewhere and $n\geq3$, 
then $|\psi|_g^{2/(n-1)}$ is not smooth on the zero set of $\psi$ 
and thus the theorem cannot be applied. 
In this case one can still prove that the zero set of a nontrivial solution 
does not contain any open subset of $M$ and that its complement 
is connected (see \cite{amm}). 

In this thesis we will not solve this problem. 
However our result on the zero sets of eigenspinors for generic metrics 
suggests that for generic metrics any solution to 
(\ref{nonlinear_ev_equation}) should be nowhere zero.

\section{Witten spinors}

In his proof of the positive energy theorem for spin manifolds Witten 
introduces asymptotically constant harmonic spinors on certain non-compact 
Riemannian manifolds (see \cite{wi}). 
A rigorous proof of his ideas is given in \cite{pt}. Following the latter 
article we will state existence and uniqueness of these so called 
Witten spinors. Then we will see that the zero set of certain 
Witten spinors is not empty. 
For $x\in\R^n$ we denote by $r$ its euclidean distance from $0$. 
We assume that $n\geq3$.
\begin{definition}
For $s\in\R$ we define
\begin{displaymath}
  f\in O''(r^s) :\gdw \nabla^j f\in O(r^{s-j})\textrm{ for }j=0,1,2.
\end{displaymath}
\end{definition}

\begin{definition}
A complete Riemannian manifold $(M,g)$ of dimension $n$ is called 
asymptotically flat, if there exists a compact subset $K\subset M$ 
such that $M\setminus K$ is the disjoint union of a finite 
number of subsets $M_1$,...,$M_d$ with the property that 
for every $i\in\{1,...,d\}$ there exists a contractible and compact 
subset $K_i\subset\R^n$ and a diffeomorphism 
$\Phi_i$: $\R^n\setminus K_i\to M_i$, 
such that in the standard coordinates 
of $\R^n$ for all $j$, $k$ we have:
\begin{displaymath}
  (\Phi_i^*g)_{jk}-\delta_{jk}\in O''(r^{2-n})
\end{displaymath}
as $r\to\infty$. 
The subsets $M_i$ are called the ends of $M$.
\end{definition}

The simplest example of an asymptotically flat Riemannian manifold is 
the Euclidean space $(\R^n,\geucl)$. 
A more interesting example is $M=\R^3\setminus\{0\}$ 
with the Schwarz\-schild metric $g=(1+\frac{1}{r})^4\geucl$, where $m>0$ is a constant. 
It has two asymptotically flat ends, one at infinity and one at zero. 
In order to see the end at zero note that the map $x\mapsto\frac{1}{r^2}x$ 
is an isometry of $(M,g)$ (see \cite{bn}). 
Using this map one obtains a diffeomorphism, which defines the end at zero.

Since the different spin structures over $\R^n\setminus K_i$ are classified by 
the elements of $H^1(\R^n\setminus K_i,\Z/2\Z)=0$ we see that the pullback bundle 
$\Phi_i^*\Sigma^gM$ is equal to the trivial spinor bundle over
$\R^n\setminus K_i$ and thus extends to the trivial bundle 
$\R^n\times\Sigma_n$ over all of $\R^n$. The pullbacks under $\Phi_i^{-1}$ 
of the constant sections of this bundle will be called the constant 
spinors on the end $M_i$. From now we assume $n=3$.

\begin{theorem}[\cite{pt}]
\label{pt_theorem}
Let $(M,g,\Theta)$ be a $3$-dimensional asymptotically flat Riemannian spin manifold. 
For all $i\in\{1,...,d\}$ let $\gamma_i$ be a constant spinor 
on the end $M_i$. Then there exists a unique spinor $\psi$ on $M$ 
such that $D^g\psi=0$ and such that for every $\ep>0$ we have 
$\lim_{r\to\infty} r^{1-\ep}|\psi-\gamma_i|=0$ on each end $M_i$.
\end{theorem}

The spinor $\psi$ is called the Witten spinor for the constant spinors 
$\gamma_1$,...,$\gamma_d$.

Next we describe the relation between nowhere vanishing spinors on 
manifolds of dimension $3$ and orthonormal frames of its tangent bundle. 
It has long been known that every orientable manifold $M$ of dimension $3$ is 
parallelizable, i.\,e.\,there exists a global frame of the tangent bundle, 
which can of course be orthonormalized (see \cite{s}, \cite{w}). 
More recently it has been suggested by people working in general relativity 
(see \cite{n}, \cite{dm}, \cite{fns}) to use Witten spinors in order 
to construct special orthonormal frames of $TM$. This is possible if the 
Witten spinor is nowhere zero on $M$. 
In order to explain this, we first recall that $\Spin(3)=\SU(2)$ and therefore
\begin{displaymath}
  \Spin(3)=\left\{\left.\bmatrix{cc}{ a & b \\ -\overline{b} & \overline{a} }
  \in\GL(2,\C)\right| |a|^2+|b|^2=1\right\}.
\end{displaymath}
It follows that for any $v$, $w\in S^3\subset\C^2$ there exists exactly one 
$q\in\SU(2)$ such that $\rho(q)v=w$. Using this fact we can show:

\begin{theorem}
\label{frame_theorem}
Let $(M,g,\Theta)$ be a Riemannian spin manifold of dimension $3$. 
Then every smooth spinor on $M$ 
which is nowhere zero leads to a global orthonormal frame of $TM$.
\end{theorem}

\begin{proof}
Let $\psi\in\spinor{M}{g}{\infty}$ be nowhere zero on $M$. We define 
$\chi\in\spinor{M}{g}{\infty}$ by $\chi(x):=\frac{\psi(x)}{|\psi(x)|_g}$. 
Let $U\subset M$ be open such that there exists a smooth section $s$ of $P_{\Spin}(M,g)|_U$ 
and a smooth map $\sigma$: $U\to S^3\subset\C^2$ such that on $U$ we have $\chi|_U=[s,\sigma]$. 
For $x\in U$ let $q(x)\in\Spin(3)$ 
be the unique element such that $\rho(q(x))\sigma(x)$ is equal to the first 
vector $E_1$ of the standard basis of $\C^2$. We obtain $q\in C^{\infty}(U,\Spin(3))$ 
such that for all $x\in U$ we have $\chi|_U(x)=[s(x)q(x)^{-1},E_1]$. 
It follows that $s(x)q(x)^{-1}$ is independent of the choices of $s$, $\sigma$ 
and thus can be defined on all of $M$. 
Then $e$: $M\to P_{\SO}(M,g)$ with 
\begin{displaymath}
  e(x):=\Theta(s(x)q(x)^{-1})
\end{displaymath}
is a smooth section of $P_{\SO}(M,g)$. 
\end{proof}

However it is easy to see that Witten spinors can have zeros. 
More precisely if $W$ denotes the space of all Witten spinors on an 
asymptotically flat manifold $M$ with $d$ ends, 
then by Theorem \ref{pt_theorem} there exists a $\C$-linear isomorphism 
$(\Sigma_3)^d\to W$. Let $x\in M$. The evaluation map
\begin{displaymath}
  ev_x:\quad W\to\Sigma^g_xM,\quad \psi\mapsto\psi(x)
\end{displaymath}
is a $\C$-linear map between complex vector spaces of dimensions $2d$ 
and $2$. Thus in the case $d>1$ for every $x\in M$ 
there exists a Witten spinor on $M$, which is zero in $x$. 

In the case $d=1$ it is not clear whether Witten spinors can have zeros. 
If $\psi$ is a Witten spinor, then $W=\span_{\C}\{\psi,J\psi\}$, 
where $J$ is a quaternionic structure. 
Thus if there exists $x\in M$ and 
a Witten spinor $\psi$ with $\psi(x)=0$, then all Witten spinors 
on $(M,g,\Theta)$ are zero in $x$. 

In this thesis we will not solve this problem. 
However one can hope to carry over some of our techniques 
in order to prove that in the generic case Witten 
spinors are nowhere zero. This would mean that the zero sets 
of Witten spinors have no physical significance.

\section{Zero sets of eigenspinors}

Let $(M,g,\Theta)$ be a closed Riemannian spin manifold of dimension $n\geq2$. 
We observe that eigenspinors of the Dirac operator can have non empty zero set. 
Namely there are explicit computations of eigenvalues whose multiplicities are greater 
than the rank of the spinor bundle (see e.\,g.\,Theorem \ref{theorem_dirac_sphere}). 
Thus for these eigenvalues there exist eigenspinors with zeros. 
Harmonic spinors on Riemann surfaces provide another example, 
which we will treat in Section \ref{surfaces_section}. 

It would be useful to have eigenspinors, which are nowhere zero, since then 
one would have a simple proof of Hijazi's inequality, which we now explain. 
First let 
\begin{displaymath}
  \Delta^g: C^{\infty}(M,\R)\to C^{\infty}(M,\R)
\end{displaymath}
denote the Laplace operator for the metric $g$ acting on functions on $M$. 
If $x_1$,...,$x_n$ are local coordinates on $M$ with coordinate vector fields 
$\partial_1$,...,$\partial_n$ and the matrices $(g_{ij})$ and $(g^{ij})$ 
are defined by
\begin{displaymath}
  g_{ij}:=g(\partial_i,\partial_j),\quad g:=(g_{ij}),\quad (g^{ij}):=g^{-1},
\end{displaymath}
then we have for all locally defined smooth functions $f$:
\begin{equation}
  \label{laplace}
  \Delta^g f=-\det(g)^{-1/2}\sum_{i,j=1}^n \partial_i(g^{ij}\det(g)^{1/2}\partial_j f).
\end{equation}
The conformal Laplace operator is a linear second order differential operator 
\begin{displaymath}
  L^g:\quad C^{\infty}(M,\R)\to C^{\infty}(M,\R),
\end{displaymath}
which is defined by 
\begin{displaymath}
  L^g f:=\frac{4(n-1)}{n-2}\Delta^g f+\scal^g f,
\end{displaymath}
where $\scal^g$ denotes the scalar curvature for the metric $g$. 
The operator $L^g$ is called the conformal Laplace operator because 
it satisfies a nice transformation law under conformal changes of the metric. 
More precisely let $g$ and $h=e^{2u}g$ be two conformally related metrics, 
where $u$ is a smooth function on $M$. Then we have for all $f\in C^{\infty}(M,\R)$ 
(see \cite{lp}, p. 43)
\begin{equation}
  \label{Lg_conform}
  L^g(e^{(n-2)u/2}f)=e^{(n+2)u/2}L^hf.
\end{equation}
The spectrum of $L^g$ consists of a sequence of isolated real eigenvalues, 
which is bounded from below by the smallest eigenvalue $\mu$ and is not bounded from above. 
\begin{displaymath}
  \mu:=\mu_0\leq\mu_1\leq...\to\infty.
\end{displaymath}
Hijazi's inequality gives a lower bound on the eigenvalues of 
the Dirac operator in terms of the smallest eigenvalue 
of the conformal Laplace operator. 

\begin{theorem}[\cite{hij86}]
\label{hij_ineq}
Let $(M,g,\Theta)$ be a closed Riemannian spin manifold of dimension $n\geq3$ 
and let $\mu$ be the smallest eigenvalue of the conformal Laplace operator. 
Then any eigenvalue $\lambda$ of the Dirac operator satisfies
\begin{displaymath}
  \lambda^2\geq\frac{n}{4(n-1)}\mu.
\end{displaymath}
\end{theorem}

For the new simple proof we need the following assumption 
\begin{displaymath}
  \textrm{(A)   }
  \begin{array}{l}
  \textrm{There exists an eigenspinor }\psi
  \textrm{ corresponding to }\lambda,\\
  \textrm{which is nowhere zero on }M.
  \end{array}
\end{displaymath}

\begin{proof}[Proof of Theorem \ref{hij_ineq} using Assumption (A)]
We define a smooth function $u$ on $M$ by $u:=\frac{2}{n-1}\ln|\psi|_g$. 
The metric $h:=e^{2u}g$ is conformal to $g$ and the spinor 
\begin{displaymath}
  \varphi:=e^{-(n-1)u/2}\beta_{g,h}\psi\in\spinor{M}{h}{\infty}
\end{displaymath}
satisfies $|\varphi|_h\equiv1$ on $M$ 
and $D^h\varphi=e^{-u}\lambda\varphi$ by (\ref{dirac_conform}). 
By the Schr\"odinger-Lichnerowicz formula (\ref{schroed_lichn}) we find 
\begin{displaymath}
  \int_M |D^h\varphi|_h^2\,\dv^h=\int_M |\nabla^h\varphi|_h^2\,\dv^h+\frac{1}{4}\int_M \scal^h\,\dv^h.
\end{displaymath}
We use the estimate 
\begin{displaymath}
  \int_M |\nabla^h\varphi|_h^2\,\dv^h\geq\frac{1}{n}\int_M |D^h\varphi|_h^2\,\dv^h.
\end{displaymath}
Since the volume elements are related by $\dv^h=e^{nu}\dv^g$ we find 
\begin{displaymath}
  \int_M |D^h\varphi|_h^2\,\dv^h=\lambda^2\int_M e^{(n-2)u}\dv^g
\end{displaymath}
and thus altogether 
\begin{displaymath}
  \int_M \scal^h\,\dv^h \leq \frac{4(n-1)}{n}\lambda^2 \int_M e^{(n-2)u}\dv^g.
\end{displaymath}
By the transformation formula (\ref{Lg_conform}) we find 
\begin{displaymath}
  \scal^h=L^h1=e^{-(n+2)u/2}L^g(e^{(n-2)u/2})
\end{displaymath}
and thus by setting $f:=e^{(n-2)u/2}$ we obtain 
\begin{displaymath}
  \int_M f L^gf\dv^g\leq\frac{4(n-1)}{n}\lambda^2 \int_M f^2\dv^g.
\end{displaymath}
Since $L^g$ is an elliptic self-adjoint differential operator on a compact manifold, 
there exists an $L^2$-orthonormal Hilbert basis of $L^2(M,\R)$ 
consisting of eigenfunctions of $L^g$ (see e.\,g.\,\cite{lm}, p. 196). 
We write $f$ as a possibly infinite linear combination of these eigenfunctions 
and conclude that the left hand side is bounded from below by $\mu\int_M f^2\dv^g$.  
\end{proof}

The eigenvalues of $D^g$ depend continuously on $g$ 
with respect to the $C^k$-topology for all $k\geq1$ (see Proposition \ref{cont_eigenvalues}). 
For $k\in\N$ the $k$-th eigenvalue $\mu_k(g)$ of the conformal Laplacian $L^g$ 
can be characterized as follows. For every $(k+1)$-dimensional subspace $V_{k+1}$ 
of $C^{\infty}(M,\R)$ we define
\begin{displaymath}
  \Lambda_g(V_{k+1}):=\sup\{ (f,L^gf)_2/\|f\|^2_2\,\,|f\in V_{k+1}\setminus\{0\}\}.
\end{displaymath}
Then we have
\begin{displaymath}
  \mu_k(g):=\inf_{V_{k+1}} \Lambda_g(V_{k+1}),
\end{displaymath}
where the infimum is taken over all $(k+1)$-dimensional subspaces. 
This can be proven similarly to Proposition 2.1 in \cite{bu}. 
Since the scalar curvature contains second derivatives of the metric, 
it follows that $\mu_k(g)$ depends continuously on $g$ 
with respect to the $C^2$-topology. 

We will show in Section \ref{gen_eigenspinor_section} that for $n\in\{2,3\}$ the subset of all 
Riemannian metrics on $M$, for which Assumption (A) holds, is dense in the 
space of all smooth Riemannian metrics on $M$ with respect to any $C^k$-topology, $k\geq1$.
By continuity of the eigenvalues of $D^g$ and $L^g$ 
we obtain a new proof of Theorem \ref{hij_ineq} for $n=3$.

\chapter{Examples of zero sets of eigenspinors}
\label{examples_chapter}

\section{Harmonic spinors on Riemann surfaces}
\label{surfaces_section}

First we recall a way of identifying positive harmonic spinors on closed Riemann 
surfaces with holomorphic sections of certain line bundles (see \cite{hit74}). 
Let $(M,g)$ be a compact K\"ahler manifold of complex dimension $m$. 
We denote by $K:=\Lambda^m T^{1,0}M$ the canonical line bundle over $M$. 
Then $M$ is a spin manifold if and only if there exists a holomorphic line bundle 
$L$ such that $L\otimes_{\R} L=K$. In this case the different spin structures 
on $M$ are in one-to-one-correspondence with such holomorphic line bundles. 
The spinor bundle is isomorphic to 
$\Lambda^*T^{0,1}M\otimes_{\R} L$ and the $\Z/2\Z$-grading is given by the 
decomposition of $\Lambda^*T^{0,1}M$ into even and odd forms. Let
\begin{displaymath}
  \overline{\partial}: \Lambda^pT^{0,1}M\to\Lambda^{p+1}T^{0,1}M
\end{displaymath}
be the exterior derivative in the Dolbeault complex. 
If $s$ is a local holomorphic section of $L$ and $\varphi$ is a 
smooth local section of $\Lambda^*T^{0,1}M$, we define 
$\overline{\partial}(\varphi\otimes s):=\overline{\partial}\varphi\otimes s$. 
Then the Dirac operator corresponds to 
$\sqrt{2}(\overline{\partial}+\overline{\partial}^*)$. 
Now let $M$ be a closed orientable surface with a Riemannian metric $g$. Then $M$ is 
a spin manifold and since $\Lambda^0T^{0,1}M$ is the trivial line bundle, 
there is a canonical isomorphism $\Sigma^+M\cong L$. 
Thus positive harmonic spinors can be identified with holomorphic sections of $L$.

Next let $(M,g,\Theta)$ be a two-dimensional Riemannian spin manifold. 
We recall a way of associating a vector field to a positive or negative spinor 
on $M$ (see \cite{am98}). 
First we define $\tau_{\pm}$: $\SO(2)\to\C$ by
\begin{displaymath}
  \bmatrix{cc}{\cos t & -\sin t \\ \sin t & \cos t}\mapsto\exp(\pm it).
\end{displaymath}
We define a complex structure $J$ on $M$ such that for every $p\in M$ and for every unit vector  
$X\in T_pM$ the system $(X,JX)$ is a positively oriented orthonormal basis of $T_pM$. 
Then the map $\P_{\SO}(M,g)\times_{\tau_+}\C\to(TM,J)$ which sends $[(e_1,e_2),1]$ to $e_1$ 
is an isomorphism of complex line bundles. In the same way we obtain an isomorphism 
$\P_{\SO}(M,g)\times_{\tau_-}\C\to(TM,-J)$. 
We quote the following Lemma from \cite{am98} including the proof. 
It enables us to associate a vector field on $M$ 
to a given section of $\Sigma^{\pm}M$.

\begin{lemma}
\label{spinor_to_vector}
Let $\Theta$ be a spin structure on $(M,g)$. Then the map
\begin{eqnarray*}
\Phi_{\pm}:\quad\Sigma^{\pm}M=\P_{\Spin}(M,g)\times_{\rho}\Sigma^{\pm}_2&\to&\P_{\SO}(M,g)\times_{\tau_{\mp}}\C=(TM,\mp J)\\
{}[s,\sigma]&\mapsto&[\Theta(s),\sigma^2]
\end{eqnarray*}
is well defined.
\end{lemma}

\begin{proof}
For any $n$ the multiplicative group $\Cl(n)^*$ of invertible 
elements in the Clifford algebra $\Cl(n)$ is an open subset. 
Thus one may identify the Lie algebra of the spin group $\Spin(n)$ 
with a subspace of $\Cl(n)$. By \cite{lm}, Proposition I.6.1 the 
Lie algebra of $\Spin(2)$ is spanned by $E_1\cdot E_2$. Therefore 
every element of $\Spin(2)$ can be written in the form $\exp(tE_1\cdot E_2)$. 
Since $E_1\cdot E_2$ acts on $\Sigma^{\pm}_2$ as $\mp i\id$, 
we find that $\exp(tE_1\cdot E_2)$ acts as $\exp(\mp it)$. 
Furthermore by \cite{lm}, Proposition I.6.2 we have
\begin{displaymath}
 \vartheta_*(E_1\cdot E_2)=\bmatrix{cc}{0 & -2 \\ 2 & 0}=2J.
\end{displaymath}
Thus for every $g\in\Spin(n)$ the element $g^2$ acts 
on $\Sigma^{\pm}_2$ as $\tau_{\mp}(\vartheta(g))$. 
We conclude that
\begin{eqnarray*}
[\Theta(sg),\sigma^2]&=&[\Theta(s)\vartheta(g),\sigma^2]\\
&=&[\Theta(s),\tau_{\mp}(\vartheta(g))\cdot\sigma^2]\\
&=&[\Theta(s),(g\cdot\sigma)^2]
\end{eqnarray*}
and therefore $\Phi_{\pm}$ is well defined.
\end{proof}

From now assume that $M$ is a closed oriented surface of genus $\gamma$ 
with a Riemannian metric $g$. 
Assume that $\psi$ is a nontrivial positive harmonic spinor on $M$ and that 
$p\in M$ is a point with $\psi(p)=0$. Since $\Sigma^+M$ can be 
canonically identified with a holomorphic line bundle, there exists 
a local chart of $M$ mapping $p$ to $0\in\C$ 
and a local trivialization of $\Sigma^+M$ 
around $p$, such that in this chart and in this trivialization $\psi$ corresponds 
to a holomorphic function $f$: $U\to\C$ on an open 
subset $U\subset\C$ containing $0$. 
Then in a neighborhood of $0$ the function $f$ 
is given by a convergent power series $f(z)=\sum_{k=1}^{\infty} a_k z^k$ 
with $a_k\in\C$. Define $m_p$ as the smallest integer $k$ 
such that $a_k\neq0$. Let $X$ be the vector field on $M$ associated 
to $\psi$ via Lemma \ref{spinor_to_vector}. It follows that $X$ 
has an isolated zero at $p$ with index equal to $-2m_p$. 
Let $\chi(M)=2-2\gamma$ denote the Euler characteristic of $M$. 
Denote by $N$ the zero set of $\psi$. Since $M$ is compact, 
the set $N$ is finite. By the Poincar\'e-Hopf Theorem 
(see e.\,g.\,\cite{mi}, p. 35) we obtain 
$-2\sum_{p\in N}m_p=\chi(M)$. Thus we find the following result.

\begin{theorem}
\label{theorem_surfaces}
Assume that $\psi$ is a positive harmonic spinor on a 
closed oriented surface $M$ and let $N\subset M$ 
be its zero set. Then $N$ is finite and we have
\begin{displaymath}
  \sum_{p\in N}m_p=-\frac{1}{2}\chi(M)=\gamma-1.
\end{displaymath}
\end{theorem}

From this we can see the well known fact, that for $\gamma=0$ 
there are no harmonic spinors for any Riemannian metric (see e.\,g.\,\cite{hit74}). 
In the case $\gamma=1$ it has been proven in \cite{fr84} that for the flat 
Riemannian metric $g$ there exists exactly one spin structure with 
$\dim\ker(D^g)=2$ and exactly three spin structures with no 
harmonic spinors. Since the dimension of the space of harmonic spinors 
is conformally invariant and since any two Riemannian metrics on $M$ 
are conformally related, it follows that this result holds for all 
Riemannian metrics on $M$. If $\psi$ is a positive harmonic spinor 
and $J$ is a parallel quaternionic structure on the spinor bundle, 
then $J\psi$ is a negative harmonic spinor and it has the same zero 
set as $\psi$. By Theorem \ref{theorem_surfaces} 
all the harmonic spinors are nowhere zero on $M$. 
For $\gamma=2$ it follows from \cite{hit74}, Proposition 2.3, that 
there exist 6 distinct spin structures on $M$, such that for every Riemannian 
metric $g$ on $M$ we have $\dim\ker(D^g)=2$. 
We take one of these spin structures and fix a Riemannian metric $g$ on $M$. 
Using again a parallel quaternionic structure and using Theorem \ref{theorem_surfaces} 
we find that the zero set of every harmonic spinor consists of exactly 
one point and it is the same point for all harmonic spinors.

\section{The sphere with the standard metric}
\label{sphere_section}

We will explicitly calculate some zero sets of eigenspinors on the 
sphere with the standard metric. 
The reader, who is only interested in the main results may skip this 
section and continue in Chapter \ref{green_chapter}.

\begin{theorem}[\cite{baer96}]
\label{theorem_dirac_sphere}
Let $M=S^n$, $n\geq2$, be equipped with the standard metric and the unique spin structure. 
The eigenvalues of the Dirac operator are $\lambda=\pm(\frac{n}{2}+m)$, $m\in\N$, 
with multiplicities
\begin{displaymath}
  2^{[n/2]}\bmatrix{c}{ m+n-1 \\ m}.
\end{displaymath}
\end{theorem}

We will describe the method used in \cite{baer96} for the computation 
of the spectrum. In this way we will obtain the eigenspinors and their 
zero sets explicitly.

Let $\Delta^g$ be the Laplace operator acting on smooth functions on $S^n$ 
defined as in (\ref{laplace}). 
It is a classical result that the eigenvalues of $\Delta^g$ are given 
by $k(n+k-1)$, $k=0,1,2,...$ with multiplicities
\begin{displaymath}
  \frac{n+2k-1}{n+k-1}\left(\begin{array}{c}
  n+k-1 \\ k
  \end{array}\right).
\end{displaymath}
The eigenfunctions for $k(n+k-1)$ are exactly 
the restrictions to $S^n$ of polynomial functions defined on $\R^{n+1}$ which 
are homogeneous of degree $k$ and harmonic with respect to the Laplace operator 
on $\R^{n+1}$ (see \cite{bgm}).

If $n$ is even, let $\alpha_j$, $1\leq j\leq N$, be constant and 
pointwise orthonormal sections of $\Sigma\R^{n+1}$. 
If $n$ is odd, then we have a decomposition
\begin{displaymath}
  \Sigma\R^{n+1}=\Sigma^+\R^{n+1}\oplus\Sigma^-\R^{n+1},
\end{displaymath}
and we define $\alpha_j$, $1\leq j\leq N$ to be constant and pointwise orthonormal 
sections of $\Sigma^+\R^{n+1}$. In each case we define spinors $\beta_j$ 
on $\R^{n+1}$ by $x\mapsto x\cdot\alpha_j$, $1\leq j\leq N$. 
One finds that the restrictions of the $\alpha_j$ to the 
hypersurface $(S^n,g)$ are Killing spinors on $S^n$ for the value $1/2$ 
and that the restrictions of the $\beta_j$ to $S^n$ are 
Killing spinors for the value $-\frac{1}{2}$. 
We denote the restrictions of the $\alpha_j$ by $\varphi_{1,j}$ 
and the restrictions of the $\beta_j$ by $\varphi_{-1,j}$. 
Using this one can show that for every eigenfunction $f_k$ of $\Delta^g$ 
corresponding to $k(n+k-1)$, for all $j\in\{1,...,N\}$ and 
for all $\mu\in\{\pm1\}$ the equation
\begin{displaymath}
  \big(D^g+\frac{\mu}{2}\big)^2 (f_k\varphi_{\mu,j})=\big(k+\frac{n-1}{2}\big)^2f_k\varphi_{\mu,j}.
\end{displaymath}
holds. Therefore if $\{f_i\}_{i\in\N}$ is an 
orthonormal Hilbert basis of $L^2(S^n,\R)$ consisting of eigenfunctions of $\Delta^g$ 
on $S^n$ then the $f_i\varphi_{\mu,j}$ form an orthonormal Hilbert basis of 
$\Lspinor{S^n}{g}{2}$ consisting of eigenspinors of $(D^g+\frac{\mu}{2})^2$. 
In general, if $A$ is an endomorphism of a complex vector space and 
$\nu$ is a nonzero complex number, then we have an isomorphism
\begin{displaymath}
  \ker(A^2-\nu^2)\to\ker(A-\nu)\oplus\ker(A+\nu),
  \quad u\mapsto(Au+\nu u)\oplus(Au-\nu u).
\end{displaymath}
Therefore the eigenvectors of $A$ corresponding to $\pm\nu$ are exactly the 
nonzero vectors $Au\pm\nu u$, where $u\in\ker(A^2-\nu^2)$. 
In this way the eigenvalues of $D^g$ and their multiplicities are obtained 
in \cite{baer96}.

We put $A=D^g+\frac{\mu}{2}$, $\lambda=k+\frac{n-1}{2}$ and compute
\begin{eqnarray*}
  \psi^{\varepsilon,\mu}_{k,j}&:=&\big(D^g+\frac{\mu}{2}\big)(f_k\varphi_{\mu,j})+\varepsilon(k+\frac{n-1}{2})f_k\varphi_{\mu,j}\\
              &=&\big(\frac{\mu(1-n)}{2}+\varepsilon\big(k+\frac{n-1}{2}\big)\big)f_k\varphi_{\mu,j}+\grad^{S^n}(f_k)\cdot\varphi_{\mu,j},
\end{eqnarray*}
where $\varepsilon\in\{\pm1\}$. 
If $\psi^{\varepsilon,\mu}_{k,j}$ is not identically zero, then it is an eigenspinor 
of $D^g$ corresponding to the value $\varepsilon(k+\frac{n-1}{2})-\frac{\mu}{2}$. We find
\begin{eqnarray*}
  \psi^{--}_{k,j}&=&-kf_k\varphi_{-1,j}+\grad^{S^n}(f_k)\cdot\varphi_{-1,j},\quad\lambda=-\frac{n}{2}-k+1,\\
  \psi^{+-}_{k,j}&=&(k+n-1)f_k\varphi_{-1,j}+\grad^{S^n}(f_k)\cdot\varphi_{-1,j},\quad\lambda=\frac{n}{2}+k,\\
  \psi^{-+}_{k,j}&=&(-k+1-n)f_k\varphi_{1,j}+\grad^{S^n}(f_k)\cdot\varphi_{1,j},\quad\lambda=-\frac{n}{2}-k,\\
  \psi^{++}_{k,j}&=&kf_k\varphi_{1,j}+\grad^{S^n}(f_k)\cdot\varphi_{1,j},\quad\lambda=\frac{n}{2}+k-1.
\end{eqnarray*}
Therefore all the $\psi^{\varepsilon,\mu}_{k,j}$ except 
$\psi^{--}_{0,j}$ and $\psi^{++}_{0,j}$ are not identically zero. 

In order to investigate the zero sets of the eigenspinors we first note 
that the spinors $\psi^{+-}_{0,j}$, 
$1\leq j\leq 2^{[n/2]}$, are eigenspinors for $\lambda=\frac{n}{2}$ 
and are non-zero multiples of $\varphi_{-1,j}$. 
Thus any nontrivial linear combination of them 
is nowhere zero on $S^n$. Similarly the spinors 
$\psi^{-+}_{0,j}$, $1\leq j\leq 2^{[n/2]}$, 
are eigenspinors for $\lambda=-\frac{n}{2}$ 
and are non-zero multiples of $\varphi_{1,j}$. 
Thus any nontrivial linear combination of them 
is nowhere zero on $S^n$. 
In the case $k\geq1$ we determine the zero set of $\psi^{\varepsilon,\mu}_{k,j}$ 
as follows: We have
\begin{displaymath}
  \psi^{\varepsilon,\mu}_{k,j}(x)=0
  \Longleftrightarrow\grad^{S^n}(f_k)(x)=0\textrm{ and }f_k(x)=0.
\end{displaymath}
Since $f_k$ is homogeneous of degree $k$, we obtain for the 
normal component of $\grad^{\R^{n+1}}(f_k)$ in $x\in S^n$ that
\begin{displaymath}
  \geucl(\grad^{\R^{n+1}}(f_k)(x),x)=kf_k(x).
\end{displaymath}
If $k\geq1$, it follows that
\begin{displaymath}
  \psi^{\varepsilon,\mu}_{k,j}(x)=0
  \Longleftrightarrow\grad^{\R^{n+1}}(f_k)(x)=0.
\end{displaymath}

For example one finds eigenspinors which vanish to high order at some point, 
i.\,e.\,the spinor and all of its derivatives up to a finite order are zero at this point. 
Namely given $m\in\N\setminus\{0\}$ we take
\begin{displaymath}
  f_{m+1}(x_1,...,x_{n+1}):=\Re((x_1+ix_2)^{m+1}).
\end{displaymath}
Then the spinor $\psi^{--}_{m+1,1}$ on $S^n$ defined by
\begin{displaymath}
  \psi^{--}_{m+1,1}:=-(m+1)f_{m+1}\varphi_{-1,1}+\grad^{S^n}(f_{m+1})\cdot\varphi_{-1,1}
\end{displaymath}
is an eigenspinor for the value $-\frac{n}{2}-m$ and with $p=e_{n+1}\in S^n$ 
we have $\nabla^r\psi^{--}_{m+1,1}(p)=0$ for $0\leq r\leq m-1$.

We also find spinors whose zero set is not a submanifold of $S^n$.  
As an example consider $n=3$ and take $f_3(x_1,...,x_4):=x_1x_2x_3$. Then we have 
\begin{displaymath}
  \grad^{\R^4}(f_3)=(x_2x_3,x_1x_3,x_1x_2,0).
\end{displaymath}
Therefore $\psi^{--}_{3,1}$ has the zero set
\begin{displaymath}
  \{x\in S^3|x_1=x_2=0\textrm{ or }x_1=x_3=0\textrm{ or }x_2=x_3=0\}
\end{displaymath}
which consists of three copies of $S^1$ intersecting in $\pm e_4$.

\section{Flat tori}
\label{tori_section}

We will give some explicit computations of zero sets 
of eigenspinors on flat tori. 
The reader, who is only interested in the main results, may skip this section. 

Let $\Gamma\subset\R^n$ and $M=\R^n/\Gamma$, $n\geq2$ and let $g$ 
be the metric on $M$ induced by the Euclidean metric on $\R^n$. 
The spin structures on $(M,g)$ are classified by 
\begin{displaymath}
  H^1(M,\Z/2\Z)\cong\Hom(\Gamma,\Z/2\Z)
\end{displaymath}
and thus there are $2^n$ distinct spin structures. 
The calculation of the eigenvalues of the Dirac operator on 
$(M,g)$ for all spin structures has been carried out in \cite{fr84}. 
It is also explained e.\,g.\,in \cite{g}.
Denote by $\Gamma^*$ the dual lattice. 
Let $(\gamma_j)_{j=1}^n$ be a basis of $\Gamma$ 
and let $(\gamma_j^*)_{j=1}^n$ be the dual basis. 
The spin structures will be denoted by the tuples 
$(\delta_1,...,\delta_n)$, where $\delta_j\in\{0,1\}$ 
is the image of $\gamma_j$ under a homomorphism 
from $\Gamma$ to $\Z/2\Z$, $j\in\{1,...,n\}$. 
For a fixed spin structure $(\delta_1,...,\delta_n)$ 
we define the one-form $\delta$ on $\R^n$ by 
$\delta=\frac{1}{2}\sum_{j=1}^n \delta_j \gamma_j^*$.
Then the result is the following:

\begin{theorem}
The eigenvalues of the Dirac operator on $(\R^n/\Gamma,g)$ 
with the spin structure $(\delta_1,...,\delta_n)$ are 
\begin{displaymath}
  \{\pm2\pi|\alpha+\delta|\,|\alpha\in\Gamma^*\}.
\end{displaymath}
In case $\delta_1=...=\delta_n=0$ the multiplicity of 
the eigenvalue $0$ is $2^{[n/2]}$. For every other 
spin structure $0$ is not an eigenvalue. 
For the spin structure $(\delta_1,...,\delta_n)$ the 
multiplicity of every non-zero eigenvalue $\lambda$ is
\begin{displaymath}
  2^{[n/2]-1}\#\{\alpha\in\Gamma^*|\,2\pi|\alpha+\delta|=|\lambda|\}.
\end{displaymath}
\end{theorem}

In order to obtain the eigenspinors explicitly we summarize the proof. 
It is well known that the spinors on $M$ equipped with the 
spin structure $(\delta_1,...,\delta_n)$ are in one-to-one-cor\-res\-pon\-dence with 
the spinors $\psi$ on $\R^n$ 
which satisfy the condition 
\begin{equation}
  \label{equiv_cond}
  \psi(x+\gamma_j)=(-1)^{\delta_j}\psi(x)
\end{equation}
for all $x\in\R^n$ and all $j\in\{1,...,n\}$ (see e.\,g.\,\cite{g}). 
We fix a spin structure $(\delta_1,...,\delta_n)$ on $M$.  
If $\delta=0$, we see that every constant spinor 
on $\R^n$ satisfies the equivariance condition (\ref{equiv_cond}) 
and thus is a harmonic spinor on $(M,g,\Theta)$. 
Therefore we have $\dim\ker(D^g)=2^{[n/2]}$. 
Let $\alpha\in\Gamma^*$ and assume that $\alpha\neq0$ 
or that $\delta\neq0$. Then $\alpha+\delta\neq0$ 
and the action of the one-form 
$i\frac{\alpha+\delta}{|\alpha+\delta|}$ 
on $\Sigma\R^n$ is parallel, unitary and an involution. 
Thus one obtains a parallel and orthogonal splitting of 
the spinor bundle
\begin{displaymath}
  \Sigma\R^n=\Sigma^+_{\alpha}\R^n\oplus\Sigma^-_{\alpha}\R^n,
\end{displaymath}
where $\Sigma^{\pm}_{\alpha}\R^n$ are the eigenbundles
corresponding to the eigenvalue $\pm1$. They are isomorphic, 
since the Clifford multiplications with two orthogonal vectors 
anticommute. Note that
\begin{displaymath}
  \Sigma^{\pm}_{-\alpha-2\delta}\R^n=\Sigma^{\mp}_{\alpha}\R^n.
\end{displaymath}
For every $\mu\in\{\pm1\}$ we choose 
sections $\varphi^{\mu}_{\alpha,j}$ of $\Sigma^{\mu}_{\alpha}\R^n$, 
$1\leq j\leq2^{[n/2]-1}$, which are constant and pointwise 
orthonormal. 
Then for all $\mu\in\{\pm1\}$ the spinor 
$\psi_{\alpha,j}^{\mu}$ defined by 
\begin{displaymath}
  \psi_{\alpha,j}^{\mu}(x):=e^{2\pi i (\alpha+\delta)(x)}\varphi_{\alpha,j}^{\mu}
\end{displaymath}
satisfies the equivariance condition (\ref{equiv_cond}) and is an eigenspinor of $D^g$ 
for the value $2\pi \mu|\alpha+\delta|$. 
Since the functions $\{e^{2\pi i\alpha}|\alpha\in\Gamma^*\}$ 
form a Hilbert basis of $L^2(M,\C)$, it follows that 
the spinors $\psi_{\alpha,j}^{\mu}$ 
form a Hilbert basis of $\Lspinor{M}{g}{2}$. 

Concerning the zero set of eigenspinors consider first 
the so called trivial spin structure $\delta=0$. 
It is the only spin structure which admits harmonic 
spinors. These are parallel spinors and they are nowhere zero. 
Next consider an arbitrary spin structure. 
In order to investigate the zero set of  
eigenspinors to nonzero eigenvalues first note that for 
any nonzero $\alpha\in\Gamma^*$ the spinors
\begin{displaymath}
  \psi_{\alpha,1}^{+},...,\psi_{\alpha,2^{[n/2]-1}}^{+},\psi_{-\alpha-2\delta,1}^{+},...,\psi_{-\alpha-2\delta,2^{[n/2]-1}}^{+}
\end{displaymath}
are eigenspinors corresponding to $2\pi|\alpha+\delta|$ 
and they are pointwise linearly independent. Thus any non trivial 
linear combination of them is nowhere zero. 
These spinors are contributed by the elements $\alpha$, $-\alpha-2\delta$ of $\Gamma^*$. 
Assume now that there is $\alpha'\notin\{\alpha,-\alpha-2\delta\}$ 
with $|\alpha'+\delta|=|\alpha+\delta|$. Let $\varphi$ be a 
parallel spinor on $\R^n$ such that $i\frac{\alpha'+\delta}{|\alpha'+\delta|}\cdot\varphi=\varphi$ 
and write $\varphi=\varphi^+ +\varphi^-$ where 
$i\frac{\alpha+\delta}{|\alpha+\delta|}\cdot\varphi^{\pm}=\pm\varphi^{\pm}$. 
Then the spinor $\psi$ defined by 
\begin{eqnarray*}
  \psi(x)&=&e^{2\pi i(\alpha+\delta)(x)}\varphi^{+}+e^{-2\pi i(\alpha+\delta)(x)}\varphi^{-}
  -e^{2\pi i(\alpha'+\delta)(x)}\varphi\\
  &=&\big(e^{2\pi i(\alpha+\delta)(x)}-e^{2\pi i(\alpha'+\delta)(x)}\big)\varphi^{+}
  +\big(e^{-2\pi i(\alpha+\delta)(x)}-e^{2\pi i(\alpha'+\delta)(x)}\big)\varphi^{-}
\end{eqnarray*}
is an eigenspinor corresponding to the value $2\pi|\alpha+\delta|$ 
and has the zero set 
\begin{displaymath}
  \big\{[x]\in M| (\alpha+\delta)(x),(\alpha'+\delta)(x)\in\Z\textrm{ or }(\alpha+\delta)(x),(\alpha'+\delta)(x)\in\frac{1}{2}+\Z\big\}.
\end{displaymath}
This is a submanifold of codimension $2$.

\chapter{Green's function for the Dirac operator}
\label{green_chapter}

\section{Trivialization of the spinor bundle}
\label{bourg-triv_section}

Let $(M,g,\Theta)$ be a Riemannian spin manifold of dimension $n$. 
In this section we explain a certain local trivialization of the 
spinor bundle $\Sigma^gM$, which we will use later. 
In the literature (e.\,g.\,\cite{aghm}, \cite{ah}) it is known as the 
Bourguignon-Gauduchon trivialization. 

Let $p\in M$ and let $W\subset T_pM$ be an open neighborhood of $0$ such that 
the restriction of the exponential map $\exp_p$ to $W$ is a diffeomorphism. 
Let $I$: $(\R^n,\geucl)\to(T_pM,g)$ be a linear isometry, let $U:=\exp_p(W)\subset M$ 
and let $V:=I^{-1}(W)$. Then 
\begin{displaymath}
  \rho:=\exp_p\circ I|_V:\quad V\to U
\end{displaymath}
defines a local parametrization of $M$ by Riemannian normal coordinates. 
For $x\in V$ we use the canonical isomorphism $T_xV\cong\R^n$ and in this way we obtain 
the Euclidean scalar product $\geucl$ on $T_xV$.
In addition we have the scalar product $\rho^*g$ on $T_xV$ defined by
\begin{displaymath}
  (\rho^*g)(v,w)=g(d\rho|_x v,d\rho|_x w),
\end{displaymath}
where $v$, $w\in T_xV$. 
There exists $G_x\in\End(T_xV)$, such that 
for all vectors $v$, $w\in T_xV$ we have
\begin{displaymath}
  (\rho^*g)(v,w)=\geucl(G_x v,w)
\end{displaymath}
and $G_x$ is $\geucl$-self-adjoint and positive definite. 
There is a unique positive definite endomorphism 
$B_x\in\End(T_xV)$ such that we have $B_x^2=G_x^{-1}$. If $(E_i)_{i=1}^n$ is any 
$\geucl$-orthonormal basis of $T_xV$, then $(B_x E_i)_{i=1}^n$, 
is a $\rho^*g$-orthonormal basis of $T_xV$. 
Therefore the vectors $d\rho|_x B_x E_i$, $1\leq i\leq n$, 
form a $g$-orthonormal basis of $T_{\rho(x)}M$. 
We assemble the maps $B_x$ to obtain a vector bundle endomorphism $B$ of $TV$ and we define
\begin{displaymath}
  b:\quad TV\to TM|_U,\quad b=d\rho\circ B.
\end{displaymath}
From this we obtain an isomorphism of principal $\SO(n)$-bundles
\begin{displaymath}
  P_{\SO}(V,\geucl)\to P_{\SO}(U,g),\quad (E_i)_{i=1}^n\mapsto(b(E_i))_{i=1}^n,
\end{displaymath}
which lifts to an isomorphism
\begin{displaymath}
  c:\quad P_{\Spin}(V,\geucl)\to P_{\Spin}(U,g)
\end{displaymath}
of principal $\Spin(n)$-bundles. We define 
\begin{displaymath}
  \beta:\quad \Sigma\R^n|_V\to\Sigma^gM|_U,\quad [s,\sigma]\mapsto[c(s),\sigma].
\end{displaymath}
This gives an identification of the spinor bundles, 
which is a fibrewise isometry with respect to the bundle metrics on $\Sigma\R^n|_V$ 
and on $\Sigma^gM|_U$. Furthermore for all $X\in TV$ and for all $\varphi\in\Sigma\R^n|_V$ 
we have $\beta(X\cdot\varphi)=b(X)\cdot\beta(\varphi)$. 
We obtain an isomorphism
\begin{displaymath}
  A:\quad \spinor{M|_U}{g}{\infty}\to\spinor{\R^n|_V}{}{\infty},\quad
  \psi\mapsto \beta^{-1}\circ\psi\circ\rho,
\end{displaymath}
which sends a spinor on $U$ to the corresponding spinor in the trivialization. 
Let $\nabla^g$ and $\nabla$ denote the Levi Civita connections on $(U,g)$ resp. on $(V,\geucl)$ 
as well as its lifts to $\Sigma^gM|_U$ and $\Sigma\R^n|_V$. Let $(e_i)_{i=1}^n$ 
be a positively oriented orthonormal frame of $TM|_U$. 
Then for all $\psi\in\spinor{M|_U}{g}{\infty}$ we have by the formula (\ref{connection_local})
\begin{eqnarray*}
  \nabla^g_{e_i}\psi&=&\partial_{e_i}\psi+\frac{1}{4}\,\sum_{j,k=1}^n\widetilde{\Gamma}^k_{ij} e_j\cdot e_k\cdot\psi,
\end{eqnarray*}
where
\begin{displaymath}
  \widetilde{\Gamma}^k_{ij}:=g(\nabla^g_{e_i}e_j,e_k).
\end{displaymath}
In particular we can take the standard basis $(E_i)_{i=1}^n$ of $\R^n$ and put $e_i:=b(E_i)$, $1\leq i\leq n$. 
We define the matrix coefficients $B^j_i$ by $B(E_i)=\sum_{j=1}^n B^j_i E_j$. 
It follows that
\begin{eqnarray}
  \nonumber
  A\nabla^g_{e_i}\psi&=&\nabla_{d\rho^{-1}(e_i)}A\psi+\frac{1}{4}\,\sum_{j,k=1}^n\widetilde{\Gamma}^k_{ij} E_j\cdot E_k\cdot A\psi\\
  \nonumber
  &=&\nabla_{B(E_i)}A\psi+\frac{1}{4}\,\sum_{j,k=1}^n\widetilde{\Gamma}^k_{ij} E_j\cdot E_k\cdot A\psi\\
  \label{nabla_triv}
  &=&\nabla_{E_i}A\psi+\sum_{j=1}^n(B^j_i-\delta^j_i)\nabla_{E_j}A\psi+\frac{1}{4}\,\sum_{j,k=1}^n\widetilde{\Gamma}^k_{ij} E_j\cdot E_k\cdot A\psi.\quad
\end{eqnarray}
Hence we obtain
\begin{eqnarray}
  \nonumber
  AD^g\psi&=&D^{\geucl}A\psi
  +\sum_{i,j=1}^n(B^j_i-\delta^j_i)E_i\cdot\nabla_{E_j}A\psi\\
  \label{dirac_triv}
  &&{}+\frac{1}{4}\sum_{i,j,k=1}^n\widetilde{\Gamma}^k_{ij} E_i\cdot E_j\cdot E_k\cdot A\psi.
\end{eqnarray}
Let $\partial_j:=d\rho(E_j)$, $1\leq j\leq n$, be the coordinate vector fields of the normal coordinates. 
It is well known that the Taylor expansion of 
\begin{displaymath}
  g_{ij}:\quad V\to\R,\quad g_{ij}(x):=g(\partial_i|_{\rho(x)},\partial_j|_{\rho(x)})
\end{displaymath}
around $0$ is given by
\begin{displaymath}
  g_{ij}(x)=\delta_{ij}+\frac{1}{3}\sum_{a,b=1}^n R_{iabj}(p)x_a x_b+O(|x|^3),
\end{displaymath}
where $R_{iabj}=g(R(\partial_b,\partial_j)\partial_a,\partial_i)$ denotes the components 
of the Riemann curvature tensor (see e.\,g.\,\cite{lp}, p. 61). 
Since we have $(B^j_i)_{ij}=(g_{ij})_{ij}^{-1/2}$ it follows that
\begin{equation}
  \label{B-expansion}
  B^j_i(x)=\delta^j_i-\frac{1}{6}\sum_{a,b=1}^n R_{iabj}(p)x_a x_b +O(|x|^3).
\end{equation}
By definition we have
\begin{displaymath}
  e_i=b(E_i)=d\rho(B(E_i))=\sum_{j=1}^n B^j_i d\rho(E_j)
  =\sum_{j=1}^n B^j_i\partial_j.
\end{displaymath}
We define the Christoffel symbols $\Gamma^m_{kr}$ by
\begin{displaymath}
  \nabla^g_{\partial_k}\partial_r=\sum_{m=1}^n\Gamma^m_{kr}\partial_m.
\end{displaymath}
We obtain
\begin{eqnarray*}
  \sum_{k=1}^n\widetilde{\Gamma}^k_{ij}e_k&=&\sum_{s,m=1}^n B^s_i\nabla^g_{\partial_s}(B^m_j\partial_m)\\
  &=&\sum_{s,m=1}^n B^s_i\big(\partial_s B^m_j+\sum_{r=1}^n B^r_j\Gamma^m_{sr}\big)\partial_m\\
  &=&\sum_{s,m,k=1}^n B^s_i\big(\partial_s B^m_j+\sum_{r=1}^n B^r_j\Gamma^m_{sr}\big)(B^{-1})^k_m e_k.
\end{eqnarray*}
From the formula
\begin{displaymath}
  \Gamma^m_{kr}=\frac{1}{2}\sum_{s=1}^n g^{ms}(\partial_k g_{rs}+\partial_r g_{ks}-\partial_s g_{kr})
\end{displaymath}
it follows that $\Gamma^m_{kr}(x)=O(|x|)$ and thus 
\begin{equation}
  \label{gammatilde}
  \widetilde{\Gamma}^m_{kr}(x)=O(|x|)
\end{equation}
as $x\to0$ for all $m$, $k$, $r$.

\section{The Euclidean Dirac operator}

The aim of this section is to calculate preimages under the Dirac operator 
of certain spinors on $\R^n\setminus\{0\}$ with the Euclidean metric. 
The results will be useful for obtaining the expansion 
of Green's function for the Dirac operator on a closed spin manifold. 

\begin{definition}
\label{pkmi_def}
For $k\in\R$, $m\in\N$ and $i\in\{0,1\}$ we define the vector subspaces $P_{k,m,i}(\R^n)$ of $\spinor{\R^n|_{\R^n\setminus\{0\}}}{}{\infty}$ 
as follows. For $k\neq0$ define 
\begin{eqnarray*}
  P_{k,m,0}(\R^n)&:=&\span\Big\{x\mapsto x_{i_1}...x_{i_m}|x|^k\gamma\,
  \Big|\begin{array}{l}1\leq i_1,...,i_m\leq n,\\ \gamma\in\Sigma_n\textrm{ constant}\end{array}\Big\}\\
  P_{k,m,1}(\R^n)&:=&\span\Big\{x\mapsto x_{i_1}...x_{i_m}|x|^k x\cdot\gamma\,
  \Big|\begin{array}{l}1\leq i_1,...,i_m\leq n,\\ \gamma\in\Sigma_n\textrm{ constant}\end{array}\Big\}
\end{eqnarray*}
and furthermore define 
\begin{eqnarray*}
  P_{0,m,0}(\R^n)&:=&\span\Big\{x\mapsto x_{i_1}...x_{i_m}\ln|x|\gamma\,
  \Big|\begin{array}{l}1\leq i_1,...,i_m\leq n,\\ \gamma\in\Sigma_n\textrm{ constant}\end{array}\Big\}\\
  P_{0,m,1}(\R^n)&:=&\span\Big\{x\mapsto x_{i_1}...x_{i_m}(1-n\ln|x|) x\cdot\gamma\,
  \Big|\begin{array}{l}1\leq i_1,...,i_m\leq n,\\ \gamma\in\Sigma_n\textrm{ constant}\end{array}\Big\}.
\end{eqnarray*}
\end{definition}

Note that there exist inclusions among these spaces. For example one has 
$P_{k+2,m,0}(\R^n)\subset P_{k,m+2,0}(\R^n)$ for all $m$ and all $k\neq-2$. 
However in the following we will very often not use these inclusions, 
since exceptions like the case $k=-2$ in this example would make 
the following statements rather more complicated. 
The exception $k=-2$ in this example is due to our definition of $P_{0,m,0}(\R^n)$. 
The following proposition will show that this definition is nevertheless useful. 

\begin{proposition}
\label{prop_invert_dirac}
For all $m\in\N$, $k\in\R$ with $-n\leq k$ and $-n<k+m\leq0$ we have 
\begin{displaymath}
  P_{k,m,0}(\R^n)\subset D^{\geucl}\Big(\sum_{j=1}^{[(m+1)/2]}P_{k+2j,m+1-2j,0}(\R^n)+\sum_{j=0}^{[m/2]}P_{k+2j,m-2j,1}(\R^n)\Big)
\end{displaymath}
For all $m\in\N$, $k\in\R$ with $-n\leq k$ and $-n<k+m+1\leq0$ we have 
\begin{displaymath}
  P_{k,m,1}(\R^n)\subset D^{\geucl}\Big(\sum_{j=0}^{[m/2]}P_{k+2+2j,m-2j,0}(\R^n)+\sum_{j=1}^{[(m+1)/2]}P_{k+2j,m+1-2j,1}(\R^n)\Big).
\end{displaymath}
\end{proposition}

\begin{proof}
We use induction on $m$. 
Let $m=0$ and let $\gamma$ be a constant spinor on $(\R^n,\geucl)$. 
We want to prove that $P_{k,0,0}(\R^n)\subset D^{\geucl}(P_{k,0,1}(\R^n))$ for all $k$ with $-n<k\leq0$ 
and $P_{k,0,1}(\R^n)\subset D^{\geucl}(P_{k+2,0,0}(\R^n))$ for all $k$ with $-n\leq k$ and $k+1\leq0$. 
One calculates easily that 
\begin{eqnarray*}
  D^{\geucl}\big(-\frac{1}{n+k}|x|^kx\cdot\gamma\big)&=&|x|^k\gamma,\quad k\neq-n\\
  D^{\geucl}\big(\frac{1-n\ln|x|}{n^2}x\cdot\gamma\big)&=&\ln|x|\gamma,\\
  D^{\geucl}\big(\frac{1}{k+2}|x|^{k+2}\gamma\big)&=&|x|^kx\cdot\gamma,\quad k\neq-2\\
  D^{\geucl}(\ln|x|\gamma)&=&|x|^{-2}x\cdot\gamma.
\end{eqnarray*}
Using the definition of $P_{k,0,i}(\R^n)$ one finds that the assertion holds for $m=0$.

Let $m\geq1$ and assume that all the inclusions in the assertion hold for $m-1$. 
Using the equation $E_i\cdot x=-2x_i-x\cdot E_i$ we find 
\begin{eqnarray*}
  &&D^{\geucl}(-x_{i_1}...x_{i_m}|x|^k x\cdot\gamma)\\
  &=&-\sum_{j=1}^m x_{i_1}...\widehat{x_{i_j}}...x_{i_m}|x|^k E_{i_j}\cdot x\cdot\gamma
  -x_{i_1}...x_{i_m} D^{\geucl}(|x|^k x\cdot\gamma)\\
  &=&(2m+n+k)x_{i_1}...x_{i_m}|x|^k\gamma
  +\sum_{j=1}^m x_{i_1}...\widehat{x_{i_j}}...x_{i_m}|x|^k x\cdot E_{i_j}\cdot\gamma.
\end{eqnarray*}
Since $E_{i_j}\cdot\gamma$ is a parallel spinor the sum on the right hand side 
is contained in $P_{k,m-1,1}(\R^n)$. We apply the induction hypothesis and since $2m+n+k\neq0$ we find 
that the assertion for $P_{k,m,0}(\R^n)$ holds. 
We define $f_k(x):=\frac{1}{k}|x|^k$ for $k\neq0$ and $f_0(x):=\ln|x|$. 
Then we find 
\begin{eqnarray*}
  &&D^{\geucl}(x_{i_1}...x_{i_m}f_{k+2}(x)\gamma)\\
  &=&x_{i_1}...x_{i_m}|x|^k x\cdot\gamma
  +\sum_{j=1}^m x_{i_1}...\widehat{x_{i_j}}...x_{i_m}f_{k+2}(x) E_{i_j}\cdot\gamma.
\end{eqnarray*}
The sum on the right hand side is in $P_{k+2,m-1,0}(\R^n)$. 
Again we apply the induction hypothesis and we find 
that the assertion for $P_{k,m,1}(\R^n)$ holds.
\end{proof}


\section{Expansion of Green's function}

Let~$(M,g,\Theta)$ be a closed Riemannian spin manifold of dimension $n$ 
and let $\lambda\in\R$. 
Since $D^g-\lambda$ is a self-adjoint elliptic differential operator, 
there is an $L^2$-orthogonal decomposition
\begin{displaymath}
  \spinor{M}{g}{\infty}=\ker(D^g-\lambda)\oplus\im(D^g-\lambda).
\end{displaymath}
Let $P$: $\spinor{M}{g}{\infty}\to\ker(D^g-\lambda)$ denote the 
$L^2$-orthogonal projection. 
Then there is an operator $G$: $\spinor{M}{g}{\infty}\to\spinor{M}{g}{\infty}$, 
called Green's operator such that
\begin{displaymath}
  G(D^g-\lambda)=(D^g-\lambda)G=\id-P
\end{displaymath}
and it extends to a bounded linear operator 
$\Hspinor{M}{g}{k}\to\Hspinor{M}{g}{k+1}$ for every $k\in\N$ 
(see e.\,g.\,\cite{lm}, p. 195). In this section we will examine the 
integral kernel of Green's operator, which is called 
Green's function for $D^g-\lambda$. 
Let $\pi_i$: $M\times M\to M$, $i=1,2$ be the projections. We define
\begin{displaymath}
  \Sigma^gM\boxtimes\Sigma^gM^*:=\pi_1^*\Sigma^gM\otimes(\pi_2^*\Sigma^gM)^*
\end{displaymath}
i.\,e.\,$\Sigma^gM\boxtimes\Sigma^gM^*$ is the vector bundle over~$M\times M$ 
whose fibre over the point $(x,y)\in M\times M$ is given by $\Hom_{\C}(\Sigma^g_yM,\Sigma^g_xM)$. 
Let~$\Delta:=\{(x,x)|x\in M\}$ be the diagonal. In the following we will abbreviate
\begin{displaymath}
  \int_{M\setminus\{p\}}:=\lim_{\ep\to0}\int_{M\setminus B_{\ep}(p)}.
\end{displaymath}

\begin{definition}
A smooth section $G^g_{\lambda}$: $M\times M\setminus\Delta\to\Sigma^gM\boxtimes\Sigma^gM^*$ 
which is locally integrable on~$M\times M$ is called a Green's function for $D^g-\lambda$ 
if for all~$p\in M$, for all~$\varphi\in\Sigma^g_pM$ and for all~$\psi\in\im(D^g-\lambda)$ we have
\begin{equation}
  \label{green_function1}
  \int_{M\setminus\{p\}} \big\langle (D^g-\lambda)\psi,G^g_{\lambda}(.,p)\varphi\big\rangle \dv^g=\big\langle \psi(p),\varphi \big\rangle,
\end{equation}
and if for all~$p\in M$, for all~$\varphi\in\Sigma^g_pM$ and for all~$\psi\in\ker(D^g-\lambda)$ we have
\begin{equation}
  \label{green_function2}
  \int_{M\setminus\{p\}} \big\langle \psi,G^g_{\lambda}(.,p)\varphi\big\rangle \dv^g=0.
\end{equation}
\end{definition}
In this section we will prove existence and uniqueness of Green's function for $D^g-\lambda$ 
in such a way that we also obtain 
the expansion of Green's function around the singularity. 
The smooth spinor $G^g_{\lambda}(.,p)\varphi$ on $M\setminus\{p\}$ will sometimes also 
be called Green's function for $\varphi$. 
Thus we have for all $\psi\in\spinor{M}{g}{\infty}$ and for all $\varphi\in\Sigma^g_pM$
\begin{equation}
  \label{greens_function}
  \int_{M\setminus\{p\}} \big\langle (D^g-\lambda)\psi,G^g_{\lambda}(.,p)\varphi\big\rangle \dv^g
  =\big\langle \psi(p)-P\psi(p),\varphi \big\rangle.
\end{equation}

On Euclidean space we define a Green's function as follows.

\begin{definition}
Let $(M,g)=(\R^n,\geucl)$ with the unique spin structure. 
A smooth section $G^g_{\lambda}$: $M\times M\setminus\Delta\to\Sigma^gM\boxtimes\Sigma^gM^*$ 
which is locally integrable on~$M\times M$ is called a Green's function for $D^g-\lambda$ 
if for all~$p\in M$, for all~$\varphi\in\Sigma^g_pM$ and for all~$\psi\in\spinor{M}{g}{\infty}$ 
with compact support the equation (\ref{green_function1}) holds.
\end{definition}

Of course a Green's function for $D^{\geucl}-\lambda$ is not uniquely 
determined by this definition. 
We will explicitly write down a Green's function for $D^{\geucl}-\lambda$ 
and use it later to find the expansion of Green's function 
for $D^g-\lambda$ on a closed Riemannian spin manifold $(M,g,\Theta)$ 
around the diagonal. 
First observe that for every spinor $\chi\in\spinor{\R^n|_{\R^n\setminus\{0\}}}{}{\infty}$ 
and for every $\lambda\in\R$ the equation
\begin{displaymath}
  (D^{\geucl}-\lambda)(D^{\geucl}+\lambda)\chi
  =-\sum_{i=1}^n \nabla_{E_i}\nabla_{E_i}\chi-\lambda^2\chi
\end{displaymath}
holds on $\R^n\setminus\{0\}$. 
Therefore if $\gamma$ is a constant spinor on $(\R^n,\geucl)$ and 
if $f\in C^{\infty}(\R^n\setminus\{0\},\R)$ satisfies
\begin{displaymath}
  -\sum_{i=1}^n \frac{\partial^2 f}{\partial x_i^2}-\lambda^2f=\delta_0,
\end{displaymath}
then the spinor $G^{\geucl}_{\lambda}(.,0)\gamma:=(D^{\geucl}+\lambda)(f\gamma)$ 
is a Green's function. 
Writing $f(x)=g(|x|)$ for a function $g$ of one variable we get a solution if $g$ 
solves the ordinary differential equation 
\begin{equation}
  \label{green_ode}
  g''(z)+\frac{n-1}{z}g'(z)+\lambda^2g(z)=-\delta_0.
\end{equation}

In the following let $\Gamma$ denote the Gamma function and $J_m$, $Y_m$ 
the Bessel functions of the first and second kind for the parameter $m\in\R$. 
In the notation of \cite{as}, p. 360 they are defined for $z\in(0,\infty)$ by 
\begin{eqnarray*}
J_m(z)&=&\frac{1}{2^m\Gamma(m+1)}z^m\big(1+\sum_{k=1}^{\infty}a_k z^{2k}\big),\quad m\in\R,\\
Y_0(z)&=&\frac{2}{\pi}\big(\ln\big(\frac{z}{2}\big)+c\big)J_0(z)+\sum_{k=1}^{\infty}b_k z^{2k},\\
Y_m(z)&=&-\frac{2^m}{\pi}\Gamma(m)z^{-m}\big(1+\sum_{k=1}^{\infty}c_k z^{2k}\big),\quad m=\frac{1}{2}+k,k\in\N, \\
Y_m(z)&=&-\frac{2^m}{\pi}\Gamma(m)z^{-m}\big(1+\sum_{k=1}^{\infty}d_k z^{2k}\big)+\frac{2}{\pi}\ln\big(\frac{z}{2}\big)J_m(z),\quad m\in\N\setminus\{0\},
\end{eqnarray*}
where $c$ is a real constant, the $a_k$, $b_k$, $c_k$, $d_k$ are real coefficients, 
the $a_k$, $c_k$, $d_k$ depend on $m$ 
and all the power series converge for all $z\in(0,\infty)$.
Let $\omega_{n-1}=\vol(S^{n-1},\gcan)$ be the volume of the $(n-1)$-dimensional unit sphere 
with the standard metric.

\begin{theorem}
\label{theorem_green_eucl}
Let $m:=\frac{n-2}{2}$. 
We define $f_{\lambda}$: $\R^n\setminus\{0\}\to\R$ as follows. 
For $\lambda\neq0$ and $n=2$ 
\begin{displaymath}
  f_{\lambda}(x):=-\frac{1}{4}Y_0(|\lambda x|)+\frac{\ln|\lambda|-\ln(2)+c}{2\pi}J_0(|\lambda x|),
\end{displaymath}
for $\lambda\neq0$ and odd $n\geq3$ 
\begin{displaymath}
  f_{\lambda}(x):=-\frac{\pi|\lambda|^m}{2^m\Gamma(m)(n-2)\omega_{n-1}}|x|^{-m}Y_m(|\lambda x|),
\end{displaymath}
for $\lambda\neq0$ and even $n\geq4$ 
\begin{displaymath}
  f_{\lambda}(x):=-\frac{\pi|\lambda|^m}{2^m\Gamma(m)(n-2)\omega_{n-1}}|x|^{-m}\Big(Y_m(|\lambda x|)-\frac{2(\ln|\lambda|-\ln(2))}{\pi}J_m(|\lambda x|)\Big)
\end{displaymath}
and 
\begin{displaymath}
  f_0(x):=-\frac{1}{2\pi}\ln|x|,\quad n=2,\qquad f_0(x):=\frac{1}{(n-2)\omega_{n-1}|x|^{n-2}},\quad n\geq3.
\end{displaymath}
Then for every constant spinor $\gamma$ on $\R^n$ a Green's function for $D^{\geucl}-\lambda$ is given by
\begin{displaymath}
  G^{\geucl}_{\lambda}(x,0)\gamma=(D^{\geucl}+\lambda)(f_{\lambda}\gamma)(x).
\end{displaymath}
\end{theorem}

\begin{corollary}
\label{coroll_green_asymp}
For every constant spinor $\gamma\in\Sigma_n$ there exists a Green's function 
of $D^{\geucl}-\lambda$, which has the following form. For $n=2$ 
\begin{displaymath}
  G^{\geucl}_{\lambda}(x,0)\gamma
  =-\frac{1}{2\pi|x|}\frac{x}{|x|}\cdot\gamma-\frac{\lambda}{2\pi}\ln|x|\gamma+\ln|x|\vartheta_{\lambda}(x)+\zeta_{\lambda}(x),
\end{displaymath}
for odd $n\geq3$ 
\begin{displaymath}
  G^{\geucl}_{\lambda}(x,0)\gamma
  =-\frac{1}{\omega_{n-1}|x|^{n-1}}\frac{x}{|x|}\cdot\gamma
  +\frac{\lambda}{(n-2)\omega_{n-1}|x|^{n-2}}\gamma+|x|^{2-n}\zeta_{\lambda}(x),
\end{displaymath}
for even $n\geq4$ 
\begin{eqnarray*}
  G^{\geucl}_{\lambda}(x,0)\gamma
  &=&-\frac{1}{\omega_{n-1}|x|^{n-1}}\frac{x}{|x|}\cdot\gamma
  +\frac{\lambda}{(n-2)\omega_{n-1}|x|^{n-2}}\gamma+|x|^{2-n}\zeta_{\lambda}(x)\\
  &&{}-\frac{\lambda^{n-1}}{2^{n-2}\Gamma(\frac{n}{2})^2\omega_{n-1}}\ln|x|\gamma+\ln|x|\vartheta_{\lambda}(x),
\end{eqnarray*}
where for every $n$ and for every $\lambda$ the spinors $\vartheta_{\lambda}$, $\zeta_{\lambda}$ 
extend smoothly to $\R^n$ and satisfy 
\begin{displaymath}
  |\zeta_{\lambda}(x)|_{\geucl}=O(|x|),\quad |\vartheta_{\lambda}(x)|_{\geucl}=O(|x|)\quad\textrm{ as } x\to0
\end{displaymath}
and where for every $n$ and for every $x$ the spinors $\vartheta_{\lambda}(x),\zeta_{\lambda}(x)\in\Sigma_n$ 
are power series in $\lambda$ with $\vartheta_0(x)=\zeta_0(x)=0$.
\end{corollary}

\begin{proof}[Proof of Corollary \ref{coroll_green_asymp}]
We find
\begin{displaymath}
  Y_0(|\lambda x|)-\frac{2(\ln|\lambda|-\ln(2)+c)}{\pi}J_0(|\lambda x|)
  =\frac{2}{\pi}\ln|x|J_0(|\lambda x|)+\sum_{k=1}^{\infty}b_k |\lambda x|^{2k}
\end{displaymath}
and for $m\in\N\setminus\{0\}$
\begin{eqnarray*}
Y_m(|\lambda x|)-\frac{2(\ln|\lambda|-\ln(2))}{\pi}J_m(|\lambda x|)
&=&-\frac{2^m\Gamma(m)}{\pi|\lambda x|^m}\big(1+\sum_{k=1}^{\infty}d_k |\lambda x|^{2k}\big)\\
&&{}+\frac{2}{\pi}\ln|x|J_m(|\lambda x|).
\end{eqnarray*}
Thus we find for $n=2$
\begin{displaymath}
  f_{\lambda}(x)=-\frac{1}{2\pi}\ln|x|\big(1+\sum_{k=1}^{\infty}a_k|\lambda x|^{2k}\big)
  -\frac{1}{4}\sum_{k=1}^{\infty}b_k|\lambda x|^{2k},
\end{displaymath}
for odd $n\geq3$
\begin{displaymath}
  f_{\lambda}(x)=\frac{1}{(n-2)\omega_{n-1}|x|^{n-2}}\big(1+\sum_{k=1}^{\infty}c_k|\lambda x|^{2k}\big)
\end{displaymath}
and for even $n\geq4$
\begin{eqnarray*}
  f_{\lambda}(x)&=&\frac{1}{(n-2)\omega_{n-1}|x|^{n-2}}\big(1+\sum_{k=1}^{\infty}d_k|\lambda x|^{2k}\big)\\
  &&{}-\frac{\lambda^{n-2}}{2^{n-2}(m!)^2\omega_{n-1}}\ln|x|\big(1+\sum_{k=1}^{\infty}a_k|\lambda x|^{2k}\big).
\end{eqnarray*}
The assertion follows.
\end{proof}

\begin{proof}[Proof of Theorem \ref{theorem_green_eucl}]
Let $f_{\lambda}$ be as in the assertion and write $f_{\lambda}(x)=g_{\lambda}(|x|)$ 
with $g_{\lambda}$: $(0,\infty)\to\R$. 
One finds that $g_{\lambda}$ solves the equation (\ref{green_ode}). 
It remains to show that $(D^{\geucl}+\lambda)(f_{\lambda}\gamma)$ satisfies (\ref{green_function1}). 
The calculation of the proof of the Corollary shows
\begin{displaymath}
  (D^{\geucl}+\lambda)(f_{\lambda}\gamma)(x)=-\frac{1}{\omega_{n-1}} \frac{x}{|x|^n}\cdot\gamma+\zeta(x),
\end{displaymath}
where $|\zeta(x)|_{\geucl}=o(|x|^{1-n})$ as $x\to0$. 
It is well known (e.\,g.\,\cite{lm}, p. 115) that for a Riemannian spin manifold $(M,g,\Theta)$ with boundary $\partial M\neq\emptyset$ 
and $\psi$, $\varphi$ compactly supported spinors we have
\begin{displaymath}
  (D^g\psi,\varphi)_2-(\psi,D^g\varphi)_2=\int_{\partial M}\langle \nu\cdot\psi,\varphi\rangle\,dA,
\end{displaymath}
where $\nu$ is the outer unit normal vector field on $\partial M$. 
We apply this to $(\R^n\setminus B_{\varepsilon}(0),\geucl)$ and $\nu(x):=-\frac{x}{|x|}$ and we find
\begin{eqnarray*}
  &&\int_{\R^n\setminus B_{\varepsilon}(0)}\big\langle (D^{\geucl}-\lambda)\psi(x),(D^{\geucl}+\lambda)(f_{\lambda}\gamma)(x)\big\rangle \dv^g\\
  &=&{}-\int_{\partial B_{\varepsilon}(0)} \big\langle \frac{x}{|x|}\cdot\psi(x),(D^{\geucl}+\lambda)(f_{\lambda}\gamma)(x) \big\rangle\,dA\\
  &=&{}\int_{\partial B_{\varepsilon}(0)} \big\langle \psi(x),\frac{1}{\omega_{n-1}|x|^{n-1}}\gamma+\frac{x}{|x|}\cdot\zeta(x) \big\rangle\,dA.
\end{eqnarray*}
With the substitution $x=\varepsilon y$ we find that
\begin{displaymath}
  \lim_{\varepsilon\to0}
  \int_{\R^n\setminus B_{\varepsilon}(0)}\big\langle (D^{\geucl}-\lambda)\psi(x),(D^{\geucl}+\lambda)(f_{\lambda}\gamma)(x)\big\rangle \dv^g
  =\langle\psi(0),\gamma\rangle.
\end{displaymath}
The assertion follows.
\end{proof}

\begin{remark}
In the case $n=3$ we get more familiar expressions by using that
\begin{displaymath}
  \sqrt{\frac{\pi}{2z}}Y_{1/2}(z)=-\frac{\cos(z)}{z}
\end{displaymath}
(see \cite{as}, p. 437f). Similar formulas exist for all odd $n$.
\end{remark}

\begin{definition}
For $m\in\R$ we define 
\begin{displaymath}
  P_m(\R^n):=\sum_{r+s+t\geq m\atop r\geq-n} P_{r,s,t}(\R^n)
  +(\spinor{\R^n|_{\R^n\setminus\{0\}}}{}{\infty}\cap\spinor{\R^n}{}{0}),
\end{displaymath}
where the second space is the space of all spinors which are smooth 
on $\R^n\setminus\{0\}$ and have a continuous extension to $\R^n$.
\end{definition}

\begin{remark}
Let $\vartheta\in P_m(\R^n)$. 
Then we have $E_i\cdot\vartheta\in P_m(\R^n)$ for all $i\in\{1,...,n\}$. 
If $f\in\C^{\infty}(\R^n)$ then for every $s\in\N$ by Taylor's formula we may write 
\begin{displaymath}
  f(x)=\sum_{|\alpha|<s} \frac{1}{\alpha!} \frac{\partial^{|\alpha|}f(0)}{\partial x^{\alpha}}x^{\alpha}+R_s(x),
\end{displaymath}
where $|R_s(x)|=O(|x|^s)$ as $x\to0$. 
Thus by choosing $s\geq m+1$ we find that the spinor $f\vartheta$ is in $P_m(\R^n)$.
\end{remark}

\begin{remark}
\label{remark_invert_dirac}
A spinor $\vartheta\in P_m(\R^n)$ has a continuous extension to $\R^n$ if and only if $m>0$. 
Furthermore by Proposition \ref{prop_invert_dirac} it follows that for all $m\in(-n,0]$ we have 
\begin{eqnarray*}
  P_m(\R^n)&=&\sum_{r+s+t=m\atop r\geq-n} P_{r,s,t}(\R^n)+P_{m+1}(\R^n)\\
  &\subset& D^{\geucl}(P_{m+1}(\R^n))+P_{m+1}(\R^n).
\end{eqnarray*}
\end{remark}

\begin{lemma}
\label{green_p_2-n_lemma}
Let $(M,g,\Theta)$ be a closed Riemannian spin manifold of dimension $n$ 
and let $\lambda\in\R$. Let $\gamma\in\Sigma_n$ be a constant spinor on $\R^n$. 
Then the spinor $G^{\geucl}_{\lambda}(.,0)\gamma$ 
is in $P_{1-n}(\R^n)$. Let the matrix coefficients 
$B^j_i$ be defined as in (\ref{dirac_triv}). 
Then for all $i$ the spinor
\begin{displaymath}
  x\mapsto\sum_{j=1}^n (B^j_i(x)-\delta^j_i)\nabla_{E_j}G^{\geucl}_{\lambda}(x,0)\gamma
\end{displaymath}
is in $P_{2-n}(\R^n)$.
\end{lemma}

\begin{proof}
The assertion for $G^{\geucl}_{\lambda}(.,0)\gamma$ 
can be seen immediately from Corollary \ref{coroll_green_asymp}. 
Let $f_{\lambda}$ be as in Theorem \ref{theorem_green_eucl} and write $f_{\lambda}(x)=g_{\lambda}(|x|)$ 
with a func\-tion $g_{\lambda}$: $(0,\infty)\to\R$. Then we have 
\begin{displaymath}
  G^{\geucl}_{\lambda}(x,0)\gamma=\frac{g_{\lambda}'(|x|)}{|x|}x\cdot\gamma+\lambda g_{\lambda}(|x|)\gamma
\end{displaymath}
and for every $j\in\{1,...,n\}$ we have 
\begin{displaymath}
  \nabla_{E_j}G^{\geucl}_{\lambda}(x,0)\gamma=\frac{g_{\lambda}''(|x|)x_j}{|x|^2}x\cdot\gamma
  -\frac{g_{\lambda}'(|x|)x_j}{|x|^3}x\cdot\gamma
  +\frac{g_{\lambda}'(|x|)}{|x|}E_j\cdot\gamma+\lambda\frac{g_{\lambda}'(|x|)x_j}{|x|}\gamma.
\end{displaymath}
Since the exponential map is a radial isometry, we have 
$\sum_{j=1}^n g_{ij}(x)x_j=x_i$ and thus 
$\sum_{j=1}^n B^j_i(x)x_j=x_i$ for every fixed $i$. Thus we find 
\begin{displaymath}
  \sum_{j=1}^n (B^j_i(x)-\delta^j_i)\nabla_{E_j}G^{\geucl}_{\lambda}(x,0)\gamma
  =\sum_{j=1}^n (B^j_i(x)-\delta^j_i)\frac{g_{\lambda}'(|x|)}{|x|}E_j\cdot\gamma.
\end{displaymath}
Since we have $g_{\lambda}'(|x|)=O(|x|^{1-n})$ as $x\to0$ the assertion 
now follows from the Taylor expansion (\ref{B-expansion}) of $B^j_i(x)$. 
\end{proof}

Next we prove existence and uniqueness of Green's function for $D^g-\lambda$ 
on a closed Riemannian spin manifold in such a way that we also obtain 
the expansion of Green's function around the singularity. 
The idea is to apply the equation (\ref{dirac_triv}) for the Dirac operator 
in the trivialization to a Euclidean Green's function and then determine 
the correction terms. This has been carried out in \cite{ah}, 
where for some technical steps Sobolev embeddings were used. 
We present a more simple argument using the preimages under the 
Dirac operator from Proposition \ref{prop_invert_dirac}.

In the following for a fixed point $p\in M$ let $\rho$: $V\to U$ be a local 
parametrization of $M$ by Riemannian normal coordinates, where $U\subset M$ 
is an open neighborhood of $p$, $V\subset\R^n$ is an open neighborhood of $0$ 
and $\rho(0)=p$. 
Furthermore let 
\begin{displaymath}
  \beta:\quad \Sigma\R^n|_V\to\Sigma^gM|_U,\quad 
  A:\quad \spinor{M|_U}{g}{\infty}\to\spinor{\R^n|_V}{}{\infty}
\end{displaymath}
denote the maps which send a spinor to its 
corresponding spinor in the Bourguignon-Gauduchon trivialization 
defined in Section \ref{bourg-triv_section}.

\begin{theorem}
\label{theorem_green}
Let $(M,g,\Theta)$ be a closed $n$-dimensional Riemannian spin manifold, $p\in M$. 
For every $\varphi\in\Sigma^g_pM$ there exists a unique Green's function 
$G^g_{\lambda}(.,p)\varphi$. 
If $\gamma:=\beta^{-1}\varphi\in\Sigma_n$ is the constant spinor on $\R^n$ 
corresponding to $\varphi$, 
then the first two terms 
of the expansion of $AG^g_{\lambda}(.,p)\varphi$ at $0$ 
coincide with the first two terms of the expansion of  
$G^{\geucl}_{\lambda}(.,0)\gamma$ given in Corollary \ref{coroll_green_asymp}.
\end{theorem}

\begin{proof}
Let $\varepsilon>0$ such that $B_{2\varepsilon}(0)\subset V$ and let 
$\eta:$ $\R^n\to[0,1]$ be a smooth function with $\supp(\eta)\subset B_{2\varepsilon}(0)$ 
and $\eta\equiv1$ on $B_{\varepsilon}(0)$. 
Then the spinor $\Theta_1$ defined on $\R^n\setminus\{0\}$ by 
$\Theta_1(x):=\eta(x)G^{\geucl}_{\lambda}(x,0)\gamma$ is smooth on $\R^n\setminus\{0\}$. 
For $r\in\{1,...,n\}$ we define smooth spinors $\Phi_r$ on $M\setminus\{p\}$ 
and $\Theta_{r+1}$ on $\R^n\setminus\{0\}$ inductively as follows. For $r=1$ we define 
\begin{displaymath}
  \Phi_1(q):=\left\{\begin{array}{ll}
  A^{-1}\Theta_1(q),&q\in U\setminus\{p\}\\ 
  0,&q\in M\setminus U\end{array}\right.
\end{displaymath}
and 
\begin{displaymath}
  \Theta_2(x):=\left\{\begin{array}{ll}
  A(D^g-\lambda)\Phi_1(x),&x\in V\setminus\{0\}\\ 
  0,&x\in \R^n\setminus V\end{array}\right. .
\end{displaymath}
By the formula (\ref{dirac_triv}) for the Dirac operator 
in the trivialization we have on $V\setminus\{0\}$
\begin{eqnarray*}
  \Theta_2
  &=&(D^{\geucl}-\lambda)\Theta_1
  +\sum_{i,j=1}^n(B^j_i-\delta^j_i)E_i\cdot\nabla_{E_j}\Theta_1\\
  &&{}+\frac{1}{4}\sum_{i,j,k=1}^n\widetilde{\Gamma}^k_{ij} E_i\cdot E_j\cdot E_k\cdot \Theta_1.
\end{eqnarray*}
The first term vanishes on $B_{\ep}(0)\setminus\{0\}$. 
It follows from 
the expansions of $\widetilde{\Gamma}^k_{ij}$ and $B^j_i-\delta^j_i$ 
in (\ref{B-expansion}), (\ref{gammatilde}) and from Lemma \ref{green_p_2-n_lemma} 
that $\Theta_2\in P_{2-n}(\R^n)$.

Next let $r\in\{2,...,n\}$ and assume that $\Phi_{r-1}$ and $\Theta_r$ 
have already been defined. We may assume that $\Theta_r\in P_{r-n}(\R^n)$. 
By Remark \ref{remark_invert_dirac} there exists 
$\beta_{r+1}\in P_{r+1-n}(\R^n)$ such that 
$\Theta_r-(D^{\geucl}-\lambda)\beta_{r+1}\in P_{r+1-n}(\R^n)$. 
We define $\Phi_r$ and $\Theta_{r+1}$ by 
\begin{displaymath}
  \Phi_r(q):=\left\{\begin{array}{ll}
  \Phi_{r-1}(q)-A^{-1}(\eta\beta_{r+1})(q),&q\in U\setminus\{p\}\\ 
  0,&q\in M\setminus U\end{array}\right.
\end{displaymath}
and 
\begin{displaymath}
  \Theta_{r+1}(x):=\left\{\begin{array}{ll}
  A(D^g-\lambda)\Phi_r(x),&x\in V\setminus\{0\}\\ 
  0,&x\in \R^n\setminus V\end{array}\right. .
\end{displaymath}
By the formula (\ref{dirac_triv}) for the Dirac operator 
in the trivialization we have on $B_{\ep}(0)\setminus\{0\}$
\begin{eqnarray*}
  \Theta_{r+1}&=&A(D^g-\lambda)\Phi_{r-1}-A(D^g-\lambda)A^{-1}\beta_{r+1}\\
  &=&\Theta_r-(D^{\geucl}-\lambda)\beta_{r+1}-\sum_{i,j=1}^n(B^j_i-\delta^j_i)E_i\cdot\nabla_{E_j}\beta_{r+1}\\
  &&{}-\frac{1}{4}\sum_{i,j,k=1}^n\widetilde{\Gamma}^k_{ij} E_i\cdot E_j\cdot E_k\cdot \beta_{r+1}.
\end{eqnarray*}
Using the expansions of $\widetilde{\Gamma}^k_{ij}$ and $B^j_i-\delta^j_i$ 
in (\ref{B-expansion}), (\ref{gammatilde}) we conclude that $\Theta_{r+1}\in P_{r+1-n}(\R^n)$. 

We see that $\Theta_{n+1}$ has a continuous extension to $\R^n$ 
and we obtain a continuous extension $\Psi$ of $(D^g-\lambda)\Phi_n$ to all of $M$. 
Thus there exists
\begin{displaymath}
  \Psi'\in\spinor{M|_{M\setminus\{p\}}}{g}{\infty}\cap\Hspinor{M}{g}{1}
\end{displaymath}
such that $(D^g-\lambda)\Psi'=P\Psi-\Psi$. Define
\begin{displaymath}
  \Gamma:=\Phi_n+\Psi',\quad \Theta:=-\eta\beta_3-...-\eta\beta_{n+1}+A\Psi'.
\end{displaymath}
Then on $B_{\ep}(0)\setminus\{0\}$ we have $A\Gamma=G^{\geucl}_{\lambda}(.,0)\gamma+\Theta$.

If $P$ is the $L^2$-orthogonal projection onto $\ker(D^g-\lambda)$, 
then 
\begin{displaymath}
  G^g_{\lambda}(.,p)\varphi:=\Gamma-P\Gamma
\end{displaymath}
satisfies (\ref{green_function1}), (\ref{green_function2}) 
and thus is a Green's function. 
Uniqueness also follows from (\ref{green_function1}), (\ref{green_function2}). 
The statement on the expansion of 
$AG^g_{\lambda}(.,p)\varphi$ is obvious, since we have $\Theta\in P_{3-n}(\R^n)$. 
\end{proof}

\section{Definition of the mass endomorphism}
\label{def_mass_endo_section}

Let $(M,g,\Theta)$ be a closed Riemannian spin manifold 
and let $p\in M$. 
Assume that the metric $g$ is flat on an open neighborhood of $p$. 
Let $U\subset M$ be an open 
neighborhood of $p$, let $V\subset\R^n$ be an open neighborhood 
of $0$ and let $\rho$: $V\to U$ be a local parametrization of $M$ 
by Riemannian normal coordinates sending $0$ to $p$. Let 
\begin{displaymath}
  \beta^g:\quad \Sigma\R^n|_V\to\Sigma^gM|_{U}.
\end{displaymath}
be the identification of the spinor bundles defined in Section 
\ref{bourg-triv_section} and let
\begin{displaymath}
  A^g:\quad \spinor{M|_U}{g}{\infty}\to\spinor{\R^n|_V}{}{\infty},\quad
  \psi\mapsto (\beta^g)^{-1}\circ\psi\circ\rho.
\end{displaymath}
Let $\varphi\in\Sigma^g_pM$ and let $\gamma:=(\beta^g)^{-1}\varphi$. 
Since $g$ is flat on an open neighborhood of $p$, the terms 
$\widetilde{\Gamma}^k_{ij}$ and $B^j_i-\delta^j_i$ in the formula 
(\ref{dirac_triv}) vanish on an open neighborhood of $0\in\R^n$. 
Thus in the proof of Theorem \ref{theorem_green} the spinor $\Theta_2$ 
has a smooth extension to $\R^n$ vanishing on an open neighborhood of $0$. 
Then one obtains a smooth extension 
$\Psi$ of $(D^g-\lambda)\Phi_1$ to all of $M$ 
vanishing on an open neighborhood of $p$. 
Now there exists $\Psi'\in C^\infty(\Sigma^gM)$ 
such that $(D^g-\lambda)\Psi'=P\Psi-\Psi$. 
As above we define 
\begin{displaymath}
  \Gamma:=\Phi_1+\Psi',\quad G^g_{\lambda}(.,p)\varphi:=\Gamma-P\Gamma. 
\end{displaymath}
Thus the spinor 
\begin{displaymath}
  x\mapsto w^g(x,p)\varphi:=\Psi'(x)-P\Gamma(x)
\end{displaymath}
is smooth on all of $M$. 
It is independent of the choice of $\Psi'$ and can be regarded as the difference 
between Green's function for $D^g-\lambda$ 
and the spinor $(A^g)^{-1}(\eta G^{\geucl}_{\lambda}(.,0)\gamma)$. 

In the following we are interested in the case $\lambda=0$ and we denote 
Green's function for $D^g$ by $G^g:=G^g_0$. 
By Theorem \ref{theorem_green_eucl} a Euclidean Green's function 
for $D^{\geucl}$ is given by 
\begin{displaymath}
  G^{\geucl}(x,y)\gamma:=-\frac{1}{\omega_{n-1}|x-y|^{n-1}}\frac{x-y}{|x-y|}\cdot\gamma.
\end{displaymath}
Using the definition of $\Phi_1$ and $w^g(.,p)\varphi$ 
we have for all $x\in M\setminus\{p\}$
\begin{displaymath}
  G^g(x,p)\varphi=(A^g)^{-1}(\eta G^{\geucl}(.,0)\gamma)(x)+w^g(x,p)\varphi,
\end{displaymath}
where the first term on the right hand side is understood to be zero for $x\in M\setminus U$. 
Since $\eta\equiv1$ on $B_{\ep}(0)$, the spinor $w^g(p,p)\varphi\in\Sigma^g_pM$ 
is independent of the choice of $\eta$. It can be regarded as the constant term in an 
expansion of Green's function for $D^g$ around $p$. 

As in \cite{ahm} we define the mass endomorphism.

\begin{definition}
Let $(M,g,\Theta)$ be a closed Riemannian spin manifold with $\dim M\geq2$, 
which is conformally flat on an open 
neighborhood of~$p\in M$. 
Choose a metric~$h\in[g]$, which is flat on an open neighborhood of~$p$ and such that~$h_p=g_p$. 
Let~$G^h$ be Green's function for $D^h$. 
Then, we define the mass endomorphism in~$p$ as
\begin{displaymath}
  m^g_p:\Sigma^g_pM\rightarrow\Sigma^g_pM,\qquad\varphi\mapsto \beta_{h,g}w^h(p,p)\beta_{g,h}\varphi,
\end{displaymath}
where $w^h$ is the term in the above expansion.
\end{definition}

It is shown in \cite{ahm} 
that this definition does not depend on the choice of~$h\in[g]$ and 
that~$m^g_p$ is linear and self-adjoint. 
There is an analogy in conformal geometry: the constant term of 
Green's function~$\Gamma(.,p)$ for the Yamabe operator in~$p$ can be interpreted 
as the mass of the asymptotically flat manifold~$(M\setminus\{p\},\Gamma(.,p)^{4/(n-2)}g)$ 
(see \cite{lp}). Therefore, the endomorphism is called mass endomorphism. 

The aim of introducing the mass endomorphism in \cite{ahm} is to obtain 
at least one of the strict inequalities
\begin{displaymath}
  \lambda_{\min}^+(M,[g],\Theta)<\lambda_{\min}^+(\mathbb{S}^n),\quad
  \lambda_{\min}^-(M,[g],\Theta)<\lambda_{\min}^+(\mathbb{S}^n)
\end{displaymath}
from the hypothesis of Theorem \ref{theorem_lambdamin+}. 
The result of this article then reads as follows. 

\begin{theorem}
Let $(M,g,\Theta)$ be a closed Riemannian spin manifold of dimension $n\geq2$ with 
$\ker(D^g)=0$. Assume that there is a point $p\in M$ which has a conformally 
flat neighborhood and that the mass endomorphism in~$p$ possesses a positive 
(resp. negative) eigenvalue. Then we have 
\begin{displaymath}
  \lambda_{\min}^+(M,[g],\Theta)<\lambda_{\min}^+(\mathbb{S}^n)\quad
  \textrm{resp. }\lambda_{\min}^-(M,[g],\Theta)<\lambda_{\min}^+(\mathbb{S}^n).
\end{displaymath}
\end{theorem}

If $h\in[g]$ is a Riemannian metric such that $m^g_p=\beta_{h,g}w^h(p,p)\beta_{g,h}$, 
then the additional assumption $\ker(D^g)=0$ implies that $D^hw^h(.,p)\varphi$ 
vanishes on an open neighborhood of $p$ and this fact is used in the proof. 
It is not clear how to obtain a proof without this assumption. 

Unfortunately it is only known for very few Riemannian spin manifolds, 
whether there are points with nonzero mass endomorphism. 
For example, on the flat torus and 
on the sphere with the standard metric 
it always vanishes, whereas on the projective spaces $\R P^{4k+3}$ 
with the standard metric one has points 
with nonzero mass endomorphism (see \cite{ahm}). 
We will prove in Section \ref{gen_massendo_section} 
that in dimension $3$ the mass endomorphism is not zero in the generic case.

\chapter{Eigenspinors for generic metrics}
\label{gen_eigenspinor_chapter}

\section{Transversality}
\label{transversality_section}

A nice introduction to finite dimensional transversality theory can be 
found in \cite{hir}, while the infinite dimensional case is treated in 
\cite{la}. 
Transversality theory has already been used in the literature to determine 
generic properties of eigenvalues and eigenfunctions of 
differential operators. 
Some examples are the scalar Laplacian \cite{u} and 
the Hodge Laplacian \cite{ep} for generic metrics 
and solutions of the Dirac equation in Minkowski spacetime 
for generic initial data \cite{tum}.
The goal of this section is to quote a 
transversality theorem which will be useful later. 

\begin{definition}
Let $f$: $Q\to N$ be a $C^1$ map between two smooth manifolds. 
Let $A\subset N$ be a submanifold. 
$f$ is called transverse to $A$,\index{transverse} if for all $x\in Q$ with $f(x)\in A$ 
we have
\begin{displaymath}
  T_{f(x)}A+\im(df|_x)=T_{f(x)}N.
\end{displaymath}
\end{definition}

One reason for the importance of transversality is the following theorem.

\begin{theorem}
\label{transvers_submfd}
Assume that $Q$, $N$ are smooth manifolds without boundary.  
Let $f$: $Q\to N$ be a $C^r$ map, $r\geq1$ and $A\subset N$ a $C^r$ submanifold. 
If $f$ is transverse to $A$, then $f^{-1}(A)$ is a $C^r$ submanifold of $Q$. 
The codimension of $f^{-1}(A)$ in $Q$ is the same as the codimension of $A$ in $N$.
\end{theorem}

\begin{proof}
see \cite{hir}, p. 22 
\end{proof}

\begin{definition}
Let $X$ be a topological space. A subset $E\subset X$ 
is called nowhere dense in $X$, if the interior of $\overline{E}$ is empty. 
A subset of $X$ is called of first category, if it is a countable union of 
sets which are nowhere dense in $X$. Otherwise it is called of second 
category. A subset $B\subset X$ is called residual\index{residual} 
in $X$, if it contains a countable intersection of sets which are 
open and dense in $X$.
\end{definition}

If $A\subset X$ is of first category, it follows that the complement of $A$ 
is residual in $X$. An important result is the following theorem, 
which is called the Baire category theorem.

\begin{theorem}
\label{baire_theorem}
If $(X,d)$ is a complete metric space, then every residual subset of $X$ 
is dense in $X$.
\end{theorem}

\begin{proof}
see e.\,g.\,\cite{ru}, p. 43.
\end{proof}

An important result is Sard's theorem. 
\begin{theorem}
\label{sard_theorem}
Let $f$: $Q\to N$ be a $C^r$-map between smooth manifolds. If 
\begin{displaymath}
  r>\max\{0,\dim Q-\dim N\},
\end{displaymath}
then the set of all critical values of $f$ has measure zero in $N$. 
The set of all regular values of $f$ is residual and therefore dense. 
\end{theorem}

\begin{proof}
see \cite{hir}, p. 69.
\end{proof}

We quote the following transversality theorem from \cite{hir}, \cite{u} including the proof. 

\begin{theorem}
\label{param_transvers}
Let $V$, $M$, $N$ be smooth manifolds and let $A\subset N$ be a smooth submanifold. 
Let $F$: $V\to C^r(M,N)$ be a map, such that the evaluation map 
$F^{ev}$: $V\times M\to N$, $(v,m)\mapsto F(v)(m)$ is $C^r$ and transverse 
to $A$ and furthermore 
\begin{displaymath}
  r>\max\{0,\dim M+\dim A-\dim N\}.
\end{displaymath}
Then the set of all $v\in V$, 
such that the map $F(v)$ is transverse to $A$, 
is residual and therefore dense in $V$.
\end{theorem}

\begin{proof}
Since $F^{ev}$ is transverse to $A$, the set 
$P:=(F^{ev})^{-1}(A)$ is a $C^r$ submanifold of $V\times M$
of dimension $\dim V+\dim M+\dim A-\dim N$. 
Let $\pi_1$: $V\times M\to V$ be the projection and let $\alpha:=\pi_1|_P$. 
Let $v\in V$, $m\in M$ and $F^{ev}(v,m)=a\in A$. Then
\begin{displaymath}
  \im (d\alpha|_{(v,m)})=d\pi_1|_{(v,m)}(dF^{ev}|_{(v,m)}^{-1}(T_aA))
\end{displaymath}
and thus we find
\begin{eqnarray}
  \nonumber
  \codim\im (d\alpha|_{(v,m)})&=&\codim (dF^{ev}|_{(v,m)}^{-1}(T_aA)+\ker d\pi_1|_{(v,m)})\\
  \nonumber
  &=&\codim(dF^{ev}|_{(v,m)}^{-1}(T_aA)+T_mM).\quad
\end{eqnarray}
By Theorem \ref{sard_theorem} 
the set of all regular values 
of $\alpha$ is residual in $V$. 
Fix a regular value $v$ of $\alpha$. 
If $\alpha^{-1}(v)=\emptyset$, 
then $\im(F(v))\cap A=\emptyset$ 
and thus $F(v)$ is transverse to $A$. 
Now assume that there exists $(v,m)\in\alpha^{-1}(v)$. 
Since $\codim\im (d\alpha|_{(v,m)})=0$ we obtain 
\begin{displaymath}
  dF^{ev}|_{(v,m)}^{-1}(T_aA)+T_mM=T_{(v,m)}(V\times M).
\end{displaymath}
We apply $dF^{ev}|_{(v,m)}$ to both sides. Since $F^{ev}$ is transverse to $A$, 
the right hand side contains a complement to $T_aA$. Thus 
$dF^{ev}|_{(v,m)}(T_mM)$ also contains a complement to $T_aA$, 
which means that $F(v)$ is transverse to $A$. The assertion follows.
\end{proof}

\begin{remark}
There exist infinite dimensional versions of this theorem 
and of Sard's theorem (see e.\,g.\,\cite{u}, \cite{sm}, \cite{ab}). 
However for our purpose this theorem is sufficient. 
\end{remark}

\section{Eigenspinors in dimensions $2$ and $3$}
\label{gen_eigenspinor_section}

Let $M$ be a closed oriented spin manifold of dimension $n$ with a fixed spin structure. 
We define $R(M)$ as the set of all smooth Riemannian metrics on $M$. 

\begin{definition}
A subset of $R(M)$ is called generic\index{generic}, if it is open in $R(M)$ with respect 
to the $C^1$-topology and dense in $R(M)$ with respect to all $C^k$-topologies, $k\in\N$.
\end{definition}

\begin{remark}
Let $k\in\N$ and let $A\subset R(M)$. 
If $A$ is open in $R(M)$ with respect to the $C^k$-topology, 
then it is open in $R(M)$ with respect to the $C^m$-topology for all $m>k$. 
If $A$ is dense in $R(M)$ with respect to the $C^k$-topology, 
then it is dense in $R(M)$ with respect to the $C^m$-topology for all $m<k$. 
\end{remark}

It is an important result that the eigenvalues of $D^g$ depend continuously on $g$ 
with respect to the $C^1$-topology. 
Namely Proposition 7.1 in \cite{baer96a} states the following.

\begin{proposition}
\label{cont_eigenvalues}
Let $g$ be a smooth Riemannian metric on $M$ and let $\varepsilon>0$, $\Lambda>0$ 
such that $-\Lambda$, $\Lambda\notin\spec(D^g)$. Write
\begin{displaymath}
  \spec(D^g)\cap(-\Lambda,\Lambda)=\{\lambda_1\leq...\leq\lambda_k\}.
\end{displaymath}
Then there exists an open neighborhood $V$ of $g$ in the set of all smooth Riemannian 
metrics with respect to the $C^1$-topology such that for any 
$h\in V$ we have
\begin{displaymath}
  \spec(D^h)\cap(-\Lambda,\Lambda)=\{\mu_1\leq...\leq\mu_k\}
\end{displaymath}
and $|\lambda_i-\mu_i|<\varepsilon$ for all $i$. 
Here the eigenvalues $\lambda_i$ and $\mu_i$ are repeated 
according to their complex multiplicities.
\end{proposition}


Assume that $n\equiv2,3,4\bmod 8$. Then there exists a quaternionic structure on $\Sigma^gM$, 
which commutes with the Dirac operator $D^g$ and thus the eigenspaces of $D^g$ have even 
complex dimension (see Remark \ref{remark_real_quat}). 
Therefore the following notation introduced in \cite{d} is useful. 

\begin{definition}
A non-zero eigenvalue $\lambda$ of $D^g$ is called simple, if 
\begin{displaymath}
  \dim_{\C}\ker(D^g-\lambda)=\left\{\begin{array}{ll}2, & n\equiv2,3,4\bmod 8 \\ 1, & \textrm{otherwise}\end{array}\right.
\end{displaymath}
\end{definition}

If $n\equiv2,3,4\bmod8$ and $\lambda$ is simple, then we can always choose 
an $L^2$-orthonormal basis of $\ker(D^g-\lambda)$ of the form $\{\psi,J\psi\}$. 

For every $g\in R(M)$ we enumerate the nonzero eigenvalues of $D^g$ in the following way
\begin{displaymath}
  ...\leq\lambda^-_2\leq\lambda^-_1<0<\lambda^+_1\leq\lambda^+_2\leq...\,.
\end{displaymath}
Here all the non-zero eigenvalues are repeated by their complex multiplicities, 
while $\dim\ker(D^g)\geq0$ is arbitrary. 
For $m\in\N\setminus\{0\}$ we define 
\begin{eqnarray*}
  S_m(M)&:=&\{g\in R(M)\,|\lambda^{\pm}_1,...,\lambda^{\pm}_m\textrm{ are simple}\}\\
  N_m(M)&:=&\{g\in R(M)\,|\textrm{all eigenspinors to }\lambda^{\pm}_1,...,\lambda^{\pm}_m\textrm{ are nowhere zero}\}.
\end{eqnarray*}

Using Proposition \ref{cont_eigenvalues} together with a study 
of the change of the non-zero eigenvalues under conformal deformations 
of the metric M.\,Dahl could prove the following result. 

\begin{theorem}[\cite{d}]
\label{gen_simple}
Let $M$ be a closed spin manifold of dimension $2$ or $3$ and let $m\in\N\setminus\{0\}$. 
Then for every $g\in R(M)$ the subset $S_m(M)\cap[g]$ is generic in $[g]$.
\end{theorem}


In order to prove Theorem \ref{main_theorem} it is sufficient to prove the following. 
\begin{theorem}
\label{theorem_no_zeros}
Let $M$ be a closed connected spin manifold of dimension $2$ or $3$ and let $m\in\N\setminus\{0\}$. 
Then for every $g\in R(M)$ the subset $N_m(M)\cap[g]$ is generic in $[g]$. 
\end{theorem}

\begin{remark}
On a closed oriented surface $M$ of genus $2$ 
there exist spin structures, such that for every Riemannian 
metric on $M$ there exist harmonic spinors 
(see Proposition 2.3 in \cite{hit74}). 
By the result on the zero set of positive harmonic spinors 
in Theorem \ref{theorem_surfaces} we find that 
Theorem \ref{theorem_no_zeros} in the case $n=2$ is not true 
for harmonic spinors. 
\end{remark}

In the proof we will use the following unique continuation theorem 
due to Aronszajn (\cite{ar}, quoted from \cite{baer97}).
\begin{theorem}
\label{unique_cont}
Let $M$ be a connected Riemannian manifold and let $\Sigma$ be a vector 
bundle over $M$ with a connection $\nabla$. Let $P$ be an operator of 
the form $P=\nabla^*\nabla+P_1+P_0$ acting on sections of $\Sigma$, where 
$P_1$, $P_0$ are differential operators of order $1$ and $0$ respectively. 
Let $\psi$ be a solution to $P\psi=0$. If there exists a point, at which 
$\psi$ and all derivatives of $\psi$ of any order vanish, 
then $\psi$ vanishes identically.
\end{theorem}

We will also use Theorem \ref{param_transvers} for families of spinors 
parametrized by Riemannian metrics. A basic problem is that $[g]$ is not 
a smooth manifold. If one replaces $[g]$ by the space of all $k$ times 
continuously differentiable metrics, $k\geq1$, which are conformal to $g$, 
then we see from formula (\ref{connection_local}) that the coefficients 
of the Dirac operator are not smooth in general. 
In this case we cannot expect that the eigenspinors are smooth and that 
Theorem \ref{unique_cont} remains valid. 
In order to get around this problem we will use finite dimensional 
manifolds of smooth Riemannian metrics of the form 
\begin{displaymath}
  V_{f_1...f_r}:=\Big\{\Big(1+\sum_{i=1}^r t_i f_i\Big)g\,\Big|t_1,...,t_r\in\R\Big\}\cap R(M),
\end{displaymath}
where $r\in\N\setminus\{0\}$ and $f_1,...,f_r\in C^{\infty}(M,\R)$. 

Our first aim is to construct a map, 
which associates to a Riemannian metric $h\in[g]$ 
an eigenspinor of $D^{g,h}$ in a continuous way. 

\begin{lemma}
\label{lemma_F_psi}
Let $k,m\in\N\setminus\{0\}$. Let $g\in S_m(M)$ and equip $[g]$ with the $C^k$-topology. 
Let $\lambda\in\{\lambda^{\pm}_1(g),...,\lambda^{\pm}_m(g)\}$ 
and let $\psi$ be an eigenspinor of $D^g$ corresponding to $\lambda$. 
Then there exists an open neighborhood $V\subset[g]$ of $g$
and a map $F_{\psi}$: $V\to\spinor{M}{g}{\infty}$ 
such that for every $h\in V$ the spinor 
$F_{\psi}(h)$ is an eigenspinor of $D^{g,h}$ and such that 
the map $F^{ev}_{\psi}$: $V\times M\to\Sigma^gM$ 
defined by $F^{ev}_{\psi}(h,x):=F_{\psi}(h)(x)$ is continuous 
and such that for all functions $f_1,...,f_r\in C^{\infty}(M,\R)$ 
the restriction $F^{ev}_{\psi}|_{V_{f_1...f_r}\times M}$ 
is differentiable. 
\end{lemma}

\begin{proof}
Let $V\subset[g]$ be an open convex neighborhood of $g$, 
which is contained in $S_m(M)\cap[g]$. 
Let $h\in V$ and let $I\subset\R$ be an open interval containing $[0,1]$ 
such that for every $t\in I$ the tensor field $g_t:=g+t(h-g)$ 
is a Riemannian metric on $M$. 
Let $\lambda\in\{\lambda^{\pm}_1(g),...,\lambda^{\pm}_m(g)\}$. 
Then there exists exactly one real-analytic 
family $t\mapsto\lambda_t$ of eigenvalues of $D^{g,g_t}$ such that $\lambda_0=\lambda$. 
Furthermore there exists a real-analytic family $\psi_t$ of spinors, 
such that $\{\psi_t\}$ respectively $\{\psi_t,J\psi_t\}$ 
forms an $L^2$-orthonormal basis of $\ker(D^{g,g_t}-\lambda_t)$ for every $t$. 
After possibly shrinking $V$ we may assume that for all metrics $h\in V$ 
and for all $\lambda\in\{\lambda^{\pm}_1(g),...,\lambda^{\pm}_m(g)\}$ 
the families $\lambda_t$ and $\psi_t$ 
are defined for all $t\in[0,1]$. 
Since we are free to 
replace $\psi_t$, $J\psi_t$ by linear combinations 
$a\psi_t+bJ\psi_t$ with $a,b\in\C$, $|a|^2+|b|^2>0$, 
we may assume that $\psi_0=\psi$. 
We define $F_{\psi}(h):=\psi_1$.

Let $h\in V$ and let $k=fg$ for some $f\in C^{\infty}(M,\R)$. 
With $h_t:=h+tk$ consider $F_{\psi}(h_t)$ for small $|t|$ such that $h_t\in V$. 
Since $a_{g,h}$ and $a_{h,h_t}$ commute, 
we have $b_{h,h_t}\circ b_{g,h}=b_{g,h_t}$ 
by Lemma \ref{lemma_agh_ahk}. 
and thus $\overline{\beta}_{h,h_t}\circ\overline{\beta}_{g,h}=\overline{\beta}_{g,h_t}$. 
It follows that 
\begin{displaymath}
  D^{g,h_t}=\overline{\beta}_{h,g}\circ D^{h,h_t}\circ \overline{\beta}_{g,h}.
\end{displaymath}
There exists a real-analytic family $t\mapsto\chi(t)\in\spinor{M}{h}{\infty}$ 
of eigenspinors of $D^{h,h_t}$ with $\chi(0)=\overline{\beta}_{g,h}(F_{\psi}(h))$. 
It follows that 
\begin{displaymath}
  F_{\psi}(h_t)=\overline{\beta}_{h,g}(\chi(t)).
\end{displaymath}
By taking $k:=h-h'$ for $h,h'\in V$ we find that $F^{ev}_{\psi}$ is continuous. 
Furthermore for all functions $f_1,...,f_r\in C^{\infty}(M,\R)$ 
we can use this equation to see that the restriction $F^{ev}_{\psi}|_{V_{f_1...f_r}\times M}$ 
is differentiable. 
\end{proof}

The strategy for the proof of Theorem \ref{theorem_no_zeros} is based 
on the following remark.

\begin{remark}
\label{remark_transverse}
Let $A\subset\Sigma^gM$ be the zero section.
Since the dimension of the total space $\Sigma^gM$ of the spinor bundle is 
\begin{displaymath}
  \dim\Sigma^gM=n+2^{1+[n/2]}>2n=\dim M+\dim A,
\end{displaymath}
a map $f$: $M\to\Sigma^gM$ is transverse to $A$ if and only if 
$f^{-1}(A)=\emptyset$.
\end{remark}

If for every 
$\lambda_i\in\{\lambda^{\pm}_1,...,\lambda^{\pm}_m\}$ we choose an 
eigenspinor $\psi_i$ and define the map 
\begin{displaymath}
  F_{\psi_i}:\quad V_{f_1...f_r}\to\spinor{M}{g}{\infty}
\end{displaymath}
as in Lemma \ref{lemma_F_psi}, 
then by this remark $V_{f_1...f_r}\cap N_m(M)$ is the set of all Riemannian metrics 
$h\in V_{f_1...f_r}$ such that all the $F_{\psi_i}(h)$ are transverse to the zero section. 
Therefore in order to prove that $N_m(M)\cap[g]$ is dense in $[g]$ 
we would like to apply Theorem \ref{param_transvers}. 
Our aim is then to show that a suitable restriction of 
the map $F_{\psi}^{ev}$ defined as in Lemma \ref{lemma_F_psi} 
is transverse to the zero section. 

Let $p\in M$ with $\psi(p)=0$. We have a canonical decomposition of the tangent space
\begin{displaymath}
  T_{\psi(p)}\Sigma^gM\cong \Sigma^g_pM\oplus T_pM
\end{displaymath}
and thus
\begin{displaymath}
  dF_{\psi}^{ev}|_{(g,p)}:\quad T_gV_{f_1...f_r}\oplus T_pM\to\Sigma^g_pM\oplus T_pM.
\end{displaymath}

For a given $h\in V_{f_1...f_r}$ we will write $g_t=g+t(h-g)$ and 
\begin{displaymath}
  \psi_t:=F_{\psi}(g_t),\quad \frac{d\psi_t(x)}{dt}\big|_{t=0}:=\pi_1(dF_{\psi}^{ev}|_{(g,x)}(h-g,0)).
\end{displaymath}
Then it follows that
\begin{equation}
  \label{derivative_of_eigenspinor_equation}
  0=\big(\frac{d}{dt}D^{g,g_t}\big|_{t=0}-\frac{d\lambda_t}{dt}\big|_{t=0}\big)\psi
  +\big(D^g-\lambda\big)\frac{d\psi_t}{dt}\big|_{t=0}.
\end{equation}

\begin{remark}
\label{remark_real_basis}
Let $\varphi\in\Sigma^g_pM$ and $X$, $Y\in T_pM$. If one polarizes 
the identity
\begin{displaymath}
  \langle X\cdot\varphi,X\cdot\varphi\rangle=g(X,X)\langle\varphi,\varphi\rangle,
\end{displaymath}
then one obtains
\begin{equation}
  \label{Xgamma_Ygamma}
  \Re\langle X\cdot\varphi,Y\cdot\varphi\rangle=g(X,Y)\langle\varphi,\varphi\rangle.
\end{equation}
Since Clifford multiplication with vectors is antisymmetric, it follows that 
$\Re\langle X\cdot\varphi,\varphi\rangle=0$. Let $\varphi\neq0$ and let 
$(e_i)_{i=1}^n$ be an orthonormal basis of $T_pM$. 
It follows that for $n=2$ the spinors
\begin{displaymath}
  \varphi,\,e_1\cdot\varphi,\,e_2\cdot\varphi,\,e_1\cdot e_2\cdot\varphi
\end{displaymath}
form an orthogonal basis of $\Sigma^g_pM$ with respect to the real 
scalar product $\Re\langle.,.\rangle$. Similarly for $n=3$ the spinors 
\begin{displaymath}
  \varphi,\,e_1\cdot\varphi,\,e_2\cdot\varphi,\,e_3\cdot\varphi
\end{displaymath}
form an orthogonal basis of $\Sigma^g_pM$ with respect 
to $\Re\langle.,.\rangle$.
\end{remark}

The following rather long lemma is the most important step in showing 
that a suitable restriction of $F^{ev}_{\psi}$ is transverse to the 
zero section.

\begin{lemma}
\label{lemma_transverse}
Let $n\in\{2,3\}$ and $m\in\N\setminus\{0\}$. Let $g\in S_m(M)$ and 
let $\lambda\in\{\lambda^{\pm}_1(g),...,\lambda^{\pm}_m(g)\}$. 
Let $\psi$ be an eigenspinor of $D^g$ corresponding to $\lambda$ 
and let $p\in M$ with $\psi(p)=0$. 
Then there exist $f_1,...,f_4\in C^{\infty}(M,\R)$ such that the map 
$F_{\psi}^{ev}$: $V_{f_1...f_4}\times M\to\Sigma^gM$ satisfies
\begin{displaymath}
  \pi_1(dF_{\psi}^{ev}|_{(g,p)}(T_gV_{f_1...f_4}\oplus\{0\}))=\Sigma^g_pM.
\end{displaymath}
\end{lemma}

\begin{proof}
Assume that the claim is wrong. 
Then there exists $\varphi\in\Sigma^g_pM\setminus\{0\}$ 
such that for all $f\in C^{\infty}(M,\R)$ we have 
\begin{displaymath}
  0=\Re\langle\pi_1(dF_{\psi}^{ev}|_{(g,p)}(fg,0)),\varphi\rangle
  =\Re\langle\frac{d\psi_t}{dt}|_{t=0}(p),\varphi\rangle.
\end{displaymath}
From the formula (\ref{greens_function}) for Green's function it follows that
\begin{displaymath}
  0=\Re \int_{M\setminus\{p\}} \big\langle \big(D^g-\lambda\big)\frac{d\psi_t}{dt}\big|_{t=0}, G^g_{\lambda}(.,p)\varphi \big\rangle \dv^g 
  +\Re \big\langle P\big(\frac{d\psi_t}{dt}\big|_{t=0}\big)(p), \varphi \big\rangle.
\end{displaymath}
Since $\lambda$ is a simple eigenvalue, all spinors in $\ker(D^g-\lambda)$ 
vanish at $p$. 
Thus the last term vanishes. 
By (\ref{green_function2}) and (\ref{derivative_of_eigenspinor_equation}) we have
\begin{eqnarray*}
  0&=&-\Re \int_{M\setminus\{p\}} \big\langle \big(\frac{d}{dt}D^{g,g_t}\big|_{t=0}
  -\frac{d\lambda_t}{dt}\big|_{t=0}\big)\psi, G^g_{\lambda}(.,p)\varphi \big\rangle \dv^g\\
  &=&-\Re \int_{M\setminus\{p\}} \big\langle \frac{d}{dt}D^{g,g_t}\big|_{t=0}\psi, G^g_{\lambda}(.,p)\varphi \big\rangle \dv^g
\end{eqnarray*}
for all $f\in C^{\infty}(M,\R)$. 
If we use the formula (\ref{dirac_derivative_conform}) 
for the derivative of the Dirac operator and
\begin{displaymath}
  \grad^g(f)\cdot\psi=(D^g-\lambda)(f\psi)
\end{displaymath}
it follows that
\begin{eqnarray*}
  0&=&\frac{1}{2}\,\Re \int_{M\setminus\{p\}} \lambda f \big\langle \psi, G^g_{\lambda}(.,p)\varphi \big\rangle \dv^g\\
  &&{}+\frac{1}{4}\,\Re \int_{M\setminus\{p\}} \big\langle (D^g-\lambda)(f\psi), G^g_{\lambda}(.,p)\varphi \big\rangle \dv^g.
\end{eqnarray*}
Using the definition of Green's function and that all spinors in $\ker(D^g-\lambda)$ 
vanish at $p$, we find that
\begin{eqnarray*}
  0&=&\frac{1}{2}\,\Re \int_{M\setminus\{p\}} \lambda f \big\langle \psi, G^g_{\lambda}(.,p)\varphi \big\rangle \dv^g
  +\frac{1}{4}\,\Re \big\langle (f\psi)(p)-P(f\psi)(p),\varphi \big\rangle\\
  &=&\frac{1}{2}\,\Re \int_{M\setminus\{p\}} \lambda f \big\langle\psi, G^g_{\lambda}(.,p)\varphi \big\rangle \dv^g
\end{eqnarray*}
for all $f\in C^{\infty}(M,\R)$. Since we have $\lambda\neq0$ it follows that 
$\Re\langle\psi,G^g_{\lambda}(.,p)\varphi\rangle$ vanishes identically on $M\setminus\{p\}$. 

Our aim is now to conclude that all the derivatives of $\psi$ at the point $p$ vanish. 
Then by Theorem \ref{unique_cont} it follows that $\psi$ is identically zero, which is a contradiction. 
In order to show this we choose a local parametrization $\rho$: $V\to U$ of $M$ 
by Riemannian normal coordinates, where $U\subset M$ is an open neighborhood of $p$, 
$V\subset\R^n$ is an open neighborhood of $0$ and $\rho(0)=p$. 
Furthermore let 
\begin{displaymath}
  \beta:\quad \Sigma\R^n|_V\to\Sigma^gM|_U,\quad 
  A:\quad \spinor{M|_U}{g}{\infty}\to\spinor{\R^n|_V}{}{\infty}
\end{displaymath}
denote the maps which send a spinor to its 
corresponding spinor in the Bourguignon-Gauduchon trivialization 
defined in Section \ref{bourg-triv_section}.
We show by induction that $\nabla^r A\psi(0)=0$ for all $r\in\N$, where $\nabla$ 
denotes the covariant derivative on $\Sigma\R^n$. 
The case $r=0$ is clear. 

Let $r\geq1$ and assume that we have $\nabla^s A\psi(0)=0$ for all $s\leq r-1$. 
Let $(E_i)_{i=1}^n$ be the standard basis of $\R^n$. 
First consider the case $n=2$. 
In the Bourguignon-Gauduchon trivialization we have
\begin{displaymath}
  A\psi(x)=\frac{1}{r!}\sum_{j_1,...,j_r=1}^2 x_{j_1}...x_{j_r}\nabla_{E_{j_1}}...\nabla_{E_{j_r}}A\psi(0)+O(|x|^{r+1})
\end{displaymath}
by Taylor's formula and
\begin{displaymath}
  A (G^g_{\lambda}(.,p)\varphi)(x)=-\frac{1}{2\pi|x|^2} x\cdot\gamma-\frac{\lambda}{2\pi}\ln|x|\gamma+O(|x|^0)
\end{displaymath}
by Theorem \ref{theorem_green}, where $\gamma:=\beta^{-1}\varphi\in\Sigma_n$ is the constant spinor 
on $\R^n$ corresponding to $\varphi$. It follows that
\begin{eqnarray*}
  0&=&-2\pi r!|x|^2 \Re\langle A (G^g_{\lambda}(.,p)\varphi)(x),A\psi(x)\rangle\\
  &=&\sum_{i,j_1,...,j_r=1}^2 x_ix_{j_1}...x_{j_r}
  \Re\langle E_i\cdot\gamma,\nabla_{E_{j_1}}...\nabla_{E_{j_r}}A\psi(0) \rangle\\
  &&{}+\lambda\sum_{j_1,...,j_r=1}^2 x_{j_1}...x_{j_r}|x|^2\ln|x|
  \Re\langle\gamma,\nabla_{E_{j_1}}...\nabla_{E_{j_r}}A\psi(0)\rangle+O(|x|^{r+2})
\end{eqnarray*}
and therefore
\begin{eqnarray}
  \nonumber
  0&=&\Re\langle E_i\cdot\gamma,\nabla_{E_{j_1}}...\nabla_{E_{j_r}}A\psi(0) \rangle\\
  \label{taylor_dim2_1}
  &&{}+\sum_{s=1}^r \Re\langle E_{j_s}\cdot\gamma,\nabla_{E_i}\nabla_{E_{j_1}}...\widehat{\nabla_{E_{j_s}}}...\nabla_{E_{j_r}}A\psi(0) \rangle\\
  \label{taylor_dim2_2}
  0&=&\Re\langle\gamma,\nabla_{E_{j_1}}...\nabla_{E_{j_r}}A\psi(0)\rangle
\end{eqnarray}
for all $j_1,...,j_r,i\in\{1,2\}$, where the hat means that the operator is left out. 
By (\ref{taylor_dim2_2}) and Remark \ref{remark_real_basis} 
there exist $a_{j_1,...,j_r,k}$, $b_{j_1,...,j_r}\in\R$ such that 
\begin{displaymath}
  \nabla_{E_{j_1}}...\nabla_{E_{j_r}}A\psi(0)=\sum_{k=1}^2 a_{j_1,...,j_r,k}E_k\cdot\gamma
  +b_{j_1,...,j_r}E_1\cdot E_2\cdot\gamma.
\end{displaymath}
Observe that the coefficients $a_{j_1,...,j_r,k}$ are symmetric in the first $r$ indices. 
We insert this into (\ref{taylor_dim2_1}) and we obtain
\begin{equation}
  \label{taylor_dim2_3}
  0=a_{j_1,...,j_r,i}+\sum_{k=1}^r a_{i,j_1,...,\widehat{j_k},...,j_r,j_k}
\end{equation}
for all $j_1,...,j_r,i\in\{1,2\}$. On the other hand since $\psi\in\ker(D^g-\lambda)$ 
we find using the induction hypothesis 
\begin{eqnarray}
  \nonumber
  0&=&\lambda\nabla_{E_{j_1}}...\nabla_{E_{j_{r-1}}}A\psi(0)\\
  \nonumber
  &=&\nabla_{E_{j_1}}...\nabla_{E_{j_{r-1}}}\sum_{i=1}^2 E_i\cdot\nabla_{E_i}A\psi(0)\\
  \nonumber
  &=&\sum_{i,k=1}^2 a_{j_1,...,j_{r-1},i,k}E_i\cdot E_k\cdot\gamma
  +\sum_{i=1}^2 b_{j_1,...,j_{r-1},i}E_i\cdot E_1\cdot E_2\cdot\gamma\\
  \nonumber
  &=&-(a_{j_1,...,j_{r-1},1,1}+a_{j_1,...,j_{r-1},2,2})\gamma\\
  \nonumber
  &&{}+(a_{j_1,...,j_{r-1},1,2}-a_{j_1,...,j_{r-1},2,1})E_1\cdot E_2\cdot\gamma\\
  \label{taylor_dim2_4}
  &&{}+b_{j_1,...,j_{r-1},2}E_1\cdot\gamma-b_{j_1,...,j_{r-1},1}E_2\cdot\gamma
\end{eqnarray}
for all $j_1,...,j_{r-1}\in\{1,2\}$. We conclude $b_{j_1,...,j_r}=0$ 
for all $j_1,...,j_r\in\{1,2\}$. 
Next consider $a_{j_1,...,j_r,i}$ with fixed $j_1,...,j_r,i\in\{1,2\}$. 
If we have $j_k=i$ for all $k\in\{1,...,r\}$, then by (\ref{taylor_dim2_3}) we know that 
$a_{j_1,...,j_r,i}=0$. If there exists $k$ such that $j_k\neq i$ it follows from 
the coefficient of $E_1\cdot E_2\cdot\gamma$ in (\ref{taylor_dim2_4}) that
\begin{displaymath}
  a_{i,j_1,...\widehat{j_k}...,j_r,j_k}=a_{j_1,...,j_r,i}.
\end{displaymath}
Again (\ref{taylor_dim2_3}) yields $a_{j_1,...,j_r,i}=0$. 
We conclude that all $a_{j_1,...,j_r,i}$ vanish and that 
$\nabla^r A\psi(0)=0$. This proves the assertion in the case $n=2$.

Next consider $n=3$. 
In the Bourguignon-Gauduchon trivialization we have
\begin{eqnarray*}
  A\psi(x)&=&\frac{1}{r!}\sum_{j_1,...,j_r=1}^3 x_{j_1}...x_{j_r}\nabla_{E_{j_1}}...\nabla_{E_{j_r}}A\psi(0)\\
  &&{}+\frac{1}{(r+1)!}\sum_{j_1,...,j_r,i=1}^3 x_{j_1}...x_{j_r}x_i\nabla_{E_{j_1}}...\nabla_{E_{j_r}}\nabla_{E_i}A\psi(0)+o(|x|^{r+1})
\end{eqnarray*}
by Taylor's formula and
\begin{displaymath}
  A (G^g_{\lambda}(.,p)\varphi)(x)=-\frac{1}{4\pi|x|^3}x\cdot\gamma+\frac{\lambda}{4\pi|x|}\gamma+o(|x|^{-s})
\end{displaymath}
for every $s>0$ by Theorem \ref{theorem_green}, where $\gamma$ is as above. It follows that
\begin{eqnarray}
  \nonumber
  0&=&-4\pi r!|x|^3 \Re\langle A (G^g_{\lambda}(.,p)\varphi)(x),A\psi(x)\rangle\\
  \nonumber
  &=&\sum_{i,j_1,...,j_r=1}^3 x_{j_1}...x_{j_r}x_i
  \Re\langle E_i\cdot\gamma,\nabla_{E_{j_1}}...\nabla_{E_{j_r}}A\psi(0) \rangle\\
  \nonumber
  &&{}+\frac{1}{r+1}\sum_{i,j_1,...,j_r,m=1}^3 x_{j_1}...x_{j_r}x_i x_m
  \Re\langle E_i\cdot\gamma,\nabla_{E_{j_1}}...\nabla_{E_{j_r}}\nabla_{E_m}A\psi(0) \rangle\\
  \label{taylor_dim3}
  &&{}-\lambda\sum_{j_1,...,j_r=1}^3 x_{j_1}...x_{j_r}|x|^2
  \Re\langle\gamma,\nabla_{E_{j_1}}...\nabla_{E_{j_r}}A\psi(0)\rangle+o(|x|^{r+2}).
\end{eqnarray}
From the first term on the right hand side we obtain
\begin{eqnarray}
  \nonumber
  0&=&\Re\langle E_i\cdot\gamma,\nabla_{E_{j_1}}...\nabla_{E_{j_r}}A\psi(0) \rangle\\
  \label{taylor_dim3_1}
  &&{}+\sum_{s=1}^r \Re\langle E_{j_s}\cdot\gamma,\nabla_{E_i}\nabla_{E_{j_1}}...\widehat{\nabla_{E_{j_s}}}...\nabla_{E_{j_r}}A\psi(0) \rangle.
\end{eqnarray}
for all $j_1,...,j_r,i\in\{1,2,3\}$, where the hat means that the operator is left out. 
Our next aim is to obtain an analogue of (\ref{taylor_dim2_2}) from the second and third term on 
the right hand side of (\ref{taylor_dim3}). It is more difficult than in the case $n=2$, 
since derivatives of both orders $r$ and $r+1$ appear. 
The equation (\ref{dirac_triv}) reads
\begin{eqnarray*}
  \lambda A\psi&=&D^{\geucl}A\psi
  +\sum_{i,j=1}^3 (B^j_i-\delta^j_i)E_i\cdot\nabla_{E_j}A\psi\\
  &&{}+\frac{1}{4}\sum_{i,j,k=1}^3 \widetilde{\Gamma}^k_{ij} E_i\cdot E_j\cdot E_k\cdot A\psi.
\end{eqnarray*}
Using (\ref{B-expansion}), (\ref{gammatilde}) and that $|A\psi(x)|_{\geucl}=O(|x|^r)$ as $x\to0$ we find
\begin{displaymath}
  \lambda A\psi=D^{\geucl}A\psi+O(|x|^{r+1})
\end{displaymath}
and therefore
\begin{equation}
  \label{taylor_dim3_2}
  \nabla_{E_{j_1}}...\nabla_{E_{j_r}}D^{\geucl}A\psi(0)
  =\lambda\nabla_{E_{j_1}}...\nabla_{E_{j_r}}A\psi(0)
\end{equation}
for all $j_1,...,j_r\in\{1,2,3\}$. Using the equation (\ref{nabla_triv}) we find
\begin{displaymath}
  A\nabla^g_{e_i}\psi=\nabla_{E_i}A\psi+O(|x|^{r+1}),\quad
  A\nabla^g_{e_i}\nabla^g_{e_j}\psi=\nabla_{E_i}\nabla_{E_j}A\psi+O(|x|^r)
\end{displaymath}
for all $i,j\in\{1,2,3\}$. Since by definition $d\rho|_x^{-1}(e_i)=E_i+O(|x|^2)$ 
the second term in the local formula (\ref{laplacian_local}) for $\nabla^*\nabla$ 
vanishes at $p$ and therefore we find
\begin{displaymath}
  A\nabla^*\nabla\psi=-\sum_{i=1}^3 \nabla_{E_i}\nabla_{E_i}A\psi+O(|x|^r).
\end{displaymath}
From the Schr\"odinger-Lichnerowicz formula (\ref{schroed_lichn}) it follows that
\begin{displaymath}
  \lambda^2 A\psi-\frac{\scal}{4}A\psi=-\sum_{i=1}^3\nabla_{E_i}\nabla_{E_i}A\psi+O(|x|^r)
\end{displaymath}
and thus
\begin{equation}
  \label{taylor_dim3_3}
  \nabla_{E_{j_1}}...\nabla_{E_{j_{r-1}}}\sum_{i=1}^3\nabla_{E_i}\nabla_{E_i}A\psi(0)=0
\end{equation}
for all $j_1,...,j_{r-1}\in\{1,2,3\}$. Now recall the second and third term on the right hand side 
of (\ref{taylor_dim3})
\begin{eqnarray*}
  0&=&\frac{1}{r+1}\sum_{j_1,...,j_r,i,m=1}^3 x_{j_1}...x_{j_r}x_i x_m
  \Re\langle E_i\cdot\gamma,\nabla_{E_{j_1}}...\nabla_{E_{j_r}}\nabla_{E_m}A\psi(0) \rangle\\
  &&{}-\lambda\sum_{j_1,...,j_r=1}^3 x_{j_1}...x_{j_r}|x|^2
  \Re\langle\gamma,\nabla_{E_{j_1}}...\nabla_{E_{j_r}}A\psi(0)\rangle
\end{eqnarray*}
and let $k_1$, $k_2$, $k_3\in\N$ such that $k_1+k_2+k_3=r$. Then from the coefficient of 
$x_1^{k_1+2}x_2^{k_2}x_3^{k_3}$ we find
\begin{eqnarray*}
  0&=&\Re\langle E_1\cdot\gamma,\nabla_{E_1}^{k_1}\nabla_{E_2}^{k_2}\nabla_{E_3}^{k_3}\nabla_{E_1}A\psi(0)\rangle
  \frac{r!}{(k_1+1)!k_2!k_3!} \hspace{4 em} (I)\\
  &&{}+\Re\langle E_2\cdot\gamma,\nabla_{E_1}^{k_1}\nabla_{E_2}^{k_2-1}\nabla_{E_3}^{k_3}\nabla_{E_1}^2 A\psi(0)\rangle
  \frac{r!k_2}{(k_1+2)!k_2!k_3!} \hspace{2 em} (III)\\
  &&{}+\Re\langle E_3\cdot\gamma,\nabla_{E_1}^{k_1}\nabla_{E_2}^{k_2}\nabla_{E_3}^{k_3-1}\nabla_{E_1}^2 A\psi(0)\rangle
  \frac{r!k_3}{(k_1+2)!k_2!k_3!} \hspace{2 em}(IV)\\
  &&{}-\lambda\Re\langle\gamma,\nabla_{E_1}^{k_1}\nabla_{E_2}^{k_2}\nabla_{E_3}^{k_3}A\psi(0)\rangle
  \frac{r!}{k_1!k_2!k_3!} \hspace{8.7 em} (V)\\
  &&{}-\lambda\Re\langle\gamma,\nabla_{E_1}^{k_1}\nabla_{E_2}^{k_2-2}\nabla_{E_3}^{k_3}\nabla_{E_1}^2 A\psi(0)\rangle
  \frac{r!k_2(k_2-1)}{(k_1+2)!k_2!k_3!} \hspace{3.4 em} (VII)\\
  &&{}-\lambda\Re\langle\gamma,\nabla_{E_1}^{k_1}\nabla_{E_2}^{k_2}\nabla_{E_3}^{k_3-2}\nabla_{E_1}^2 A\psi(0)\rangle
  \frac{r!k_3(k_3-1)}{(k_1+2)!k_2!k_3!} \hspace{3.2 em} (VIII).\\
\end{eqnarray*}
From the coefficient of $x_1^{k_1}x_2^{k_2+2}x_3^{k_3}$ we find
\begin{eqnarray*}
  0&=&\Re\langle E_1\cdot\gamma,\nabla_{E_1}^{k_1-1}\nabla_{E_2}^{k_2}\nabla_{E_3}^{k_3}\nabla_{E_2}^2 A\psi(0)\rangle
  \frac{r!k_1}{k_1!(k_2+2)!k_3!} \hspace{3.4 em} (II)\\
  &&{}+\Re\langle E_2\cdot\gamma,\nabla_{E_1}^{k_1}\nabla_{E_2}^{k_2}\nabla_{E_3}^{k_3}\nabla_{E_2} A\psi(0)\rangle
  \frac{r!}{k_1!(k_2+1)!k_3!} \hspace{3.1 em} (I)\\
  &&{}+\Re\langle E_3\cdot\gamma,\nabla_{E_1}^{k_1}\nabla_{E_2}^{k_2}\nabla_{E_3}^{k_3-1}\nabla_{E_2}^2 A\psi(0)\rangle
  \frac{r!k_3}{k_1!(k_2+2)!k_3!} \hspace{2.2 em} (IV)\\
  &&{}-\lambda\Re\langle\gamma,\nabla_{E_1}^{k_1-2}\nabla_{E_2}^{k_2}\nabla_{E_3}^{k_3}\nabla_{E_2}^2 A\psi(0)\rangle
  \frac{r!k_1(k_1-1)}{k_1!(k_2+2)!k_3!} \hspace{3.5 em} (VI)\\
  &&{}-\lambda\Re\langle\gamma,\nabla_{E_1}^{k_1}\nabla_{E_2}^{k_2}\nabla_{E_3}^{k_3} A\psi(0)\rangle
  \frac{r!}{k_1!k_2!k_3!} \hspace{8.7 em} (V)\\
  &&{}-\lambda\Re\langle\gamma,\nabla_{E_1}^{k_1}\nabla_{E_2}^{k_2}\nabla_{E_3}^{k_3-2}\nabla_{E_2}^2 A\psi(0)\rangle
  \frac{r!k_3(k_3-1)}{k_1!(k_2+2)!k_3!} \hspace{3.6 em} (VIII).\\
\end{eqnarray*}
From the coefficient of $x_1^{k_1}x_2^{k_2}x_3^{k_3+2}$ we find
\begin{eqnarray*}
  0&=&\Re\langle E_1\cdot\gamma,\nabla_{E_1}^{k_1-1}\nabla_{E_2}^{k_2}\nabla_{E_3}^{k_3}\nabla_{E_3}^2 A\psi(0)\rangle
  \frac{r!k_1}{k_1!k_2!(k_3+2)!} \hspace{3.5 em} (II)\\
  &&{}+\Re\langle E_2\cdot\gamma,\nabla_{E_1}^{k_1}\nabla_{E_2}^{k_2-1}\nabla_{E_3}^{k_3}\nabla_{E_3}^2 A\psi(0)\rangle
  \frac{r!k_2}{k_1!k_2!(k_3+2)!} \hspace{2.2 em} (III)\\
  &&{}+\Re\langle E_3\cdot\gamma,\nabla_{E_1}^{k_1}\nabla_{E_2}^{k_2}\nabla_{E_3}^{k_3}\nabla_{E_3} A\psi(0)\rangle
  \frac{r!}{k_1!k_2!(k_3+1)!} \hspace{3.2 em} (I)\\
  &&{}-\lambda\Re\langle\gamma,\nabla_{E_1}^{k_1-2}\nabla_{E_2}^{k_2}\nabla_{E_3}^{k_3}\nabla_{E_3}^2 A\psi(0)\rangle
  \frac{r!k_1(k_1-1)}{k_1!k_2!(k_3+2)!} \hspace{3.5 em} (VI)\\
  &&{}-\lambda\Re\langle\gamma,\nabla_{E_1}^{k_1}\nabla_{E_2}^{k_2-2}\nabla_{E_3}^{k_3}\nabla_{E_3}^2 A\psi(0)\rangle
  \frac{r!k_2(k_2-1)}{k_1!k_2!(k_3+2)!} \hspace{3.6 em} (VII)\\
  &&{}-\lambda\Re\langle\gamma,\nabla_{E_1}^{k_1}\nabla_{E_2}^{k_2}\nabla_{E_3}^{k_3} A\psi(0)\rangle
  \frac{r!}{k_1!k_2!k_3!} \hspace{8.7 em} (V).
\end{eqnarray*}
We multiply the first equation with $\frac{(k_1+2)!k_2!k_3!}{r!}$, the second equation with $\frac{k_1!(k_2+2)!k_3!}{r!}$ 
and the third equation with $\frac{k_1!k_2!(k_3+2)!}{r!}$ and then add the multiplied equations. 
If we consider the lines with the same Roman numbers separately and use (\ref{taylor_dim3_2}), (\ref{taylor_dim3_3}), 
then we find
\begin{eqnarray*}
  0&=&-2\lambda\Re\langle\gamma,\nabla_{E_1}^{k_1}\nabla_{E_2}^{k_2}\nabla_{E_3}^{k_3}A\psi(0)\rangle\\
  &&{}+\Re\langle E_1\cdot\gamma,\nabla_{E_1}^{k_1}\nabla_{E_2}^{k_2}\nabla_{E_3}^{k_3}\nabla_{E_1}A\psi(0)\rangle k_1\\
  &&{}+\Re\langle E_2\cdot\gamma,\nabla_{E_1}^{k_1}\nabla_{E_2}^{k_2}\nabla_{E_3}^{k_3}\nabla_{E_2}A\psi(0)\rangle k_2\\
  &&{}+\Re\langle E_3\cdot\gamma,\nabla_{E_1}^{k_1}\nabla_{E_2}^{k_2}\nabla_{E_3}^{k_3}\nabla_{E_3}A\psi(0)\rangle k_3\hspace{7 em} (I)\\
  &&{}-\Re\langle E_1\cdot\gamma,\nabla_{E_1}^{k_1}\nabla_{E_2}^{k_2}\nabla_{E_3}^{k_3}\nabla_{E_1}A\psi(0)\rangle k_1\hspace{7 em} (II)\\
  &&{}-\Re\langle E_2\cdot\gamma,\nabla_{E_1}^{k_1}\nabla_{E_2}^{k_2}\nabla_{E_3}^{k_3}\nabla_{E_2}A\psi(0)\rangle k_2\hspace{7 em} (III)\\
  &&{}-\Re\langle E_3\cdot\gamma,\nabla_{E_1}^{k_1}\nabla_{E_2}^{k_2}\nabla_{E_3}^{k_3}\nabla_{E_3}A\psi(0)\rangle k_3\hspace{7 em} (IV)\\
  &&{}-\lambda\Re\langle\gamma,\nabla_{E_1}^{k_1}\nabla_{E_2}^{k_2}\nabla_{E_3}^{k_3}A\psi(0)\rangle\sum_{i=1}^3(k_i+2)(k_i+1)\hspace{3 em} (V)\\
  &&{}+\lambda\Re\langle\gamma,\nabla_{E_1}^{k_1}\nabla_{E_2}^{k_2}\nabla_{E_3}^{k_3}A\psi(0)\rangle k_1(k_1-1)\hspace{6.8 em} (VI)\\
  &&{}+\lambda\Re\langle\gamma,\nabla_{E_1}^{k_1}\nabla_{E_2}^{k_2}\nabla_{E_3}^{k_3}A\psi(0)\rangle k_2(k_2-1)\hspace{6.8 em} (VII)\\
  &&{}+\lambda\Re\langle\gamma,\nabla_{E_1}^{k_1}\nabla_{E_2}^{k_2}\nabla_{E_3}^{k_3}A\psi(0)\rangle k_3(k_3-1)\hspace{6.8 em} (VIII).
\end{eqnarray*}
Therefore we obtain the analogue of (\ref{taylor_dim2_2}) namely
\begin{displaymath}
  \Re\langle \gamma,\nabla_{E_{j_1}}...\nabla_{E_{j_r}}A\psi(0)\rangle=0
\end{displaymath}
for all $j_1,...,j_r\in\{1,2,3\}$. 
Thus there exist $a_{j_1,...,j_r,k}\in\R$ such that 
\begin{displaymath}
  \nabla_{E_{j_1}}...\nabla_{E_{j_r}}A\psi(0)=\sum_{k=1}^3 a_{j_1,...,j_r,k}E_k\cdot\gamma.
\end{displaymath}
Observe that the coefficients $a_{j_1,...,j_r,k}$ are symmetric in the first $r$ indices. 
We insert this into (\ref{taylor_dim3_1}) and we obtain
\begin{equation}
  \label{taylor_dim3_4}
  0=a_{j_1,...,j_r,i}+\sum_{k=1}^r a_{i,j_1,...,\widehat{j_k},...,j_r,j_k}
\end{equation}
for all $j_1,...,j_r,i\in\{1,2,3\}$. On the other hand since $\psi\in\ker(D^g-\lambda)$ 
we find using the induction hypothesis
\begin{eqnarray}
  \nonumber
  0&=&\lambda\nabla_{E_{j_1}}...\nabla_{E_{j_{r-1}}}A\psi(0)\\
  \nonumber
  &=&\nabla_{E_{j_1}}...\nabla_{E_{j_{r-1}}}\sum_{i=1}^3 E_i\cdot\nabla_{E_i}A\psi(0)\\
  \nonumber
  &=&\sum_{i,k=1}^3 a_{j_1,...,j_{r-1},i,k}E_i\cdot E_k\cdot\gamma\\
  \nonumber
  &=&-\sum_{i=1}^3 a_{j_1,...,j_{r-1},i,i}\gamma\\
  \label{taylor_dim3_5}
  &&{}+\sum_{i,k=1\atop_{i<k}}^3 (a_{j_1,...,j_{r-1},i,k}-a_{j_1,...,j_{r-1},k,i})E_i\cdot E_k\cdot\gamma
\end{eqnarray}
for all $j_1,...,j_{r-1}\in\{1,2,3\}$. Consider $a_{j_1,...,j_r,i}$ with 
$j_1,...,j_r,i\in\{1,2,3\}$. 
If $j_k=i$ for all $k\in\{1,...,r\}$ then by (\ref{taylor_dim3_4}) we know that 
$a_{j_1,...,j_r,i}=0$. If there exists $k$ such that $j_k\neq i$ it follows from 
the coefficient of $E_{j_k}\cdot E_i\cdot\gamma$ in (\ref{taylor_dim3_5}) that
\begin{displaymath}
  a_{i,j_1,...\widehat{j_k}...,j_r,j_k}=a_{j_1,...,j_r,i}.
\end{displaymath}
Again (\ref{taylor_dim3_4}) yields $a_{j_1,...,j_r,i}=0$. 
We conclude that all $a_{j_1,...,j_r,i}$ vanish and that 
$\nabla^r A\psi(0)=0$. This proves the assertion 
in the case $n=3$.
\end{proof}

\begin{remark}
\label{remark_higher_dim}
It is not clear how to prove this lemma for $n\geq4$. Namely the condition 
$\Re\langle\gamma,\nabla_{E_i}A\psi(0)\rangle=0$ for all $i$ leads to
\begin{displaymath}
  \nabla_{E_i}A\psi(0)=\sum_{k=1}^n a_{ik}E_k\cdot\gamma+\sum_{k=1}^n b_{ik}\cdot\gamma
\end{displaymath}
with $a_{ik}\in\R$ and $b_{ik}\in\Cl(n)$. As above it follows that $a_{ik}=-a_{ki}$ 
for all $i$, $k$ and furthermore
\begin{eqnarray*}
  0&=&\lambda A\psi(0)\\
  &=&\sum_{i=1}^n E_i\cdot\nabla_{E_i}A\psi(0)\\
  &=&2\sum_{i,k=1\atop{i<k}}^n a_{ik}E_i\cdot E_k\cdot\gamma+\sum_{i,k=1}^n E_i\cdot b_{ik}\cdot\gamma.
\end{eqnarray*}
But for $n\geq4$ the spinors $E_1\cdot E_2\cdot\gamma$ and $E_3\cdot E_4\cdot\gamma$ 
are not linearly independent in general. 
Thus we cannot conclude immediately that all the $a_{ik}$ vanish.
\end{remark}

\begin{proof}[Proof of Theorem \ref{theorem_no_zeros}]
Let $k,m\in\N\setminus\{0\}$, let $g\in R(M)$ 
and equip $[g]$ with the $C^k$-topology. 
Let $U\subset [g]$ be open. 
In order to show that $N_m(M)\cap[g]$ is dense in $[g]$ 
we have to show that $U\cap N_m(M)$ is not empty. 
Since $S_m(M)\cap[g]$ is dense in $[g]$, 
there exists a metric in $U\cap S_m(M)$, 
which we denote again by $g$. 
Let $\lambda$ be one of the 
eigenvalues $\{\lambda^{\pm}_1,...,\lambda^{\pm}_m\}$ of $D^g$ 
and let $\psi$ be an eigenspinor corresponding to $\lambda$. 
Choose an open neighborhood $V\subset [g]$ of $g$ 
which is contained in $U\cap S_m(M)$ and define 
\begin{displaymath}
  F_{\psi}:\quad V\to\spinor{M}{g}{\infty}
\end{displaymath}
as in Lemma \ref{lemma_F_psi}

Next we show that a suitable restriction of $F_{\psi}^{ev}$ is 
transverse to the zero section of $\Sigma^gM$. 
Let $p\in M$ with $\psi(p)=0$. By Lemma \ref{lemma_transverse} 
there exists an open neighborhood $U_p\subset M$ of $p$ 
and $f_{p,1},...,f_{p,4}\in C^{\infty}(M,\R)$ 
and an open neighborhood $V_p\subset V_{f_{p,1}...f_{p,4}}$ 
of $g$ such that for all $(h,q)\in V_p\times U_p$ 
we have
\begin{equation}
  \label{pr_1}
  \pi_1(dF_{\psi}^{ev}|_{(h,q)}(T_hV_{f_{p,1}...f_{p,4}}\oplus\{0\}))=\Sigma^g_q M.
\end{equation}
Since the zero set of $\psi$ is compact, there exist 
points $p_1$,...,$p_r\in M$ and open neighborhoods $U_{p_i}\subset M$ 
of $p_i$ and $f_{p_i,1},...,f_{p_i,4}\in C^{\infty}(M,\R)$ 
and open neighborhoods $V_{p_i}\subset V_{f_{p_i,1}...f_{p_i,4}}$ of $g$, $1\leq i\leq r$, 
such that the open neighborhoods $U_{p_1}$,...,$U_{p_r}$ cover the zero set 
of $\psi$ and such that for every $i$ the equation (\ref{pr_1})
holds for all $(h,q)\in V_{p_i}\times U_{p_i}$. 
We label the functions $f_{p_i,j}$ by $f_1,...,f_{4r}$ 
and define 
\begin{displaymath}
  V_{f_1...f_{4r}}:=\Big\{\Big(1+\sum_{i=1}^{4r} t_i f_i\Big)g\Big| t_i\in\R\Big\}\cap V.
\end{displaymath}
and $F_{\psi}$: $V_{f_1...f_{4r}}\to\spinor{M}{g}{\infty}$ as in Lemma \ref{lemma_F_psi}. 
Since $M$ is compact there exists 
$C>0$ such that $|\psi|_g\geq C$ on the complement of the union of the 
sets $U_{p_i}$. Thus we can find an open neighborhood $V_{\psi}\subset V_{f_1...f_{4r}}$ 
of $g$ such that the equation (\ref{pr_1}) holds 
for all $(h,q)\in V_{\psi}\times M$. 
It follows that the restriction of $F_{\psi}^{ev}$ to $V_{\psi}\times M$ 
is transverse to the zero section of $\Sigma^gM$. 

Define $W_{\psi}$ as the subset of all $h\in V_{\psi}$ 
such that $F_{\psi}(h)$ is nowhere zero on $M$. 
By Remark \ref{remark_transverse} this condition is 
equivalent to the condition that 
$F_{\psi}(h)$ is transverse to the zero section of $\Sigma^gM$. 
By Theorem \ref{param_transvers} 
the set $W_{\psi}$ is dense in $V_{\psi}$. 
Since the zero section is closed in $\Sigma^gM$ and $F_{\psi}^{ev}$ 
is continuous, the set $W_{\psi}$ is also open in $V_{\psi}$. 
If $h\in W_{\psi}$, then the eigenspinor $F_{\psi}(h)$ 
of $D^{g,h}$ is nowhere zero on $M$ and it corresponds to a 
simple eigenvalue of $D^{g,h}$. 
Thus all the eigenspinors of $D^{g,h}$ corresponding to this eigenvalue 
are nowhere zero on $M$.

For every one of the finitely many eigenvalues
$\lambda^{\pm}_1,...,\lambda^{\pm}_m$ of $D^g$ 
we choose an eigenspinor $\psi$ 
and obtain an open subset $W_{\psi}\subset V$ as above. 
Let $W$ be the intersection of these open subsets $W_{\psi}$. 
It is not empty, since the $W_{\psi}$ are dense in a neighborhood of $g$. 
If $h\in W$, then all the eigenspinors of $D^h$ corresponding to 
the eigenvalues $\lambda^{\pm}_1,...,\lambda^{\pm}_m$ of $D^h$ 
are nowhere zero on $M$. 
Since $W\subset U$ by construction, we have $U\cap N_m(M)\neq\emptyset$. 
Thus $N_m(M)\cap[g]$ is dense in $[g]$. 
We have already seen that $N_m(M)\cap[g]$ is open in $[g]$. 
\end{proof}

\section{Mass endomorphism in dimension $3$}
\label{gen_massendo_section}

Let $M$ be a closed spin manifold of dimension $n$ 
with a fixed spin structure and let 
$R(M)$ be the set of all smooth Riemannian metrics on $M$ 
equipped with the $C^1$-topology. We define 
\begin{displaymath}
  R^*(M):=\{g\in R(M)|\ker(D^g)=0\}.
\end{displaymath}
It follows from the Atiyah-Singer index theorem that there exist spin manifolds $M$ 
and spin structures $\Theta$ on $M$ such that for every Riemannian metric $g$ on $M$ 
there exist harmonic spinors on $(M,g,\Theta)$ (see e.\,g.\,\cite{lm}, Section 3 in \cite{bd}). 
This shows that $R^*(M)$ can be empty. 
Let $p\in M$ and define 
\begin{displaymath}
  R_p(M):=\{g\in R^*(M)|g\textrm{ is flat on an open neighborhood of }p\}.
\end{displaymath}
We consider the case $n=3$. Then by Theorem 1.2 in \cite{m} for every spin structure the 
subset $R^*(M)$ of $R(M)$ is generic. 
Let $g$ be a Riemannian metric on $M$, which is flat on an open neighborhood of $p$. 
It is shown in \cite{adh2} that by changing the metric $g$ on an arbitrarily small 
open subset of $M$ one obtains a metric in $R^*(M)$. 
Thus for every closed spin manifold $M$ of dimension $3$ with a fixed spin structure 
the set $R_p(M)$ is not empty for every $p\in M$. 
For every $g\in R_p(M)$ we define the mass endomorphism $m^g_p\in\End(\Sigma^g_pM)$ 
as in Section \ref{def_mass_endo_section}. Our aim is now to find Riemannian 
metrics on $M$, for which the mass endomorphism does not vanish. Thus we define 
\begin{displaymath}
  S_p(M):=\{g\in R_p(M)|\,m^g_p\neq0\}.
\end{displaymath}
In this section we prove the following result which has been published in \cite{he}.

\begin{theorem}
\label{mass_endo_theorem}
Let $M$ be a closed spin manifold of dimension $3$ with a fixed spin 
structure and let $p\in M$. Then $S_p(M)$ is dense in $R_p(M)$.
\end{theorem}

\begin{remark}
If $M$ is a closed spin manifold of dimension $2$ and if the 
mass endomorphism in $p\in M$ can be defined, 
then by \cite{ahm} it always vanishes. 
On the other hand it is conjectured that Theorem \ref{mass_endo_theorem} holds 
for every closed spin manifold $M$ of dimension $n\geq3$ 
if $R^*(M)$ is not empty. A proof has been proposed in \cite{adhh}. 
The idea is to use surgery methods to construct a Riemannian metric $g$ 
on $M$ which is flat near some point $p$ and has nonzero 
mass endomorphism. Then from perturbation theory it follows 
that $S_p(M)$ is open and dense in $R_p(M)$.
\end{remark}

We will need the following properties of the energy momentum tensor\index{energy momentum tensor} 
which was defined in Section \ref{identification_section}. 

\begin{lemma}
\label{Q_conform}
Let $g$ and $h=e^{2u}g$ be conformally related Riemannian metrics, $u\in C^{\infty}(M,\R)$. 
Let~$\psi\in\spinor{M}{g}{\infty}$ and let $\beta:=\beta_{g,h}$ as 
in Section \ref{identification_section}. Let $f\in C^{\infty}(M,\R)$. 
Then we have 
\begin{displaymath}
  Q_{f\psi}=f^2Q_{\psi},\quad Q_{\beta\psi}=e^u Q_{\psi}.
\end{displaymath}
\end{lemma}

\begin{proof}
The first equation follows, since $\Re\langle Y(f) X\cdot\psi,\psi\rangle=0$ 
for all $X$, $Y\in TM$. 
Recall that we have $b_{g,h}X=e^{-u}X$ for all $X\in TM$ and therefore 
$\beta(X\cdot\psi)=e^{-u}X\cdot\beta\psi$. It follows that
\begin{eqnarray*}
  \langle X\cdot\nabla^h_Y\beta\psi, \beta\psi \rangle_h
  &=&\langle X\cdot\beta(\nabla^g_Y\psi-\frac{1}{2}\,Y\cdot\grad^g(u)\cdot\psi-\frac{1}{2}\,Y(u)\psi), \beta\psi \rangle_h\\
  &=&e^u\langle X\cdot\nabla^g_Y\psi, \psi\rangle_g
  -\frac{e^u}{2}\,\langle X\cdot Y\cdot\grad^g(u)\cdot\psi, \psi \rangle_g\\
  &&{}-\frac{e^u}{2}\,Y(u)\langle X\cdot\psi, \psi \rangle_g.
\end{eqnarray*}
If we add the corresponding equation for $\langle Y\cdot\nabla^h_X\beta\psi, \beta\psi \rangle_h$ 
to this equation and use that $X\cdot Y+Y\cdot X=-2g(X,Y)$, we see that the 
sum of the second terms on the right hand side is purely imaginary. Since the 
third terms on the right hand side are each purely imaginary, the assertion follows. 
\end{proof}

Let $g\in R_p(M)$ and let $k\in\sym{\infty}$ such that $k=0$ on an open neighborhood $U$ of $p$. 
Let $I\subset\R$ be an open interval around $0$ such that for all $t\in I$ 
the tensor field $g_t:=g+tk$ is a Riemannian metric on $M$ and is in $R^*(M)$. 
Then we have $g_t\in R_p(M)$ for all $t\in I$. 
For every $t$ we obtain an isomorphism of spinor bundles 
\begin{displaymath}
  \overline{\beta}_{g,g_t}:\quad\Sigma^gM\to\Sigma^{g_t}M
\end{displaymath}
as in Lemma \ref{definition_beta_bar}. 
Let $\varepsilon>0$ such that $B_{2\varepsilon}(p)\subset U$. 
Recall from Section \ref{def_mass_endo_section} that 
Green's function $G^{g_t}$ for $D^{g_t}$ 
on $M\setminus\{p\}$ can be written as 
\begin{equation}
  \label{green_expansion}
  G^{g_t}(x,p)\overline{\beta}_{g,g_t}\varphi=(A^{g_t})^{-1}(\eta G^{\geucl}(.,0)\gamma_t)(x)+w^{g_t}(x,p)\overline{\beta}_{g,g_t}\varphi,
\end{equation}
where $\eta$: $\R^n\to[0,1]$ is a smooth function such that $\eta\equiv1$ on $B_{\ep}(0)$ and 
$\supp(\eta)\subset B_{2\ep}(0)$ and where $\varphi\in\Sigma^g_pM$ 
and $\gamma_t:=(\beta^{g_t})^{-1}\overline{\beta}_{g,g_t}\varphi$. 
Recall that the mass endomorphism for the metric $g_t$ at $p$ is by definition $m^{g_t}_p=w^{g_t}(p,p)$. 
We define a spinor $w^{g_t}_{\varphi}\in\spinor{M}{g}{\infty}$ by 
\begin{displaymath}
  w^{g_t}_{\varphi}(x):=\overline{\beta}_{g_t,g}w^{g_t}(x,p)\overline{\beta}_{g,g_t}\varphi
\end{displaymath}
and we define 
$m^{g,g_t}\varphi:=w^{g_t}_{\varphi}(p)=\overline{\beta}_{g_t,g}m^{g_t}_p\overline{\beta}_{g,g_t}\varphi$ for all $t$. 

\begin{lemma}
\label{dtm_lemma}
We have 
\begin{displaymath}
  \frac{d}{dt}\, \big\langle m^{g,g_t}\varphi, \varphi \big\rangle \big|_{t=0}
  =\frac{1}{2}\,\int_{M\setminus\{p\}} (k,Q_{G^g(.,p)\varphi})\dv^g,
\end{displaymath}
where $(.,.)$ denotes the standard pointwise inner product of $(2,0)$ tensor fields.
\end{lemma}

\begin{proof}
Since $g_t=g$ on $\supp(\eta)$ for all $t$, the first term on the right hand side 
of (\ref{green_expansion}) is independent of $t$. Therefore we find 
\begin{displaymath}
  \overline{\beta}_{g_t,g} G^{g_t}(x,p)\overline{\beta}_{g,g_t}\varphi
  =(A^g)^{-1}(\eta G^{\geucl}(.,0)\gamma)(x)+w^{g_t}_{\varphi}(x)
\end{displaymath}
for all $x\in M\setminus\{p\}$. We apply $D^{g,g_t}$ defined as in 
Section \ref{identification_section} and since $g_t=g$ on $\supp(\eta)$ 
we find 
\begin{displaymath}
  \delta_p=D^g(A^g)^{-1}(\eta G^{\geucl}(.,0)\gamma)+D^{g,g_t}w^{g_t}_{\varphi}.
\end{displaymath}
We take the derivative with respect to $t$ at $t=0$ and we obtain 
\begin{displaymath}
  0=\big(\frac{d}{dt}D^{g,g_t}\big|_{t=0}\big) w^g_{\varphi}+D^g\big(\frac{d}{dt}w^{g_t}_{\varphi}\big|_{t=0}\big).
\end{displaymath}
By the formula for Green's function (\ref{greens_function}) 
and since $\ker(D^g)=0$, we find that 
\begin{eqnarray*}
  \frac{d}{dt}\, \big\langle m^{g,g_t}\varphi, \varphi \big\rangle \big|_{t=0}
  &=&\big\langle \frac{d}{dt} w^{g_t}_{\varphi}(p)\big|_{t=0},\varphi \big\rangle \big|_{t=0}\\
  &=&\int_{M\setminus\{p\}} \big\langle D^g\big(\frac{d}{dt} w^{g_t}_{\varphi}\big|_{t=0} \big),G^g(.,p){\varphi} \big\rangle \dv^g\\
  &=&-\int_{M\setminus\{p\}} \big\langle \frac{d}{dt} D^{g,g_t}\big|_{t=0} w^g_{\varphi}, G^g(.,p){\varphi} \big\rangle \dv^g
\end{eqnarray*}
Recall that if $(e_i)_{i=1}^n$ is a local orthonormal frame, then by (\ref{dirac_derivative}) we have locally 
\begin{displaymath}
  \frac{d}{dt}D^{g,g_t}\big|_{t=0} w^g_{\varphi}=-\frac{1}{2}\sum_{i=1}^n e_i\cdot\nabla^g_{a_{g,k}(e_i)} w^g_{\varphi}
  -\frac{1}{4}\sum_{i=1}^n \div^g(k)(e_i)e_i\cdot w^g_{\varphi}.
\end{displaymath}
Since $k=0$ on $B_{2\ep}(p)$ this spinor vanishes on $B_{2\ep}(p)$. 
On the other hand the spinors $w^g_{\varphi}$ and $G^g(.,p)\varphi$ coincide on $M\setminus B_{2\ep}(p)$. 
Thus we find 
\begin{displaymath}
  \frac{d}{dt}\, \big\langle m^{g,g_t}\varphi, \varphi \big\rangle \big|_{t=0}
  =-\int_{M\setminus B_{2\ep}(p)} \big\langle \frac{d}{dt} D^{g,g_t}\big|_{t=0} G^g(.,p)\varphi, G^g(.,p)\varphi \big\rangle \dv^g.
\end{displaymath}
The assertion is now obtained in the way that (\ref{lambda_derivative2}) 
was obtained from (\ref{lambda_derivative1}).
\end{proof}

\begin{proof}[Proof of Theorem \ref{mass_endo_theorem}]
Assume that the claim is wrong. Then there is an open subset $Q\subset R_p(M)$ 
such that for all $g\in Q$ we have $m^g_p=0$. Let $g\in Q$ be flat on an open 
neighborhood $U\subset M$ of $p$. By a small deformation of $g$ on $M\setminus U$ we obtain 
a Riemannian metric, which is not conformally flat on $M\setminus U$ and is in $Q$. 
We denote the deformed Riemannian metric again by $g$. Thus there exists 
an open subset $V\subset M\setminus U$, such that $g$ is nowhere conformally flat on $V$.

Let $k\in\sym{\infty}$ such that $k=0$ on an open neighborhood of $p$ and let $I\subset\R$ 
be an open interval around $0$ such that for all $t\in I$ the tensor field $g_t:=g+tk$ 
is a Riemannian metric on $M$ and is in $Q$. 
By our assumption we have $m^{g_t}_p=0$ for all $t\in I$. For every $\varphi\in\Sigma^g_pM$ 
it follows that $\int_{M\setminus\{p\}} (k,Q_{G^g(.,p)\varphi})\dv^g=0$ by Lemma \ref{dtm_lemma}. 
Since this holds for every $k\in\sym{\infty}$ which is zero on an open neighborhood of $p$, 
it holds for $k=fQ_{G^g(.,p)\varphi}$, where $f$ is an arbitrary smooth function with 
$p\notin\supp(f)$. We conclude that $Q_{G^g(.,p)\varphi}=0$ on $M\setminus\{p\}$. 
Let 
\begin{displaymath}
  W:=\{x\in M\setminus\{p\}|G^g(x,p)\varphi\neq0\}.
\end{displaymath}
Then $W$ is an open subset of $M$. 
The function $u$: $W\to\R$, defined by 
\begin{displaymath}
  u(x):=\ln|G^g(x,p)\varphi|_g
\end{displaymath}
is smooth on $W$, 
and thus $h:=e^{2u}g$ defines a Riemannian metric on $W$. 
We obtain an isomorphism of spinor bundles 
\begin{displaymath}
  \beta_{g,h}:\quad\Sigma^gW\to\Sigma^hW,
\end{displaymath}
which is a fibrewise isometry as in Lemma \ref{definition_beta}. The spinor 
\begin{displaymath}
  \psi:=e^{-u}\beta_{g,h}G^g(.,p)\varphi\in\spinor{W}{h}{\infty}
\end{displaymath}
satisfies $D^h\psi=0$ by (\ref{dirac_conform}). By definition of $u$ we have 
$|\psi|_h\equiv1$ on $W$. Furthermore by Lemma \ref{Q_conform} 
we find $Q_{\psi}=0$ on $W$. 

Let $(e_i)_{i=1}^3$ be a local orthonormal frame of $TW$ defined on an open 
subset $S\subset W$. Then for every $x\in S$ 
the system
\begin{displaymath}
  \{\psi(x),e_1\cdot\psi(x),e_2\cdot\psi(x),e_3\cdot\psi(x)\}
\end{displaymath}
is a real basis of $\Sigma^h_xW$, which is orthonormal with respect to the 
real scalar product $\Re\langle.,.\rangle$ (see Remark \ref{remark_real_basis}). 
Since $|\psi|_h$ is constant on $W$, we have 
\begin{displaymath}
  \Re\langle\nabla^h_X\psi,\psi\rangle=0
\end{displaymath}
for all $X\in TM$. 
It follows that there exists an endomorphism $A$ of $TW$ 
such that for all $X\in TW$ we have $\nabla^h_X\psi=A(X)\cdot\psi$. 

We will now show that in each fibre the endomorphism~$A$: $T_xW\to T_xW$ 
is symmetric with respect to $h$ using an observation from \cite{amm}. 
Since $\omega_{\C}=\id$ we have $e_1\cdot e_2\cdot e_3=-\id$. 
In the following we abbreviate~$\psi:=\psi(x)$. 
Using (\ref{Xgamma_Ygamma}) we calculate
\begin{eqnarray*}
  h(Ae_2,e_1)
  &=&\Re\langle Ae_2\cdot\psi,e_1\cdot\psi\rangle\\
  &=&\Re\langle\nabla^h_{e_2}\psi,e_1\cdot\psi\rangle\\
  &=&-\Re\langle e_2\cdot\nabla^h_{e_2}\psi,e_3\cdot\psi\rangle\\
  &=&\Re\langle e_1\cdot\nabla^h_{e_1}\psi,e_3\cdot\psi\rangle
  +\Re\langle e_3\cdot\nabla^h_{e_3}\psi,e_3\cdot\psi\rangle\\
  &=&-\Re\langle e_2\cdot\nabla^h_{e_1}\psi,\psi\rangle\\
  &=&\Re\langle Ae_1\cdot\psi,e_2\cdot\psi\rangle\\
  &=&h(e_2,Ae_1).
\end{eqnarray*}
Similarly we obtain~$h(Ae_1,e_3)=h(e_1,Ae_3)$ 
and~$h(Ae_2,e_3)=h(e_2,Ae_3)$, 
i.\,e.\,$A$ is symmetric with respect to~$h$. 
Therefore, we can choose the basis vectors~$e_1$,~$e_2$,~$e_3$ 
as eigenvectors of $A$, such that $Ae_j=\lambda_je_j$, where $\lambda_j\in\R$. 
It follows that
\begin{displaymath}
  0=\Re\big\langle \psi,e_j\cdot\nabla^h_{e_j}\psi\big\rangle
  =-\lambda_j.
\end{displaymath}
Hence,~$\psi$ is a parallel spinor on~$W$. 
By \cite{fr00} the Riemannian manifold~$(W,h)$ is Ricci flat. 
Since~$\dim W=3$ it follows that~$(W,h)$ is flat. 
However, by definition of the metric $g$ there exists an open subset 
$V\subset M\setminus U$, such that $g$ is nowhere conformally flat 
on $V$. By Theorem \ref{unique_cont} the spinor 
$G^g(.,p)\varphi$ is not identically zero on $V$. Thus on 
$V\cap W\neq\emptyset$ the metric $h$ cannot be flat. 
This is a contradiction.
\end{proof}

\begin{appendix}
\chapter{Analytic perturbation theory}
\label{perturb_theory}

Results in perturbation theory 
show that if a deformation of the Riemannian metric depends real-analytically 
on a parameter, then for a given eigenspinor $\psi$ of $D^g$ one obtains a real 
analytic family of eigenspinors for the perturbed Dirac operators which 
extends $\psi$. There is a short note concerning this fact in \cite{bg}, but 
we will give a more precise explanation here following the book \cite{k}.

In this section $(X,\|.\|_X)$, $(Y,\|.\|_Y)$ will be complex Banach spaces. 
The space of bounded linear operators $X\to Y$ is denoted by $B(X,Y)$. 
It is equipped with the usual operator norm $\|.\|_{B(X,Y)}$. 
The space of closed linear operators $X\to Y$ is denoted by $C(X,Y)$. 
We will write $C(X):=C(X,X)$. The dual space of $X$ is denoted by $X^*$. 
Furthermore $\Omega\subset\C$ will always denote an open and connected subset 
of the complex plane and $I\subset\R$ will denote an open interval.

\begin{definition}
A family of vectors $u(z)\in X$ which depend on $z\in\Omega$ 
is called holomorphic on $\Omega$, if for each $z_0\in\Omega$ there exists 
$u'\in X$ such that
\begin{displaymath}
  \lim_{h\to0}\Big\|\frac{u(z_0+h)-u(z_0)}{h}-u'\Big\|_X=0.
\end{displaymath}
\end{definition}

The following theorem (\cite{k}, p. 139) gives a criterion for holomorphy of vectors.

\begin{theorem}
\label{hol_vector_theorem}
Let $u(z)\in X$ be a family of vectors which depend on $z\in\Omega$. 
If for all $f\in X^*$ the function $f(u(z))$ is holomorphic on $\Omega$, 
then $u(z)$ is holomorphic on $\Omega$.
\end{theorem}

\begin{proof}
Let $\Gamma$ be a positively oriented circle in $\Omega$. 
The Cauchy integral formula gives
\begin{displaymath}
  f(u(z))=\frac{1}{2\pi i}\oint_{\Gamma}\frac{f(u(\zeta))}{\zeta-z}d\zeta
\end{displaymath}
for all $z$ in the interior of $\Gamma$. Let $h\neq0$ such that $z+h$ is 
in the interior of $\Gamma$. Then we find
\begin{displaymath}
  \frac{f(u(z+h))-f(u(z))}{h}-\frac{d(f\circ u)(z)}{dz}
  =\frac{h}{2\pi i}\oint_{\Gamma}\frac{f(u(\zeta))}{(\zeta-z-h)(\zeta-z)^2}d\zeta.
\end{displaymath}
But the function $f(u(z))$ is bounded on $\Gamma$ since it is continuous. 
Thus there exists a constant $C>0$ such that
\begin{displaymath}
  \Big|\frac{f(u(z+h))-f(u(z))}{h}-\frac{d(f\circ u)(z)}{dz}\Big|\leq C|h|\,\|f\|_{X^*}.
\end{displaymath}
One can show that $C$ can be chosen independently of $f$ 
and that this estimate implies the convergence of $\frac{u(z+h)-u(z)}{h}$ as $h\to0$.
\end{proof}

\begin{definition}
A family of operators $D(z)\in B(X,Y)$ which depend on $z\in\Omega$ 
is called holomorphic on $\Omega$, if for each $z_0\in\Omega$ there exists 
a bounded linear operator $D'\in B(X,Y)$ such that
\begin{displaymath}
  \lim_{h\to0}\Big\|\frac{D(z_0+h)-D(z_0)}{h}-D'\Big\|_{B(X,Y)}=0.
\end{displaymath}
\end{definition}

The next theorem (\cite{k}, p. 152) gives a criterion for holomorphy of 
bounded linear operators.

\begin{theorem}
\label{hol_operator_theorem}
Let $D(z)\in B(X,Y)$ be a family of bounded linear operators which depend 
on $z\in\Omega$. If for all $u\in X$ and for all $g\in Y^*$ the function 
$g(D(z)u)$ is holomorphic on $\Omega$, then $D(z)$ is holomorphic on $\Omega$.
\end{theorem}

\begin{proof}
The proof is analogous to the proof of Theorem \ref{hol_vector_theorem}.
\end{proof}

We generalize the notion of a holomorphic family of operators to families 
of closed linear operators which are not necessarily bounded (see \cite{k}, p. 366).

\begin{definition}
A family of operators $D(z)\in C(X,Y)$ which depend on $z\in\Omega$ is called 
holomorphic at $z_0\in\Omega$, if there exists a third complex Banach space $Z$ 
and two families $E(z)\in B(Z,X)$, $F(z)\in B(Z,Y)$, which are holomorphic at $z_0$ 
such that for all $z\in\Omega$ the operator $E(z)$ maps $Z$ onto the domain of $D(z)$ 
one to one and such that for all $z\in\Omega$ we have $D(z)E(z)=F(z)$. 
The family $D(z)$ is called holomorphic on $\Omega$, if it is holomorphic 
at every $z_0\in\Omega$.
\end{definition}

If $X$, $Y$ are Hilbert spaces and $D\in C(X,Y)$, we denote by $D^*$ 
the adjoint operator of $D$. The following definition is from \cite{k}, p. 385.

\begin{definition}
Let $H$ be a Hilbert space and let $\Omega$ be an open and connected subset of $\C$ which 
is symmetric with respect to the real axis. A holomorphic family of operators 
$D(z)\in C(H)$ is called a self-adjoint holomorphic family, if for all $z\in\Omega$ 
the operator $D(z)$ is densely defined and if for all $z\in\Omega$ 
we have $D(z)^*=D(\overline{z})$.
\end{definition}

If $H$ is a Hilbert space, $D(z)\in C(H)$ a self-adjoint holomorphic family 
on $\Omega$ and if $\lambda$ is an eigenvalue of $D(z_0)$ of finite multiplicity $r$ 
for some $z_0\in\R$, 
then one obtains families of eigenvalues $\lambda_i(z)$ of $D(z)$, 
$1\leq i\leq r$, which are real-analytic in $z$ on an open interval 
$I\subset\Omega\cap\R$ around $z_0$ and such that 
$\lambda_i(z_0)=\lambda$ for all $i$ (\cite{k}, p. 386). In general $I$ depends 
on $\lambda$. 
However for the so called holomorphic families 
of type (A), one obtains an open interval around $z_0$, such that all the eigenvalues 
and eigenvectors are real-analytic on this interval.

\begin{definition}
A family of operators $D(z)\in C(X,Y)$ which depend on $z\in\Omega$ is called 
holomorphic of type (A), if the domain $\dom(D(z))$ is independent of $z$, 
and if for all $z\in\Omega$ and for all $u\in\dom(D(z))$ the family of vectors 
$D(z)u\in Y$ is holomorphic on $\Omega$.
\end{definition}

\begin{remark}
If the family $D(z)\in C(X,Y)$ is holomorphic of type (A) on $\Omega$, then 
it is holomorphic on $\Omega$ in the sense of the definition above.
\end{remark}

\begin{proof}[Proof (see \cite{k}, p. 375)]
Choose $z_0\in\Omega$ and set $D:=D(z_0)$. 
The space $Z:=\dom(D)$ with the norm 
$\|u\|_Z:=\|u\|_X+\|Du\|_Y$ is then a Banach space. Define $E$: $Z\to X$ by $E(u):=u$. 
Since $\|u\|_X\leq\|u\|_Z$ we find that $E\in B(Z,X)$. For all $z\in\Omega$ define 
$F(z)$: $Z\to Y$ by $F(z)u:=D(z)u$. Using that $D(z)$ is closed and that 
$\|u\|_X\leq\|u\|_Z$ we find that $F(z)\in C(Z,Y)$. Furthermore since $\dom(F(z))=Z$ 
the closed graph theorem (\cite{k}, p. 166) implies that $F(z)\in B(Z,Y)$. 
By hypothesis $F(z)u=D(z)u$ is holomorphic for all $u\in Z$. By Theorem 
\ref{hol_operator_theorem} it follows that $F(z)\in B(Z,Y)$ is holomorphic. 
Since $E$: $Z\to\dom(D)$ is one to one and $D(z)E=F(z)$ for all $z\in\Omega$ 
the assertion follows.
\end{proof}

We have the following result due to Rellich.

\begin{theorem}
\label{rellich_theorem}
Let $D(z)\in C(H)$ be a self-adjoint holomorphic family of type (A) defined on an 
open neighborhood $\Omega\subset\C$ of an open interval $I\subset\R$. 
Furthermore let $D(z)$ have compact resolvent for some $z\in\Omega$. Then there 
is a sequence of scalar-valued functions $(\lambda_n)_{n\in\N}$ 
and a sequence of vector-valued functions $(u_n)_{n\in\N}$ all real-analytic 
on $I$, such that for $t\in I$ the $\lambda_n(t)$ represent all the repeated 
eigenvalues of $D(t)$ and the $u_n(t)$ form a complete orthonormal family of 
the associated eigenvectors of $D(t)$.
\end{theorem}

\begin{proof}
see \cite{k}, Thm. VII, 3.9, p. 392.
\end{proof}

The following Lemma is a criterion for a family of operators to be holomorphic 
of type (A).

\begin{lemma}
Let $D$: $X\to Y$ be a closable operator with domain $\dom(D)$ 
and let $D_j$, $j\in\N\setminus\{0\}$, be operators $X\to Y$ with 
domain containing $\dom(D)$. If there exist real constants $a$, $b$, $c\geq0$ such 
that $\|D_j u\|\leq c^{j-1}(a\|u\|+b\|Du\|)$ for all $u\in\dom(D)$ and all $j\in\N\setminus\{0\}$, 
then the series $D(z)=D+\sum_{j=1}^{\infty}D_jz^j$ for $z\in\C$, $|z|<(b+c)^{-1}$ 
defines a family of closable operators with domain $\dom(D)$. 
Furthermore the closures of these operators form a holomorphic family of type (A).
\end{lemma}

\begin{proof}
see \cite{k}, Thm. VII, 2.6, p. 377.
\end{proof}

Now let $M$ be a compact Riemannian spin manifold, let $I\subset\R$ 
be an open interval and let $(g_t)_{t\in I}$ be a family of smooth 
Riemannian metrics on $M$. 

\begin{definition}
The family $(g_t)_{t\in I}$ is called real-analytic, 
if for every $t_0\in I$ and for every $j\in\N$ there exists a section 
$h_j\in\sym{\infty}$, 
such that for all charts $(U,\ph)$ of $M$ and for all compact subsets $L\subset U$ 
there exists $\ep_L>0$, such that for all 
$t$ with $|t-t_0|<\ep_L$, for all coordinate vector fields 
$\partial_a$, $\partial_b$, $1\leq a,b\leq n$, and for all $r\in\N$ 
we have 
\begin{displaymath}
  \big\|g_t(\partial_a,\partial_b)-\sum_{j=0}^N (t-t_0)^j h_j(\partial_a,\partial_b)\big\|_{C^r(L)}
  \to0\textrm{ as }N\to\infty.
\end{displaymath}
\end{definition}

We apply the criterion above to show that for a real-analytic family 
$(g_t)_{t\in I}$ of Riemannian metrics on a compact spin manifold $M$ the family 
of operators $D^{g,g_t}$ constructed in Section 
\ref{identification_section} can be extended to a holomorphic family of type (A). 
Here we consider the Dirac operator as a closed operator on the Sobolev space 
$\Hspinor{M}{g}{m}$ with domain $\dom(D^g)=\Hspinor{M}{g}{m+1}$ for some $m\in\N$. 

\begin{lemma}
\label{analyticity_lemma}
Let $(M,g,\Theta)$ be a compact Riemannian spin manifold 
and let $D^g\in C(\Hspinor{M}{g}{m})$ be the Dirac operator. 
Let $(g_t)_{t\in I}$ be a real-analytic family of Riemannian 
metrics on $M$ with $g_0=g$. Then there exists an open neighborhood 
$\Omega\subset\C$ of $0$ and a self-adjoint holomorphic family of type (A) 
of operators $D(z)\in C(\Hspinor{M}{g}{m})$ on $\Omega$, 
such that for all $t\in\Omega\cap I$ we have $D(t)=D^{g,g_t}$.
\end{lemma}

\begin{proof}
Since $(g_t)_{t\in I}$ is a real-analytic family of Riemannian metrics, there exist 
$h_j\in\sym{\infty}$, $j\in\N$, 
such that for all charts $(U,\varphi)$ of $M$ and for all 
compact subsets $L\subset U$ there exists $\ep_L>0$ such that 
for all $t$ with $|t|<\ep_L$ and for all coordinate vector fields 
$\partial_a$, $\partial_b$, $1\leq a,b\leq n$ and for all $r\in\N$ we have
\begin{displaymath}
  \big\|g_t(\partial_a,\partial_b) - \sum_{j=0}^N t^j h_j(\partial_a,\partial_b) \big\|_{C^r(L)}\to0,\quad N\to\infty.
\end{displaymath}
As in Section \ref{identification_section} we define 
endomorphisms $a_{g,g_t}$ and $a_{g,h_j}$, $j\in\N$, of $TM$ 
such that for all $v$, $w$ in $TM$ we have
\begin{displaymath}
  g(a_{g,g_t}v,w)=g_t(v,w), \quad
  g(a_{g,h_j}v,w)=h_j(v,w).
\end{displaymath}
Let $(U,\varphi)$ be a chart of $M$ and let $L\subset U$ be compact. 
For $v\in TM$ we abbreviate $|v|:=g(v,v)^{1/2}$. 
Let $(e_i)_{i=1}^n$ be a local $g$-orthonormal frame on $L$. 
Since $(g_t)_{t\in I}$ is real-analytic, there exists 
$\ep_L>0$ such that for all $t$ with $|t|<\ep_L$ we have
\begin{eqnarray*}
  &&\sup_{v\in TM|_L,\,|v|\leq1} \big| a_{g,g_t}v - \sum_{j=0}^N t^j a_{g,h_j} v \big|\\
  &=&\sup_{v\in TM|_L,\,|v|\leq1}
  \big| \sum_{i=1}^n g(a_{g,g_t}v,e_i) e_i 
  - \sum_{i=1}^n \sum_{j=0}^N t^j g(a_{g,h_j}v,e_i) e_i\big|\\
  &=&\sup_{v\in TM|_L,\,|v|\leq1}
  \big| \sum_{i=1}^n g_t(v,e_i) e_i 
  - \sum_{i=1}^n \sum_{j=0}^N t^j h_j(v,e_i) e_i \big|\\
  &\leq&\sum_{i,j=1}^n\big\| g_t(e_j,e_i) - \sum_{m=0}^N t^m h_m(e_j,e_i)\big\|_{C^0(L)}\to0,
  \quad N\to\infty.
\end{eqnarray*}
Therefore we find $\delta_L>0$ such that for all 
$t$ with $|t|<\delta_L$ and for all $v\in TM|_L$ 
the vector $b_{g,g_t}v$ is given by a power series, 
which converges uniformly in $v$, $|v|\leq1$.
Since $M$ is compact there exists $\delta>0$ and for all $j\in\N$ 
there exists $b_j\in\End(TM)$ such that 
for all $t\in I$ with $|t|<\delta$ we have
\begin{displaymath}
  \sup_{v\in TM,|v|\leq1} \big|b_{g,g_t}v-\sum_{j=0}^N t^j b_j v\big|\to0,\quad N\to\infty.
\end{displaymath}
By similar arguments one finds that also for any fixed $t$ the covariant derivatives 
of $b_{g,g_t}$ of any order are given by convergent power series in this sense. 
We insert the power series for $b_{g,g_t}$, $g_t$ and $f_{g,g_t}$ 
into (\ref{Dgh}) and define the operator $D_j$ 
as the coefficient of $t^j$ for $j\in\N$. Clearly $D_j\in C(\Hspinor{M}{g}{m})$ 
and for $\psi\in\dom(D_k)=\Hspinor{M}{g}{m+1}$ the spinor $D_j\psi$ locally has the form
\begin{displaymath}
  D_j\psi(x)=\sum_{i=1}^n e_i\cdot\nabla^g_{b_j(e_i)}\psi(x)+c_j(x)\cdot\psi(x),
\end{displaymath}
where $c_j$ are functions with values in the Clifford algebra. 
Since the series $\sum_{j=0}^{\infty}b_jt^j$ and $\sum_{j=0}^{\infty}c_jt^j$ 
converge uniformly in $x\in M$ for $|t|<\delta$, there exists $C>0$ such that 
for all $j\geq1$ we have
\begin{displaymath}
  \sup\{|b_jv|\,\,|v\in TM,|v|\leq1\}^{1/j}\leq\frac{C}{\delta},\quad
  \sup\{|c_j(x)|\,\,|x\in M\}^{1/j}\leq\frac{C}{\delta}.
\end{displaymath}
Again similar estimates hold for the covariant derivatives of $b_j$ and $c_j$. 
We find that for all $j\geq1$ and for all $\psi\in\dom(D_j)$
\begin{displaymath}
  \|D_j\psi\|_{H^m}\leq \frac{C^j}{\delta^j} (n\|D^g\psi\|_{H^m}+\|\psi\|_{H^m}).
\end{displaymath}
By the criterion above the family $D(z):=\sum_{j=0}^{\infty}z^jD_j$ defines 
a holomorphic family of type (A) on $\Omega:=\{z\in\C|\,|z|<\delta\}$ such that 
for all $t\in\Omega\cap I$ we have $D(t)=D^{g,g_t}$. 
Since $D^{g,g_t}$ is self-adjoint by the construction in Section \ref{identification_section}, 
all the operators $D_j$ are self-adjoint. Therefore the family $D(z)$ is self-adjoint.
\end{proof}

\end{appendix}

\newpage
\printindex

\end{document}